\documentclass[11pt]{amsart}
\usepackage{amsmath,amsthm,amsfonts,amscd,amssymb,eucal,latexsym,mathrsfs}
\usepackage[all]{xy}
\newtheorem{theorem}{Theorem}[section]

\newtheorem{lemma}[theorem]{Lemma}
\newtheorem{proposition}[theorem]{Proposition}

\theoremstyle{definition}
\newtheorem{definition}[theorem]{Definition}

\theoremstyle{plain}

\theoremstyle{definition}

\theoremstyle{remark}

\newcommand{\topol}{{\text{\rm top}}}

\newcommand{\cA}{{\mathcal A}}
\newcommand{\cB}{{\mathcal B}}

\newcommand{\cM}{{\mathcal M}}

\newcommand{\cU}{{\mathcal U}}

\newcommand{\cQ}{{\mathcal Q}}

\newcommand{\cR}{{\mathcal R}}
\newcommand{\cS}{{\mathcal S}}

\newcommand{\Cb}{{\mathbb C}}
\newcommand{\Zb}{{\mathbb Z}}
\newcommand{\Eb}{{\mathbb E}}

\newcommand{\Rb}{{\mathbb R}}
\newcommand{\Nb}{{\mathbb N}}

\newcommand{\diam}{{\rm diam}}

\newcommand{\sL}{{\mathscr L}}

\newcommand{\Sym}{{\rm Sym}}

\newcommand{\Hom}{{\rm Hom}}

\newcommand{\ch}{{\bar{h}}}
\newcommand{\erg}{{\rm e}}
\newcommand{\Map}{{\rm Map}}

\allowdisplaybreaks

\begin{document}

\title[Soficity, amenability, and dynamical entropy]{Soficity, amenability, and dynamical entropy}

\author{David Kerr}
\author{Hanfeng Li}
\address{\hskip-\parindent
David Kerr, Department of Mathematics, Texas A{\&}M University,
College Station TX 77843-3368, U.S.A.}
\email{kerr@math.tamu.edu}

\address{\hskip-\parindent
Hanfeng Li, Department of Mathematics, SUNY at Buffalo,
Buffalo NY 14260-2900, U.S.A.}
\email{hfli@math.buffalo.edu}

\date{November 18, 2010}

\begin{abstract}
In a previous paper the authors developed an operator-algebraic approach to Lewis Bowen's
sofic measure entropy that yields invariants for actions of countable sofic groups
by homeomorphisms on a compact metrizable space and by measure-preserving transformations
on a standard probability space. We show here that these measure and topological entropy invariants
both coincide with their classical counterparts when the acting group is amenable.
\end{abstract}

\maketitle

\section{Introduction}

In \cite{Bowen10} Lewis Bowen introduced a notion of entropy for measure-preserving actions of
a countable discrete sofic group on a standard probability space admitting a generating finite partition.
By a limiting process the definition also applies more generally whenever there exists a generating
partition with finite entropy. The idea is to dynamically model a given finite partition by partitions
of a finite set on which the group acts in an approximate way according to the definition of soficity.
Given a fixed sequence of sofic approximations, the entropy is locally defined as the exponential
growth rate of the number of model partitions relative to the size of the finite sets
on which the sofic approximations operate. Taking an infimum over the parameters which control
the localization then defines the entropy of the original partition.
This quantity is then shown to take a common value over all generating
finite partitions. It may depend though on the choice of sofic approximation sequence, yielding
in general a collection of entropy invariants for the system. However,
in the case that the acting group is amenable and there exists a generating finite partition,
Bowen showed in \cite{Bowen10a} that sofic measure entropy coincides with classical Kolmogorov-Sinai entropy
for all choices of sofic approximation sequence.

Applying an operator algebra perspective, the present authors developed in \cite{KerLi10} an alternative
approach to sofic entropy that is more akin to Rufus Bowen's definition of topological entropy for
$\Zb$-actions in terms of $\varepsilon$-separated partial orbits. This approach furnishes both
measure and topological dynamical invariants for general actions, and these entropies are related by a
variational principle as in the classical case \cite[Sect.\ 6]{KerLi10}.
For measure-preserving actions admitting a generating
partition with finite entropy, our measure entropy coincides with Lewis Bowen's \cite[Sect.\ 3]{KerLi10}.

The goal of this paper is to prove that, when the acting group is amenable, the sofic measure and
topological entropies from \cite{KerLi10} both coincide with their classical counterparts,
independently of the sofic approximation sequence. In the measurable
case this generalizes Bowen's result from \cite{Bowen10a} by means of a complete different type of argument.
%that involves the statistical analysis of embedded copies of a given sofic approximation, which can
%be viewed as systems of interlocking partial orbits.
Once we have the result for measure entropy
the topological version ensues by combining the variational principle from \cite{KerLi10}
with the classical variational principle. We will also give a direct proof in the topological case
as it illustrates some of the basic ideas without the additional probabilistic complications that arise
in the treatment of measure-preserving actions.

In \cite{KerLi10} we took the operator algebra approach to defining sofic entropy because
it was crucial for showing that one actually obtains a conjugacy invariant in the measurable case.
However, for many purposes, including that of this paper, it is simpler to express
both topological and measure entropy in terms of the dynamics on the space itself,
as in Rufus Bowen's definition. In the measurable case this requires some topological
structure, namely a compact metrizable space on which the group acts continuously
with an invariant Borel probability measure. Since such topological models always exist,
there is no loss in generality in taking this viewpoint, which we will do in this paper.
We will therefore begin in Sections~\ref{S-topological defn} and \ref{S-measure defn} by
formulating the spatial definitions of sofic topological and measure entropy and establishing
their equivalence with the original linear definitions from \cite{KerLi10}.

The basis for our analysis of the amenable case is a
sofic approximation version of the Rokhlin lemma of Ornstein and Weiss,
which can be extracted from Ornstein and Weiss's proof \cite{OrnWei87}. This appeared
in Section~4 of \cite{Ele06} in a form that treats more generally the case of finite graphs. In our sofic approximation
situation we will need a stronger statement that allows us to prescribe the quasitiling coverage of the finite approximation
space and the set from which the tiling centres come. In Section~\ref{S-Rokhlin}
we will give a self-contained proof of this Rokhlin lemma for sofic approximations of countable discrete amenable
groups following the line of argument in \cite{OrnWei87}.
%The Rokhlin lemma demonstrates in particular that sofic approximation sequences for
%countable discrete amenable groups possess an asymptotic self-partitioning property, which in Section~\ref{S-ap} is shown to
%imply a certain approximation regularity in the computation of sofic entropy. This regularity is crucial in
In Section~\ref{S-topological} we prove that the sofic and classical
topological entropies coincide for continuous actions of a countable discrete amenable group on a compact metrizable space.
%We observe as a corollary that classical topological entropy does not decrease when restricting the action to a subgroup,
%which for $\Zb^d$-actions is well known.
Finally, in Section~\ref{S-measure} we show that the sofic measure entropy from \cite{KerLi10} and the
classical Kolmogorov-Sinai entropy coincide for measure-preserving actions of a countable discrete amenable group
on a standard probability space.

We round out the introduction with some terminology concerning amenable and sofic groups
and spanning and separated sets.
For general information on unital commutative $C^*$-algebras as used in this paper and any unexplained
notation and terminology see the introduction to \cite{KerLi10}. The classical definitions of measure
and topological entropy for actions of countable discrete amenable groups
will be recalled in Sections~\ref{S-topological} and \ref{S-measure}, respectively.

For $d\in\Nb$ we write $\Sym (d)$ for the group of permutations of $\{ 1,\dots ,d \}$.
Let $G$ be a countable discrete group. The identity element of such a $G$ will always be denoted by $e$.
The group $G$ is said to be {\it amenable} if it admits a left invariant mean, i.e., a state on $\ell^\infty (G)$ which is invariant
under left translation by $G$. This is equivalent to the existence of a F{\o}lner sequence, which is a
sequence $\{ F_i \}_{i=1}^\infty$ of nonempty finite subsets of $G$ such that
$|F_i |^{-1} |sF_i \Delta F_i | \to 0$ as $i\to\infty$ for all $s\in G$.
We say that $G$ is {\it sofic} if for $i\in\Nb$ there
are a sequence $\{ d_i \}_{i=1}^\infty$ of positive integers and a sequence
$\{ \sigma_i \}_{i=1}^\infty$ of maps $s\mapsto\sigma_{i,s}$ from $G$ to $\Sym (d_i )$
which is asymptotically multiplicative and free in the sense that
\[
\lim_{i\to\infty}
\frac{1}{d_i} \big| \{ a\in \{ 1,\dots ,d_i \} : \sigma_{i,st} (a) = \sigma_{i,s} \sigma_{i,t} (a) \} \big| = 1
\]
for all $s,t\in G$ and
\[
\lim_{i\to\infty}
\frac{1}{d_i} \big| \{ a\in \{ 1,\dots ,d_i \} : \sigma_{i,s} (a) \neq \sigma_{i,t} (a) \} \big| = 1
\]
for all distinct $s,t\in G$. Such a sequence $\{ \sigma_i \}_{i=1}^\infty$ for which
$\lim_{i\to\infty} d_i = \infty$ is referred to as a {\it sofic approximation sequence} for $G$.
The condition $\lim_{i\to\infty} d_i = \infty$ is assumed in order to avoid pathologies in the theory of
sofic entropy (e.g., it is essential for the variational principle in \cite{KerLi10}) and is automatic
if $G$ is infinite. Note that if $G$ is amenable then it is sofic, as one can easily construct a sofic approximation
sequence from a F{\o}lner sequence.
%Throughout the paper $\Sigma = \{ \sigma_i : G\to\Sym (d_i ) \}_{i=1}^\infty$ will denote a fixed but
%arbitrary sofic approximation sequence.

For a map $\sigma :G\to\Sym (d)$ for some $d\in\Nb$ we will
denote $\sigma_s (a)$ for $s\in G$ and $a\in \{ 1,\dots ,d\}$ simply by $sa$ when convenient, and also use
$\sigma$ to denote the induced map from $G$ into the automorphism group of the $C^*$-algebra
$C(\{ 1,\dots ,d \} ) \cong\Cb^d$ given by $\sigma_s (f)(a) = f(s^{-1} a)$
for all $s\in G$, $f\in \Cb^d$, and $a\in \{1,\dots ,d\}$. For a $d\in\Nb$ we will invariably use $\zeta$ to denote the
uniform probability measure on $\{ 1,\dots ,d \}$,
which will be regarded as a state (i.e., a unital positive linear functional) on the $C^*$-algebra
$\Cb^d \cong C(\{ 1,\dots ,d \} )$ whenever appropriate.

Let $(Y,\rho )$ be a pseudometric space and $\varepsilon\geq 0$. A set $A\subseteq Y$ is said
to be {\it $(\rho ,\varepsilon )$-separated} or {\it $\varepsilon$-separated with respect to $\rho$} if
$\rho (x,y) \geq \varepsilon$ for all distinct $x,y\in A$, and {\it $(\rho ,\varepsilon )$-spanning} or
{\it $\varepsilon$-spanning with respect to $\rho$} if for every $y\in Y$ there is an $x\in A$
such that $\rho (x,y) < \varepsilon$.
We write $N_\varepsilon (Y, \rho )$ for the maximal cardinality of a finite $(\rho ,\varepsilon )$-separated subset of $Y$.
If $G$ is a group acting on $Y$ and $F$ is a nonempty finite subset of $G$ then
we define the pseudometric $\rho_F$ on $Y$ by $\rho_F (x,y) = \max_{s\in F} \rho (sx,sy)$.
%a set $A\subseteq Y$ is said
%to be {\it $(\rho ,F,\varepsilon )$-separated} or {\it $(F,\varepsilon )$-separated with respect to $\rho$} if
%$\max_{s\in F} \rho (sx,sy) \geq \varepsilon$ for all distinct $x,y\in Y$, and {\it $(\rho ,F,\varepsilon )$-spanning}
%or {\it $(F,\varepsilon )$-spanning with respect to $\rho$}
%if for every $y\in Y$ there is an $x\in X$ such that $\max_{s\in F} \rho (sx,sy) < \varepsilon$.
\medskip

\noindent{\it Acknowledgements.}
The first author was partially supported by NSF
grant DMS-0900938. He thanks Lewis Bowen for several helpful discussions. Part of this work was carried
out during a visit of the first author to SUNY at Buffalo in February 2010 and he thanks the
analysis group there for its hospitality. The second author was partially supported by NSF grants DMS-0701414
and DMS-1001625.
We are grateful to the referee for helpful comments.

\section{Topological entropy}\label{S-topological defn}

Let $G$ be a countable sofic group,
$X$ a compact metrizable space, and $\alpha$ a continuous action of $G$
on $X$. The action of $G$ on points will usually be
expressed by the concatenation $(s,x)\mapsto sx$, and $\alpha$ will be also be used for the induced action of $G$ on
$C(X)$ by automorphisms, so that for $f\in C(X)$ and $s\in G$ the function $\alpha_s (f)$
is given by $x \mapsto f(s^{-1} x)$. A subset of $C(X)$ is said to be {\it dynamically generating}
if it is not contained in any proper $G$-invariant unital $C^*$-subalgebra of $C(X)$.

First we recall the definition of sofic topological entropy from \cite{KerLi10} and then show how it
can be reformulated using approximately equivariant maps from the sofic approximation space into $X$.
Throughout this section $\Sigma = \{ \sigma_i : G \to \Sym (d_i ) \}_{i=1}^\infty$ is a fixed sofic approximation sequence for $G$.
Let $\cS=\{p_n\}_{n\in \Nb}$ be a sequence in the unit ball of $C_{\Rb} (X)$.
For a given $d\in\Nb$ we define on the set of unital homomorphisms from $C(X)$ to $\Cb^d$ the pseudometrics
%positive linear maps
%from some unital self-adjoint linear subspace of $C(X)$ containing $\spn (\cS )$ to $\Cb^d$ the pseudometric
\begin{align*}
\rho_{\cS, 2} (\varphi , \psi ) &= \sum_{n=1}^{\infty} \frac{1}{2^n}\| \varphi (p_n) - \psi (p_n) \|_2, \\
\rho_{\cS, \infty} (\varphi , \psi ) &= \sum_{n=1}^{\infty} \frac{1}{2^n}\| \varphi (p_n) - \psi (p_n) \|_\infty,
\end{align*}
where the norm $\|\cdot \|_2$ refers to the uniform probability measure $\zeta$ on $\{1, \dots, d\}$.
For a nonempty finite set $F\subseteq G$,
a $\delta > 0$, and a map $\sigma : G\to\Sym (d)$ for some $d\in\Nb$
we define $\Hom (\cS , F,\delta ,\sigma )$ to be
the set of all unital homomorphisms $\varphi : C(X) \to \Cb^d$ such that
\[ \rho_{\cS, 2}( \varphi\circ\alpha_s,  \sigma_s \circ\varphi)< \delta
%\sum_{n=1}^{\infty} \frac{1}{2^n}\| \varphi\circ\alpha_s (p_n) - \sigma_s \circ\varphi (p_n) \|_2 < \delta
\]
for all $s\in F$.
For an $\varepsilon > 0$ we then set
\begin{align*}
h_\Sigma^\varepsilon (\cS ,F,\delta )
&= \limsup_{i\to\infty} \frac{1}{d_i} \log N_\varepsilon (\Hom (\cS ,F,\delta ,\sigma_i ), \rho_{\cS, 2} ) ,\\
h_\Sigma^\varepsilon (\cS ,F) &= \inf_{\delta > 0} h_\Sigma^\varepsilon (\cS ,F,\delta ) ,\\
h_\Sigma^\varepsilon (\cS ) &= \inf_{F} h_\Sigma^\varepsilon (\cS ,F) ,\\
h_\Sigma (\cS ) &= \sup_{\varepsilon > 0} h_\Sigma^\varepsilon (\cS )
\end{align*}
where $F$ in the third line ranges over all nonempty finite subsets of $G$.
If $\Hom (\cS ,F,\delta ,\sigma_i )$ is empty for all sufficiently large $i$, we set
$h_\Sigma^\varepsilon (\cS ,F,\delta ) = -\infty$.
By Theorem~4.5 of \cite{KerLi10} the quantity $h_\Sigma (\cS )$ is the same for all dynamically
generating $\cS$, and we define the topological entropy $h_\Sigma (X,G)$ of the system to be this value.

The following lemma is a version of Lemma~4.8 in \cite{KerLi10}, saying that $h_\Sigma (\cS)$ can also be computed by substituting $\rho_{\cS ,\infty}$ for $\rho_{\cS ,2}$,
and can be established by a similar argument.

\begin{lemma}\label{L-infinity norm top}
%Let $G$ be a sofic countable discrete group acting continuously on a compact metrizable space $X$.
%Let $\Sigma = \{ \sigma_i : G\to\Sym (d_i ) \}_{i=1}^\infty$ be a sofic approximation sequence for $G$.
Let $\cS$ be a sequence in the unit ball of $C_\Rb (X)$. Then
%for every $\varepsilon > 0$ and $\theta>0$ there is
%an $\varepsilon' > 0$ such that
%\[
%\limsup_{i\to\infty} \frac{1}{d_i} \log N_{\varepsilon} (\Hom (\cS ,F,\delta ,\sigma_i ),\rho_{\cS , \infty} )
%\leq h_\Sigma^{\varepsilon'} (\cS ,F,\delta ) + \theta
%\]
%for all nonempty finite sets $F\subseteq G$ and $\delta > 0$.
\[
h_\Sigma (\cS) = \sup_{\varepsilon > 0} \inf_F \inf_{\delta > 0}
\limsup_{i\to\infty} \frac{1}{d_i} \log N_{\varepsilon} (\Hom (\cS ,F,\delta ,\sigma_i ),\rho_{\cS , \infty} )
\]
where $F$ ranges over the nonempty finite subsets of $G$.
\end{lemma}

%It follows that $h_\Sigma (\cS)$ can also be computed by substituting $\rho_{\cS ,\infty}$ for $\rho_{\cS ,2}$, i.e.,
%\[
%h_\Sigma (\cS) = \sup_{\varepsilon > 0} \inf_F \inf_{\delta > 0}
%\limsup_{i\to\infty} \frac{1}{d_i} \log N_{\varepsilon} (\Hom (\cS ,F,\delta ,\sigma_i ),\rho_{\cS , \infty} )
%\]
%where $F$ ranges over the nonempty finite subsets of $G$.

Now let $\rho$ be a continuous pseudometric on $X$, which will play the role of $\cS$ in our spatial definition.
For a given $d\in\Nb$, we define on the set of all maps from $\{ 1,\dots ,d\}$ to $X$ the pseudometrics
\begin{align*}
\rho_2 (\varphi , \psi ) &= \bigg( \frac{1}{d} \sum_{a=1}^d (\rho (\varphi (a),\psi (a)))^2 \bigg)^{1/2} , \\
\rho_\infty (\varphi ,\psi ) &= \max_{a=1,\dots ,d} \rho (\varphi (a),\psi (a)) .
\end{align*}

\begin{definition}\label{D-map top}
Let $F$ be a nonempty finite subset of $G$ and $\delta > 0$.
Let $\sigma$ be a map from $G$ to $\Sym (d)$ for some $d\in\Nb$.
We define $\Map (\rho ,F,\delta ,\sigma )$ to be the set of all maps $\varphi : \{ 1,\dots ,d\} \to X$ such that
$\rho_2 (\varphi\circ\sigma_s , \alpha_s \circ\varphi ) < \delta$ for all $s\in F$.
\end{definition}

\begin{definition}
Let $F$ be a nonempty finite subset of $G$ and
%$L$ a nonempty finite subset of $C(X)$,
$\delta > 0$. For $\varepsilon > 0$ we define
\begin{align*}
h_{\Sigma ,2}^\varepsilon (\rho ,F, \delta ) &=
\limsup_{i\to\infty} \frac{1}{d_i} \log N_\varepsilon (\Map (\rho ,F,\delta ,\sigma_i ),\rho_2 ) ,\\
h_{\Sigma ,2}^\varepsilon (\rho ,F) &= \inf_{\delta > 0} h_{\Sigma ,2}^\varepsilon (\rho ,F,\delta ) ,\\
h_{\Sigma ,2}^\varepsilon (\rho ) &= \inf_{F} h_{\Sigma ,2}^\varepsilon (\rho ,F) ,\\
h_{\Sigma ,2} (\rho ) &= \sup_{\varepsilon > 0} h_{\Sigma ,2}^\varepsilon (\rho ) ,
\end{align*}
where $F$ in the third line ranges over the nonempty finite subsets of $G$. If
$\Map (\rho ,F,\delta ,\sigma_i )$ is empty for all sufficiently large $i$, we set
$h_{\Sigma ,2}^\varepsilon (\rho ,F, \delta ) = -\infty$.
We similarly define $h_{\Sigma ,\infty}^\varepsilon (\rho ,F, \delta )$, $h_{\Sigma ,\infty}^\varepsilon (\rho ,F)$,
$h_{\Sigma ,\infty}^\varepsilon (\rho )$ and $h_{\Sigma ,\infty} (\rho )$
using $N_\varepsilon(\cdot, \rho_\infty)$ in place of $N_\varepsilon(\cdot, \rho_2)$.
\end{definition}

%It will be convenient for some purposes to be able to replace $\rho_2$ with $\rho_\infty$ in the
%computation of topological entropy, which the following proposition allows us to do.
%
%\begin{proposition}\label{P-infty top}
%We have
%\begin{align*}
%h_\Sigma (\rho ) = \sup_\varepsilon \inf_F \inf_{\delta > 0} \limsup_{i\to\infty}
%\frac{1}{d_i} \log N_\varepsilon (\Map (\rho ,F,\delta ,\sigma_i ),\rho_\infty )
%\end{align*}
%where $F$ ranges over the nonempty finite subsets of $G$.
%\end{proposition}

We say that $\rho$ is {\it dynamically generating} \cite[Sect.\ 4]{Li10} if
for any distinct points $x,y\in X$ one has $\rho(sx, sy)>0$ for some $s\in G$.

\begin{proposition} \label{P-topological entropy}
Suppose that $\rho$ is dynamically generating. Then
\[
h_\Sigma (X,G) = h_{\Sigma ,2} (\rho )=h_{\Sigma ,\infty} (\rho) .
\]
\end{proposition}

\begin{proof}
We will show that $h_\Sigma (X,G) = h_{\Sigma ,2} (\rho )$. The proof for $h_\Sigma (X,G)=h_{\Sigma ,\infty} (\rho)$ is similar,
in view of
%the comment following
Lemma~\ref{L-infinity norm top}.

We say that two continuous pseudometrics $\rho$ and $\rho'$ on $X$ are equivalent if for any $\varepsilon>0$
there is an $\varepsilon'>0$ such
that, for any points $x,y\in X$, if $\rho'(x, y)<\varepsilon'$ then $\rho(x, y)<\varepsilon$, and vice versa.
If $\rho$ and $\rho'$ are equivalent, then the pseudometrics $\rho_2$ and $\rho'_2$ on the set of all maps
$\{1, \dots, d\}\rightarrow X$ are uniformly equivalent in the sense that
for any $\delta>0$ there is some $\delta'>0$ such that, for any $d\in \Nb$ and any maps $\Phi$ and $\Psi$
from $\{1, \dots, d\}$ to $X$, if $\rho'_2(\Phi, \Psi)<\delta'$ then $\rho_2(\Phi, \Psi)<\delta$, and vice versa.
From this one concludes easily that $h_{\Sigma ,2} (\rho )=h_{\Sigma ,2} (\rho')$.

Let $Y$ be the quotient space of $X$ modulo $\rho$. That is, $Y$ is a quotient of $X$ such that, for any points
$x,y\in X$, $x$ and $y$ have the same image in $Y$ if and only if $\rho(x, y)=0$.
Then $\rho$ induces a compatible metric on $Y$. Let $\cS=\{p_n\}_{n\in \Nb}$ be a sequence in the unit ball of
$C_{\Rb}(Y)$ generating $C(Y)$ as a unital $C^*$-algebra.
% starting with the constant function $1$.
Then we have a compatible metric $\rho'$ on $Y$ defined by
\[
\rho'(x, y)= \sum_{n=1}^{\infty}\frac{1}{2^n}|p_n(x)-p_n(y)|.
\]
Via the quotient map $X\rightarrow Y$, we may think of $\cS$ as a sequence in $C(X)$ and $\rho'$ as a continuous pseudometric on $X$. Then both $\cS$ and $\rho'$ are dynamically generating, and, since $Y$ is compact,
$\rho$ is equivalent to $\rho'$. Now it suffices to show that $h_\Sigma (\cS)=h_{\Sigma ,2} (\rho')$.

Note that for any $d\in \Nb$ there is a natural one-to-one correspondence between the set of unital homomorphisms
$\phi: C(X)\rightarrow \Cb^d$ and the set of maps
$\Phi: \{1, \dots, d\}\rightarrow X$. For each $\Phi$, the corresponding $\phi$ sends $f\in C(X)$ to $f\circ \Phi$.
Via this correspondence, one may think
of $\rho_{\cS, 2}$ as a pseudometric on the set of all maps $\{1, \dots, d\} \rightarrow X$. It is easily checked that
$\rho_{\cS, 2}$ and $\rho'_2$ are uniformly equivalent. It follows that $h_\Sigma (\cS)=h_{\Sigma ,2} (\rho')$.
\end{proof}

%Finally, we record the following immediate consequence of the definitions.
%
%\begin{proposition}\label{P-subgroup top}
%Let $G$ be a countable sofic group acting continuously on a compact metrizable space $X$.
%Let $H$ be a subgroup of $G$. Then restriction of the action to $H$ satisfies
%$h_{\Sigma |_H} (X,H) \geq h_\Sigma (X,G)$ where $\Sigma |_H$ denotes the restriction of $\Sigma$ to $H$.
%\end{proposition}

\section{Measure entropy}\label{S-measure defn}

Let $G$ be a countable sofic group,
$(X,\mu )$ a standard probability space, and $\alpha$ an action of $G$
by measure-preserving transformations on $X$. As before $\Sigma = \{ \sigma_i : G\to\Sym (d_i ) \}$
is a fixed sofic approximation sequence. The entropy $h_{\Sigma ,\mu} (X,G)$ is defined
as in the topological case but now using approximately multiplicative linear maps from $L^\infty (X,\mu )$ to $\Cb^{d_i}$
which are approximately equivariant and approximately pull back the uniform probability measure on
$\{ 1,\dots ,d_i \}$ to $\mu$ \cite[Defn.\ 2.2]{KerLi10}.
We will not reproduce here the details of the definition, which has been formulated as such in order
to show that one obtains a measure conjugacy invariant. Instead we will recall a
more convenient equivalent definition that applies when
$\mu$ is a $G$-invariant Borel probability measure for a continuous action of $G$ on a compact metrizable
space $X$ \cite[Sect.\ 5]{KerLi10}.
This permits us to work with homomorphisms instead of maps which are merely approximately
multiplicative, which means that, as in the topological case, we can alternatively speak in terms of approximately
equivariant maps at the spatial level, as we will explain.
%As we will only need the homomorphism definition for finite partitions of unity we will formulate it in that setting
%(see Definition~5.5 of \cite{KerLi10}).

So suppose that $G$ acts continuously on a compact metrizable space $X$ with a $G$-invariant
Borel probability measure $\mu$.
Let $\cS = \{ p_n \}_{n=1}^\infty$ be a sequence in the unit ball of $C_\Rb (X)$.
Recall the pseudometrics $\rho_{\cS ,2}$ and $\rho_{\cS ,\infty}$ defined in the second paragraph of
Section~\ref{S-topological defn}.
%On the set of unital positive linear maps
%from some unital self-adjoint linear subspace of $C(X)$ containing $\spn (\cS )$ to $\Cb^d$ we define
%the pseudometrics
%\begin{align*}
%\rho_{\cS ,2} (\varphi , \psi ) &= \sum_{n=1}^\infty \frac{1}{2^n} \| \varphi (p_n ) - \psi (p_n) \|_2 ,\\
%\rho_{\cS ,\infty} (\varphi , \psi ) &= \sum_{n=1}^\infty \frac{1}{2^n} \| \varphi (p_n ) - \psi (p_n) \|_\infty .\\
%\end{align*}
Let $F$ be a nonempty finite subset of $G$
and $m\in\Nb$. We write $\cS_{F,m}$ for the set of all
products of the form $\alpha_{s_1} (f_1 ) \cdots \alpha_{s_j} (f_j )$ where $1\le j\le m$ and
$f_1 , \dots ,f_j \in \{ p_1 ,\dots ,p_m \}$ and $s_1 , \dots ,s_j \in F$.
For a map $\sigma : G\to\Sym (d)$ for some $d\in\Nb$, we
write $\Hom_\mu^X (\cS ,F,m,\delta ,\sigma )$ for the set of unital homomorphisms
$\varphi : C(X)\to\Cb^d$ such that
\begin{enumerate}
\item[(i)]
$| \zeta\circ\varphi (f) - \mu (f) | < \delta$ for all $f\in \cS_{F, m}$, and

\item[(ii)]
$\| \varphi\circ\alpha_s (f) - \sigma_s \circ\varphi (f) \|_2 < \delta$ for all $s\in F$ and $f\in \{p_1, \dots, p_m\}$.
\end{enumerate}
For $\varepsilon > 0$ we set
\begin{align*}
\ch_{\Sigma ,\mu}^\varepsilon (\cS ,F,m,\delta ) &=
\limsup_{i\to\infty} \frac{1}{d_i} \log N_\varepsilon (\Hom_\mu^X (\cS ,F,m,\delta ,\sigma_i ),\rho_{\cS ,2} ) ,\\
\ch_{\Sigma ,\mu}^\varepsilon (\cS ,F,m) &= \inf_{\delta > 0} \ch_{\Sigma ,\mu}^\varepsilon (\cS ,F,m,\delta ) ,\\
\ch_{\Sigma ,\mu}^\varepsilon (\cS ,F) &= \inf_{m\in\Nb} \ch_{\Sigma ,\mu}^\varepsilon (\cS ,F,m) ,\\
\ch_{\Sigma ,\mu}^\varepsilon (\cS ) &= \inf_{F} \ch_{\Sigma ,\mu}^\varepsilon (\cS ,F) ,\\
\ch_{\Sigma ,\mu} (\cS ) &= \sup_{\varepsilon > 0} \ch_{\Sigma ,\mu}^\varepsilon (\cS ) ,
\end{align*}
where $F$ in the fourth line ranges over the nonempty finite subsets of $G$. If $\Hom_\mu^X (\cS ,F,m,\delta ,\sigma_i )$
is empty for all sufficiently large $i$, we set $\ch_{\Sigma ,\mu}^\varepsilon (\cS ,F,m,\delta ) = -\infty$.
In the case that $\cS$ is dynamically generating in the sense of the first paragraph of the previous section,
$\ch_{\Sigma ,\mu} (\cS )$ is equal to $h_{\Sigma ,\mu} (X,G)$ \cite[Prop.\ 5.4]{KerLi10}.

The following lemma is a measure-theoretic version of Proposition~4.8 in \cite{KerLi10}, saying that $\ch_{\Sigma ,\mu} (\cS)$ can also be computed by substituting $\rho_{\cS ,\infty}$ for $\rho_{\cS ,2}$,
and can be proved in the same way.

\begin{lemma}\label{L-infinity norm}
%For every $\varepsilon > 0$ and $\theta>0$ there is an $\varepsilon' > 0$ such that
%\[
%\limsup_{i\to\infty} \frac{1}{d_i} \log N_{\varepsilon} (\Hom_\mu^X (\cS ,F,m,\delta ,\sigma_i ),\rho_{\cS , \infty} )
%\leq \ch_{\Sigma, \mu}^{\varepsilon'} (\cS ,F, m, \delta ) + \theta
%\]
%for all nonempty finite sets $F\subseteq G$, $m\in\Nb$, and $\delta > 0$.
Let $\cS$
%$= \{ p_n \}_{n=1}^\infty$ '
be a sequence in the unit ball of $C_\Rb (X)$. Then
\[
\ch_{\Sigma ,\mu} (\cS) = \sup_{\varepsilon > 0} \inf_F \inf_{m\in\Nb} \inf_{\delta > 0}
\limsup_{i\to\infty} \frac{1}{d_i} \log N_{\varepsilon} (\Hom_\mu^X (\cS ,F,m,\delta ,\sigma_i ),\rho_{\cS , \infty} )
\]
where $F$ ranges over the nonempty finite subsets of $G$.

\end{lemma}

%It follows that $\ch_{\Sigma ,\mu} (\cS)$ can also be computed by substituting $\rho_{\cS ,\infty}$ for $\rho_{\cS ,2}$, i.e.,
%\[
%\ch_{\Sigma ,\mu} (\cS) = \sup_{\varepsilon > 0} \inf_F \inf_{m\in\Nb} \inf_{\delta > 0}
%\limsup_{i\to\infty} \frac{1}{d_i} \log N_{\varepsilon} (\Hom_\mu^X (\cS ,F,m,\delta ,\sigma_i ),\rho_{\cS , \infty} )
%\]
%where $F$ ranges over the nonempty finite subsets of $G$.

Now let $\rho$ be a continuous pseudometric on $X$. Recall the associated pseudometrics $\rho_2$ and $\rho_\infty$
as defined before Definition~\ref{D-map top}.

\begin{definition}
Let $F$ be a nonempty finite subset of $G$, $L$ a finite subset of $C(X)$, and $\delta > 0$.
Let $\sigma$ be a map from $G$ to $\Sym (d)$ for some $d\in\Nb$.
We define $\Map_\mu (\rho ,F, L,\delta ,\sigma )$ to be the set of all maps $\varphi : \{ 1,\dots ,d\} \to X$ such that
\begin{enumerate}
\item $\rho_2 (\varphi\circ\sigma_s , \alpha_s \circ\varphi ) < \delta$ for all $s\in F$, and

\item $\big| (\varphi_*\zeta)(f)
%d^{-1} \sum_{a=1}^d f(\varphi (a))
- \mu (f) \big| < \delta$ for all $f\in L$.
\end{enumerate}
\end{definition}

\begin{definition}
Let $F$ be a nonempty finite subset of $G$, $L$ a finite subset of $C(X)$, and $\delta > 0$.
For $\varepsilon > 0$ we define
\begin{align*}
h_{\Sigma ,\mu ,2}^\varepsilon (\rho ,F, L,\delta ) &=
\limsup_{i\to\infty} \frac{1}{d_i} \log N_\varepsilon (\Map_\mu (\rho ,F, L,\delta ,\sigma_i ),\rho_2 ) ,\\
h_{\Sigma ,\mu ,2}^\varepsilon (\rho ,F, L) &= \inf_{\delta > 0} h_{\Sigma ,\mu ,2}^\varepsilon (\rho ,F,L,\delta ) ,\\
h_{\Sigma ,\mu ,2}^\varepsilon (\rho ,F) &= \inf_{L} h_{\Sigma ,\mu ,2}^\varepsilon (\rho ,F,L) ,\\
h_{\Sigma ,\mu ,2}^\varepsilon (\rho ) &= \inf_{F} h_{\Sigma ,\mu ,2}^\varepsilon (\rho ,F) ,\\
h_{\Sigma ,\mu ,2} (\rho ) &= \sup_{\varepsilon > 0} h_{\Sigma ,\mu ,2}^\varepsilon (\rho ) ,
\end{align*}
where $L$ in the third line ranges over the finite subsets of $C(X)$ and
$F$ in the fourth line ranges over the nonempty finite subsets of $G$.
If $\Map_\mu (\rho ,F, L,\delta ,\sigma_i )$ is empty for all
sufficiently large $i$, we set $h_{\Sigma ,\mu ,2}^\varepsilon (\rho ,F, L,\delta ) = -\infty$.
We similarly define $h_{\Sigma ,\mu ,\infty}^\varepsilon (\rho ,F, L,\delta )$,
$h_{\Sigma ,\mu ,\infty}^\varepsilon (\rho ,F, L)$, $h_{\Sigma ,\mu ,\infty}^\varepsilon (\rho ,F)$,
$h_{\Sigma ,\mu ,\infty}^\varepsilon (\rho )$, and $h_{\Sigma ,\mu ,\infty} (\rho ) $ using $
N_\varepsilon(\cdot, \rho_\infty)$ in place of $N_\varepsilon(\cdot, \rho_2)$.
\end{definition}

%The proof of the following is essentially the same as that of its topological analogue, Proposition~\ref{P-infty top}.
%
%\begin{proposition}
%We have
%\begin{align*}
%h_{\Sigma ,\mu} (X,G) = \sup_\varepsilon \inf_F \inf_L \inf_{\delta > 0} \limsup_{i\to\infty}
%\frac{1}{d_i} \log N_\varepsilon (\Map_\mu (\rho ,F,L,\delta ,\sigma_i ),\rho_\infty ).
%\end{align*}
%where $F$ ranges over the nonempty finite subsets of $G$ and $L$ ranges over the nonempty
%finite subsets of $C(X)$.
%\end{proposition}

Recall from the previous section that $\rho$ is said to be dynamically generating if
for any distinct points $x,y\in X$ one has $\rho(sx, sy)>0$ for some $s\in G$.

\begin{proposition} \label{P-measure entropy}
Suppose that $\rho$ is dynamically generating. Then
\[
h_{\Sigma ,\mu} (X,G) = h_{\Sigma ,\mu, 2} (\rho )=h_{\Sigma, \mu ,\infty}(\rho) .
\]
\end{proposition}

\begin{proof}
One can argue as in the proof of Proposition~\ref{P-topological entropy}, appealing to
%the comment after
Lemma~\ref{L-infinity norm} in the case of $h_{\Sigma, \mu ,\infty}(\rho)$. The only extra thing to observe is that for $\cS$ and
$\rho'$ as in the proof of Proposition~\ref{P-topological entropy}, given any finite subset $L$ of $C(X)$ and $\delta>0$
there exist a nonempty finite subset $F$ of $G$, an $m\in \Nb$ and a $\delta'>0$ such that, for any $d\in \Nb$ and any map
$\sigma: G\rightarrow \Sym(d)$, if $\phi$ is a unital homomorphism $C(X)\rightarrow \Cb^d$ satisfying
$|\zeta\circ \phi(f)-\mu(f)|<\delta'$ for all $f\in \cS_{F, m}$, then $|(\Phi_*\zeta)(g)-\mu(g)|<\delta$ for all $g\in L$,
where $\Phi$ is the corresponding map $\{1, \dots, d\}\rightarrow X$. Indeed, since $\cS$ is dynamically generating
one can find a nonempty finite subset $F$ of $G$ and an $m\in \Nb$ such that for each
$g\in L$ there exists some $\tilde{g}$ in the linear span of $\cS_{F, m}\cup \{1\}$ with $\|g-\tilde{g}\|_{\infty}<\delta/4$.
Denote by $M$ the maximum over all $g\in L$ of the sum of the absolute values of the coefficients of $\tilde{g}$
written as a linear combination of elements in $\cS_{F, m}\cup \{1\}$. Then one may take $\delta'$ to be $\delta/(2M)$.
\end{proof}

%As the topological case, we have the following obvious consequence of the definitions.
%
%\begin{proposition}\label{P-subgroup measure}
%Let $G$ be a countable sofic group acting on a standard probability space $(X,\mu )$ by measure-preserving transformations.
%Let $H$ be a subgroup of $G$. Then the restriction of the action to $H$ satisfies
%$h_{\Sigma |_H ,\mu} (X,H) \geq h_{\Sigma ,\mu} (X,G)$ where $\Sigma |_H$ denotes the restriction of $\Sigma$ to $H$.
%\end{proposition}

\section{The Rokhlin lemma for sofic approximations of countable discrete amenable groups}\label{S-Rokhlin}

%Amenability in this case can be characterized by the existence of a F{\o}lner sequence, that is, a
%sequence $\{ F_i \}_{i=1}^\infty$ of nonempty finite subsets of $G$ such that
%$|F_i |^{-1} |sF_i \Delta F_i | \to 0$ as $i\to\infty$ for all $s\in G$.

Here we give a proof of the Rokhlin lemma for sofic approximations of countable discrete amenable
groups (Lemma~\ref{L-Rokhlin}), which will be used in both Sections~\ref{S-topological}
and \ref{S-measure}. The argument is extracted from \cite{OrnWei87}.

\begin{definition}
Let $(X,\mu )$ be a finite measure space and let $\delta \geq 0$.
A measurable set $A\subseteq X$ is said to {\it $\delta$-cover} or be a {\it $\delta$-covering} of $X$
if $\mu (A)\geq \delta \mu (X)$. A family of measurable subsets of $X$ is said to {\it $\delta$-cover} or be a
{\it $\delta$-covering} of $X$ if the union of its elements $\delta$-covers $X$.
A collection $\{ A_i \}_{i\in I}$ of
positive measure subsets of $X$ is said to be a {\it $\delta$-even covering} of $X$ if there exists
a number $M > 0$ such that $\sum_{i\in I} \mathbf{1}_{A_i} \leq M$ and
$\sum_{i\in I} \mu (A_i ) \geq (1-\delta ) M \mu (X)$.
We call $M$ a {\it multiplicity} of the $\delta$-even covering.
\end{definition}

\begin{definition}
Let $(X,\mu )$ be a finite measure space and let $\varepsilon \geq 0$.
A collection $\{ A_i \}_{i\in I}$ of positive measure sets is said to be
{\it $\varepsilon$-disjoint} if there exist pairwise disjoint sets $\widehat{A}_i \subseteq A_i$ such that
$\mu (\widehat{A}_i ) \geq (1-\varepsilon ) \mu (A_i )$ for all $i\in I$.
\end{definition}

The following two lemmas are from page~23 of \cite{OrnWei87}.

\begin{lemma}\label{L-frac}
Let $(X,\mu )$ be a finite measure space.
Let $\delta \in (0,1)$ and let $\{ A_i \}_{i\in I}$ be a countable $\delta$-even covering of $X$.
Then for every positive measure $B\subseteq X$ there exists an $i\in I$ such that
\[ \frac{\mu (A_i \cap B)}{\mu (A_i )} \leq \frac{\mu (B)}{(1-\delta) \mu (X)} . \]
\end{lemma}

\begin{proof}
If for some measurable $B\subseteq X$ we had
\[ \mu (A_i \cap B ) > \frac{\mu (B)}{(1-\delta )\mu (X)} \mu (A_i ) \]
for every $i\in I$, then taking a multiplicity $M$ for the $\delta$-even covering
and summing over $i$ would yield
\begin{align*}
\sum_{i\in I} \mu (A_i \cap B) &> \frac{\mu (B)}{(1-\delta )\mu (X)} \sum_{i\in I} \mu (A_i )
%\geq \frac{\mu (B)}{(1-\delta)\mu (X)} (1-\delta )M\mu (X) \\
\geq \mu (B) M \\
&\geq \int_X \mathbf{1}_B (x) \bigg( \sum_{i\in I} \mathbf{1}_{A_i} (x) \bigg) d\mu (x) \\
&= \int_X \bigg( \sum_{i\in I} \mathbf{1}_{A_i \cap B} (x) \bigg) d\mu (x)
= \sum_{i\in I} \bigg( \int_X \mathbf{1}_{A_i \cap B} (x) \bigg) d\mu (x)\\
&= \sum_{i\in I} \mu (A_i \cap B) ,
\end{align*}
a contradiction.
\end{proof}

\begin{lemma}\label{L-disjointcover}
Let $(X,\mu )$ be a finite measure space.
Let $\delta ,\varepsilon\in [0,1)$ and
let $\{ A_i \}_{i\in I}$ be a finite $\delta$-even covering of $X$ by positive measure sets.
Then there is an $\varepsilon$-disjoint subcollection of $\{ A_i \}_{i\in I}$
which $\varepsilon (1-\delta )$-covers $X$.
\end{lemma}

\begin{proof}
Take a maximal $\varepsilon$-disjoint subcollection
$\{ A_i \}_{i\in J}$ of $\{ A_i \}_{i\in I}$.
If this does not $\varepsilon (1-\delta )$-cover $X$ then by Lemma~\ref{L-frac} there is an
$i_0 \in I$ such that
\[ \frac{\mu \big( A_{i_0} \cap \hspace*{0.5mm} \bigcup_{i\in J} A_i \big)}{\mu (A_{i_0} )}
 \leq \frac{\mu \big( \bigcup_{i\in J} A_i \big)}{(1-\delta )\mu (X)} < \varepsilon \]
so that by adding $A_{i_0}$ to the collection $\{ A_i \}_{i\in J}$ we again have an
$\varepsilon$-disjoint collection, contradicting maximality.
\end{proof}

%Condition (1) in the following lemma is not needed for our purposes, but it can be used
%to show that sofic approximation sequences for amenable countable discrete groups have an asymptotic
%self-partitioning property, as discussed in Remark~\ref{R-asymptotic partitioning}.

\begin{lemma} \label{L-Rokhlin}
Let $G$ be a countable discrete group.
Let $0\le \tau<1$, and $0<\eta<1$. Then there are an $\ell\in \Nb$ and $\eta', \eta''>0$
such that, whenever $e\in F_1\subseteq F_2\subseteq \cdots \subseteq F_\ell$ are finite subsets of $G$ with $|(F_{k-1}^{-1}F_k) \setminus F_k|\le \eta'|F_k|$
for $k=2, \dots, \ell$, there exists a finite set $F\subseteq G$ containing $e$
%and $c_1 , \dots , c_\ell \in [0,1]$
such that
for every $d\in \Nb$, every map $\sigma: G\rightarrow \Sym(d)$ with a set $B\subseteq \{1, \dots, d\}$ satisfying
$|B|\ge (1-\eta'')d$ and
\[\sigma_{st}(a)=\sigma_s\sigma_t(a), \sigma_s(a)\neq \sigma_{s'}(a), \sigma_e(a)=a\]
for all $a\in B$ and $s, t, s'\in F$ with $s\neq s'$, and any set $V\subseteq \{1, \dots, d\}$ with $|V|\ge (1-\tau)d$,
there exist $C_1, \dots, C_\ell\subseteq V$ such that
\begin{enumerate}
%\item $||\sigma (F_k )C_k | - c_k d| \leq |F_k |$ for all $k=1,\dots ,\ell$,

\item for every $k=1, \dots, \ell$ and $c\in C_k$, the map $s\mapsto \sigma_s(c)$ from $F_k$ to $\sigma(F_k)c$ is bijective,

\item the sets $\sigma(F_1)C_1, \dots, \sigma(F_\ell)C_\ell$ are pairwise disjoint
and the family $\bigcup_{k=1}^\ell\{\sigma(F_k)c: c\in C_k\}$ is $\eta$-disjoint and $(1-\tau-\eta)$-covers $\{1, \dots, d\}$.
\end{enumerate}
\end{lemma}

\begin{proof}
Take $\eta', \eta''>0$ such that $1-\tau-2\eta''>0$, $\eta(1+\eta'/(1-\eta))<1$,
and $(1-\tau-2\eta'')(1+\eta'/(1-\eta))^{-1}>1-\tau-\eta$.
Define an increasing sequence $\{t_n\}_{n\in \Nb}$ in $[0, +\infty)$ by setting $t_1=\eta(1-\tau-\eta'')$
and, for $n\in\Nb$,
\[
t_{n+1} = \eta\bigg(1-\tau-\eta''-\bigg(1+\frac{\eta'}{1-\eta} \bigg) t_n \bigg) + t_n
\]
if $1-\tau-\eta''- (1+\eta'/(1-\eta )) t_n\ge 0$ and $t_{n+1} = t_n$ otherwise.
%$t_{n+1}=\eta(1-\tau-\eta''-(1+\eta'/(1-\eta))t_n)+t_n$ or $t_{n+1}=t_n$ for $n\in \Nb$ depending on whether
%$1-\tau-\eta''-(1+\eta'/(1-\eta))t_n\ge 0$ or not.
It is easily checked that $1-\tau-\eta''-(1+\eta'/(1-\eta))\lim_{n\to \infty}t_n\le 0$.
Thus there exists some $\ell\in \Nb$ with $1-\tau-\eta''-(1+\eta'/(1-\eta))t_\ell< \eta''$ and $t_1<t_2<\dots <t_\ell$. Then $t_\ell \ge (1-\tau-2\eta'')(1+\eta'/(1-\eta))^{-1}>1-\tau-\eta$.

Suppose now that
  $e\in F_1\subseteq F_2\subseteq \cdots \subseteq F_\ell$
  are finite subsets of $G$ with
  $|(F_{k-1}^{-1}F_k) \setminus F_k|\le \eta'|F_k|$
  for $k=2, \dots, \ell$. Set $F=F_\ell\cup F_\ell^{-1}$. Note that for every $c\in B$ and $k=1, \dots, \ell$ the map $s\mapsto \sigma_s(c)$ from $F_k$ to $\sigma(F_k)c$ is bijective. We will recursively construct the sets $C_1, \dots, C_\ell$ in reverse order
so that the sets $\sigma(F_1)C_1, \dots, \sigma(F_\ell)C_\ell$ are pairwise disjoint
and the family $\bigcup_{n=k+1}^\ell\{\sigma(F_n)c: c\in C_n\}$ is $\eta$-disjoint and $t_{\ell-k}$-covers $\{1, \dots, d\}$
for each $k=0,\dots ,\ell-1$. Since $t_\ell \geq 1-\tau-\eta$ we will thereby obtain condition (ii). Moreover
we will choose $C_1, \dots, C_\ell$ to be subsets of $B\cap V$ so that condition (i) holds automatically.

Note that $\sigma_{s^{-1}}\sigma_s(a)=\sigma_e(a)=a$ for all $a\in B$ and $s\in F_\ell$,
and thus for all distinct $a, c\in B$ and $s\in F_\ell$ we have
\[
\sigma_{s^{-1}}\sigma_s(a)\neq \sigma_{s^{-1}}\sigma_s(c),
\]
and hence
\[
\sigma_s(a)\neq \sigma_s(c).
\]

To begin the recursive construction we observe that
\[
\sum_{c\in B\cap V}|\sigma(F_\ell)c|=|F_\ell |\cdot |B\cap V|\ge |F_\ell|\cdot (1-\tau-\eta'')d ,
\]
so that the family $\{\sigma(F_\ell)c\}_{c\in B\cap V}$ is a $(\tau+\eta'')$-even covering of $\{1, \dots, d\}$ with
multiplicity $|F_\ell|$. By Lemma~\ref{L-disjointcover}, we can find a set
$C_\ell\subseteq B\cap V$ such that the family $\{\sigma(F_\ell)c\}_{c\in C_\ell}$ is $\eta$-disjoint and
$\eta(1-\tau-\eta'')$-covers $\{1, \dots, d\}$.
%We may assume $C_\ell$ to have minimal cardinality subject to these conditions.

Suppose that $1\le k<\ell$ and we have found $C_{k+1}, \dots, C_\ell\subseteq B\cap V$ such that the sets
$\sigma(F_{k+1})C_{k+1}, \dots, \sigma(F_\ell)C_\ell$ are pairwise disjoint and
the family $\bigcup_{n=k+1}^\ell\{\sigma(F_n)c: c\in C_n\}$ is $\eta$-disjoint and $t_{\ell-k}$-covers
$\{1, \dots, d\}$. Set $t'_{\ell-k}=|\bigcup_{n=k+1}^\ell \sigma(F_n)C_n|/d$ and
$E=\big\{c\in B\cap V: \sigma(F_k)c\cap \big(\bigcup_{n=k+1}^\ell\sigma(F_n)C_n\big)=\emptyset\big\}$.
For every $c\in (B\cap V)\setminus E$ we have $\sigma_s(c)=\sigma_t(a)$ for some
$n=k+1, \dots, \ell$, $a\in C_n$, $t\in F_n$, and $s\in F_k$, and hence
\[
c=\sigma_{s^{-1}}\sigma_s(c)=\sigma_{s^{-1}}\sigma_t(a)=\sigma_{s^{-1}t}(a)\in \bigcup_{n=k+1}^\ell\sigma(F_k^{-1}F_n)C_n.
\]
Therefore
\[
(B\cap V)\setminus E\subseteq \bigcup_{n=k+1}^\ell\sigma(F_k^{-1}F_n)C_n.
\]
For every $n=k+1, \dots, \ell$, since the family $\{\sigma(F_n)c: c\in C_n\}$ is $\eta$-disjoint we have
\[
(1-\eta)|F_n|\cdot |C_n|\le |\sigma(F_n)C_n|.
\]
Thus
\begin{align*}
\bigg|\bigcup_{n=k+1}^\ell \sigma(F_k^{-1}F_n)C_n\bigg|
&\le \bigg|\bigcup_{n=k+1}^\ell \sigma((F_k^{-1}F_n)\setminus F_n)C_n\bigg|+\bigg|\bigcup_{n=k+1}^\ell \sigma(F_n)C_n\bigg|\\
&\le \sum_{n=k+1}^\ell |(F_k^{-1}F_n)\setminus F_n|\cdot |C_n|+t'_{\ell-k}d\\
&\le \sum_{n=k+1}^\ell |(F_{n-1}^{-1}F_n)\setminus F_n|\cdot |C_n|+t'_{\ell-k}d\\
&\le \sum_{n=k+1}^\ell \eta'|F_n|\cdot |C_n|+t'_{\ell-k}d\\
&\le \sum_{n=k+1}^\ell \frac{\eta'}{1-\eta}|\sigma(F_n)C_n|+t'_{\ell-k}d\\
&= \bigg(1+\frac{\eta'}{1-\eta}\bigg) t'_{\ell-k}d,
\end{align*}
where the last equality follows form the assumption that the sets
$\sigma(F_{k+1})C_{k+1}, \dots, \sigma(F_\ell)C_\ell$ are pairwise disjoint.
Therefore
\begin{align*}
|E|=|B\cap V|-|(B\cap V)\setminus E| &\ge (1-\tau-\eta'')d-\bigg|\bigcup_{n=k+1}^\ell \sigma(F_k^{-1}F_n)C_n\bigg|\\
&\ge (1-\tau-\eta'')d-\bigg(1+\frac{\eta'}{1-\eta}\bigg)t'_{\ell-k}d.
\end{align*}
It follows that
\[
\sum_{c\in E}|\sigma(F_k)c|=|F_k |\cdot |E|\ge |F_k|\cdot \bigg(1-\tau-\eta''-\bigg(1+\frac{\eta'}{1-\eta}\bigg)t'_{\ell-k}\bigg)d.
\]
Thus the family $\{\sigma(F_k)c\}_{c\in E}$ is a $(\tau+\eta''+(1+\eta'(1-\eta)^{-1})t'_{\ell-k})$-even covering of $\{1, \dots, d\}$ with
multiplicity $|F_k|$. By Lemma~\ref{L-disjointcover}, we can find a set
$C_k\subseteq E$ such that the family $\{\sigma(F_k)c\}_{c\in C_k}$ is $\eta$-disjoint and
$\eta(1-\tau-\eta''-(1+\eta'(1-\eta)^{-1})t'_{\ell-k})$-covers $\{1, \dots, d\}$. Then
the sets $\sigma(F_k)C_k, \dots, \sigma(F_\ell)C_\ell$ are pairwise disjoint,
 and the family $\bigcup_{n=k}^\ell\{\sigma(F_n)c: c\in C_n\}$
is $\eta$-disjoint. Because the family $\bigcup_{n=k+1}^\ell\{\sigma(F_n)c: c\in C_n\}$ $t_{\ell-k}$-covers
$\{1, \dots, d\}$, we have $t'_{\ell-k}\ge t_{\ell-k}$.  Since $\eta(1+\eta'/(1-\eta))<1$, we get
 %$t_{\ell-k}$-covers
%$\{1, \dots, d\}$.
\begin{align*}
\bigg|\bigcup_{n=k}^\ell\sigma(F_n)C_n\bigg|&=|\sigma(F_k)C_k|+\bigg|\bigcup_{n=k+1}^\ell\sigma(F_n)C_n\bigg|\\
&\ge \eta(1-\tau-\eta''-(1+\eta'(1-\eta)^{-1})t'_{\ell-k})d+t'_{\ell-k}d\\
&= (\eta(1-\tau-\eta''-(1+\eta'(1-\eta)^{-1})t'_{\ell-k})+t'_{\ell-k})d\\
&\ge (\eta(1-\tau-\eta''-(1+\eta'(1-\eta)^{-1})t_{\ell-k})+t_{\ell-k})d\\
&= t_{\ell-(k-1)}d,
\end{align*}
completing the recursive construction.
%
%Observe finally that, setting $c_l = \eta (1-\tau - \eta'' )$ and
%$c_k = \eta(1-\tau-\eta''-(1+\eta'(1-\eta)^{-1})t'_{\ell-k})$ for $k=1,\dots ,\ell - 1$,
%we can arrange for condition (1) to hold by requiring that, for each $k=1,\dots ,\ell$,
%the set $C_k$ produced at the $(\ell - k+1)$th stage of the construction is chosen so as to have minimal cardinality.
\end{proof}

For an amenable countable discrete group $G$, by \cite[Cor.\ 5.3]{Namioka}, there is a F{\o}lner sequence $\{F_n\}_{n\in \Nb}$ of $G$
satisfying $F_n\subseteq F_{n+1}$ and $F_n^{-1}=F_n$ for all $n\in \Nb$.
In particular, this is a two-sided F{\o}lner sequence.
By using $\eta$-disjointness to pass to a genuinely disjoint family,
we obtain from Lemma~\ref{L-Rokhlin} the following.

\begin{lemma}\label{L-Rokhlin2}
Let $G$ be an amenable countable discrete group.
Let $0\le \tau<1$, $0<\eta<1$, $K$ be a nonempty finite subset of $G$, and $\delta>0$.
Then there are an $\ell\in \Nb$, nonempty finite subsets $F_1, \dots, F_\ell$ of $G$ with $|KF_k \setminus F_k|<\delta |F_k|$ and
$|F_kK\setminus F_k|<\delta|F_k|$ for
all $k=1, \dots, \ell$, a finite set $F\subseteq G$ containing $e$,
%$c_1 , \dots ,c_\ell \in [0,1]$,
and an $\eta'>0$ such that,
%whenever $e\in F_1\subseteq F_2\subseteq \cdots \subseteq F_\ell$ are finite subsets of $G$ with $|(F_{k-1}^{-1}F_k)/F_k|\le \eta'|F_k|$
%for $k=2, \dots, \ell$, there exists a finite set $e\in F\subseteq G$  such that
for every $d\in \Nb$, every map $\sigma: G\rightarrow \Sym(d)$ for which there is a set $B\subseteq \{1, \dots, d\}$ satisfying
$|B|\ge (1-\eta')d$ and
\[
\sigma_{st}(a)=\sigma_s\sigma_t(a), \sigma_s(a)\neq \sigma_{s'}(a), \sigma_e(a)=a
\]
for all $a\in B$ and $s, t, s'\in F$ with $s\neq s'$, and every set $V\subseteq \{1, \dots, d\}$ with $|V|\ge (1-\tau)d$,
there exist $C_1, \dots, C_\ell\subseteq V$ such that
\begin{enumerate}
%\item $||\sigma (F_k )C_k | - c_k d| \leq |F_k |$ for all $k=1,\dots ,\ell$,
\item for every $k=1, \dots, \ell$, the map $(s, c)\mapsto \sigma_s(c)$ from $F_k\times C_k$ to $\sigma(F_k)C_k$ is bijective,

\item the family $\{ \sigma(F_1)C_1, \dots, \sigma(F_\ell)C_\ell \}$ is disjoint and $(1-\tau-\eta)$-covers $\{1, \dots, d\}$.
\end{enumerate}
\end{lemma}

\section{Topological entropy in the amenable case}\label{S-topological}

%Throughout this section $X$ is compact metrizable space.
We begin by recalling the classical definition of topological entropy \cite{AdlKonMcA65,Mou85}.
Let $G$ be an amenable countable discrete group and $\alpha$ a continuous action of $G$ on a compact metrizable space $X$.
%Amenability in this case can be characterized by the existence of a F{\o}lner sequence, that is, a
%sequence $\{ F_i \}_{i=1}^\infty$ of nonempty finite subsets of $G$ such that
%$|F_i |^{-1} |sF_i \Delta F_i | \to 0$ as $i\to\infty$ for all $s\in G$.
For an open cover $\cU$ of $X$ we write $N(\cU )$ for the minimal cardinality of a subcover of $\cU$.
For a nonempty finite set $F\subseteq G$ we abbreviate $\bigvee_{s\in F} s^{-1} \cU$ to $\cU^F$.
As guaranteed by the subadditivity result in Section~6 of \cite{LinWei00}, for a finite open cover
$\cU$ of $X$ the quantities
\[
\frac{1}{|F|} \log N( \cU^F)
\]
converge to a limit as the nonempty finite set $F\subseteq G$ becomes more and more left invariant in the sense
that for every $\varepsilon > 0$ there are a nonempty finite set $K\subseteq G$ and a $\delta > 0$
such that the displayed quantity is within $\varepsilon$ of the limiting value whenever $|KF\Delta F| \leq \delta |F|$.
We write this limit as $h_\topol (\cU )$.
The classical topological entropy $h_\topol (X,G)$ is defined as the supremum of the quantities $h_\topol (\cU )$
over all finite open covers $\cU$ of $X$.

Given a F{\o}lner sequence $\{ F_k \}_{k=1}^\infty$ and a compatible metric $\rho$ on $X$,
the entropy $h_\topol (X,G)$ can be alternatively expressed as
\[
\sup_{\varepsilon > 0} \limsup_{k\to\infty} \frac{1}{|F_k|} \log N_\varepsilon (X,\rho_{F_k} )
\]
using the notation established in the introduction.
%Note also that
%\[
%h_\topol (X,G) \geq \sup_{\varepsilon > 0} \limsup_{k\to\infty} \frac1k \log N_\varepsilon (X,\rho_{F_k} )
%\]
%whenever $\rho$ is a continuous pseudometric on $X$.
This approach to entropy using metrics was introduced by Rufus Bowen for $\Zb$-actions \cite{Bow71},
and the standard arguments showing its equivalence in that case with the open cover definition
apply equally well to the general amenable setting.

Let $\Sigma$ be a fixed sofic approximation sequence for $G$.
We will prove in this section that $h_\Sigma (X,G) = h_\topol (X,G)$.
The basis for the argument is the fact that
every good enough sofic approximation for $G$ can be approximately decomposed into copies of F{\o}lner sets
(Lemma~\ref{L-Rokhlin}). This decomposition implies that the maps in the definition of
sofic topological entropy approximately decompose into partial orbits over F{\o}lner sets.
%up to some corruption caused by the $\ell^2$-norm approximation.

%In the proof of the following lemma,
%for a set $X$ with a pseudometric $\rho$, we will write
%$\spa_\varepsilon (X,\rho )$ for the minimal cardinality of a $(\rho, \varepsilon)$-spanning subset of $X$
%and $\sep_\varepsilon (X,\rho )$ for the maximal cardinality of a $(\rho, \varepsilon)$-separated subset of $X$.

\begin{lemma}\label{L-ineq1}
Let $G$ be an amenable countable discrete group acting continuously on a compact metrizable space $X$.
%Let $\Sigma$ be a sofic approximation sequence for $G$.
Then $h_\Sigma (X,G)\le h_\topol(X,G)$.
\end{lemma}

\begin{proof}
We may assume that $h_\topol(X,G)<\infty$. Let $\rho$ be a compatible metric on $X$.
%Let $\cS = \{ p_n \}_{n=1}^\infty$ be a sequence in the unit ball of $C_\Rb (X)$ which generates
%$C(X)$ as a unital $C^*$-algebra. Such a sequence exists by metrizability.
Let $\varepsilon , \kappa > 0$. To establish the lemma, by Proposition~\ref{P-topological entropy} it suffices to show that
$h^{\varepsilon}_{\Sigma ,\infty} (\rho)\le h_\topol(X,G) + 4\kappa$.

%Take an $m\in\Nb$ such that $2^{-(m-1)} < \varepsilon$. Set $\cM = \{ p_1 , \dots , p_m \}$.
%Define on $X$ the continuous pseudometric
%\[
%\rho(x,y) = \sup_{f\in \cM} |f(x) - f(y)| .
%\]
%Then
There are a nonempty finite subset $K$ of $G$ and $\delta'>0$ such that
$N_{\varepsilon/4}(X, \rho_{F'})<\exp((h_\topol(X,G) + \kappa)|F'|)$ for every nonempty finite subset
$F'$ of $G$ satisfying $|KF'\setminus F'|<\delta'|F'|$.

Take an $\eta\in (0,1)$ such that $(N_{\varepsilon/4}(X, \rho ))^{2\eta} \leq \exp(\kappa)$ and
$(1-\eta )^{-1} (h_\topol(X,G) + \kappa) \leq h_\topol(X,G) + 2\kappa$.
%By Lemma~\ref{L-Rokhlin} there are an $\ell\in\Nb$ and an $\eta' > 0$ such that,
%whenever $e\in F_1\subseteq F_2\subseteq \cdots \subseteq F_\ell$ are finite subsets of $G$ with
%$|(F_{k-1}^{-1}F_k)/F_k|\le \eta'|F_k|$
%for $k=2, \dots, \ell$, for every map $\sigma : G\to\Sym(d)$ for some $d\in \Nb$ which is a good
%enough sofic approximation for $G$ and every $Y\subseteq \{1, \dots, d\}$ with $|Y|/d\ge 1-\eta$
%there exist $C_1, \dots, C_\ell \subseteq Y$ such that
%\begin{enumerate}
%\item for all $k=1, \dots, \ell$ and $c\in C_k$
%the map $s\mapsto \sigma_s(c)$ from $F_k$ to $\sigma(F_k)c$ is bijective,
%
%\item the family $\bigcup_{k=1}^\ell\{\sigma(F_k)c: c\in C_k\}$ is $\eta$-disjoint
%and $(1-2\eta)$-covers $\{1, \dots, d\}$.
%\end{enumerate}
Let $\ell\in\Nb$ and $\eta'> 0$ be as given by Lemma~\ref{L-Rokhlin} with respect to $\eta$ and $\tau = \eta$.
Take finite subsets
$e\in F_1 \subseteq F_2 \subseteq \dots \subseteq F_\ell$
of $G$ such that $|(F_{k-1}^{-1} F_k )\setminus F_k| \leq \eta' |F_k |$ for $k=2, \dots , \ell$ and $|KF_k\setminus F_k|<\delta'|F_k|$ for every $k=1, \dots, \ell$. Then
\begin{align} \label{E-top upper}
N_{\varepsilon/4}(X, \rho_{F_k}) \leq \exp((h_\topol(X,G) + \kappa)|F_k |)
\end{align}
for every $k=1, \dots, \ell$.

Let $\delta > 0$ be a small positive number which we will determine in a moment.
%such that $\delta \leq\eta$ and $2\sqrt{\delta} < \varepsilon /4$.
Let $\sigma$ be a map from $G$ to $\Sym (d)$ for some $d\in\Nb$ which is a good enough sofic approximation for $G$.
We will show that $N_{\varepsilon}(\Map(\rho, F_{\ell}, \delta, \sigma), \rho_{\infty})\le \exp((h_\topol(X,G) + 4\kappa)d$,
which will complete the proof since we can then conclude that
$h^{\varepsilon}_{\Sigma ,\infty}(\rho, F_\ell, \delta)\le h_\topol(X,G) + 4\kappa$
and hence $h^{\varepsilon}_{\Sigma ,\infty}(\rho)\le h_\topol(X,G) + 4\kappa$.

%When $\sigma$ is a good enough sofic approximation for $G$, the set
%$\Lambda$ of all $a\in \{1, \dots, d\}$ satisfying $\sigma_{s^{-1}}\sigma_s(a)=a$ for all $s\in F_{\ell}$ has cardinality
%at least $(1-\delta)d$.
For every $\varphi\in\Map (\rho ,F_\ell, \delta, \sigma )$, we have
$\rho_2(\varphi\circ \sigma_{s}, \alpha_s\circ \varphi) < \delta$
for all $s\in F_\ell$. Thus the set $\Lambda_{\varphi}$
of all $a\in\{ 1,\dots ,d \}$ such that
\[
\rho(\varphi (sa), s\varphi(a)) < \sqrt{\delta}
\]
for all $s\in F_\ell$ has cardinality at least $(1-|F_\ell |\delta) d$.
%Then $|\Lambda_{\varphi} \cap \Lambda|/d\ge 1-(|F_{\ell}|+1)\delta$.

For each $J\subseteq \{1, \dots, d\}$ we
%write $P_J$ for the canonical projection of $\Cb^d$
%onto $\Cb^J$ and define on the set of unital homomorphisms from $C(X)$ to $\Cb^d$ the pseudometric
define on the set of maps from $\{1, \dots, d\}$ to $X$ the pseudometric
\[
\rho_{J,\infty} (\varphi , \psi ) = \rho_{\infty} (\varphi|_J,  \psi|_J).
\]
%When $J = \{ 1,\dots ,d\}$ we simply write $\rho_{\cM ,\infty}$.

Take a $(\rho_\infty, \varepsilon )$-separated subset $D$ of $\Map (\rho ,F_\ell, \delta , \sigma )$
of maximal cardinality.
%Note that $D$ is $\varepsilon$-separated with respect to
%$\rho_{\cM ,\infty}$ since
%$\rho_\cS\le \rho_{\cM, \infty}+2^{-(m-1)} < \rho_{\cM, \infty}+\varepsilon$.

Set $n = |F_\ell |$. When $n\delta <1$,
the number of subsets of $\{1, \dots, d\}$ of cardinality no greater than $n\delta d$
is equal to $\sum_{j=0}^{\lfloor n\delta d \rfloor} \binom{d}{j}$, which is at most
$n\delta d \binom{d}{n\delta d}$,
which by Stirling's approximation is less than $\exp(\beta d)$
for some $\beta > 0$ depending on $\delta$ and $n$
but not on $d$ when $d$ is sufficiently large with $\beta\to 0$ as $\delta\to 0$ for a fixed $n$.
Thus when $\delta$ is small enough and $d$ is large enough, there is a subset $W$ of $D$ with $\exp(\kappa d)|W|\ge |D|$
such that the set $\Lambda_{\varphi}$ is the same, say $\Theta$, for every $\varphi\in W$,
and $|\Theta|/d>1-\eta$.

Since we chose $\ell$ and $\eta'$ so that the conclusion of Lemma~\ref{L-Rokhlin} holds, when
$\sigma$ is a good enough sofic approximation
for $G$, there exist $C_1, \dots, C_\ell \subseteq \Theta$ such that
%$\sigma^{-1}_{s^{-1}}(c) = \sigma_s (c)$ for all
%$c\in\bigcup_{k=1}^\ell C_k$ and $s\in F_\ell$ and
\begin{enumerate}
\item for all $k=1, \dots, \ell$ and $c\in C_k$
the map $s\mapsto \sigma_s(c)$ from $F_k$ to $\sigma(F_k)c$ is bijective,

\item the family $\bigcup_{k=1}^\ell\{\sigma(F_k)c: c\in C_k\}$ is $\eta$-disjoint
and $(1-2\eta)$-covers $\{1, \dots, d\}$.
\end{enumerate}
Denote by $\sL$ the set of all pairs $(k,c)$ such that $k\in \{ 1,\dots ,\ell\}$ and $c\in C_k$. By $\eta$-disjointness,
for every $(k, c)\in \sL$ we can find an $F_{k, c}\subseteq F_k$ with
$|F_{k, c}|\ge (1-\eta)|F_k|$ such that the sets $\sigma(F_{k, c})c$ for $(k, c)\in \sL$ are pairwise disjoint.

Let $(k, c)\in \sL$. Take an $(\varepsilon/2)$-spanning subset $V_{k,c}$ of $W$ with respect to
$\rho_{\sigma(F_{k, c}) c,\infty}$ of minimal cardinality. We will show that
$|V_{k, c}|\le \exp((h_\topol(X,G) + \kappa)|F_k|)$ when $\delta$ is small enough.
To this end, let $V$ be an $(\varepsilon/2)$-separated subset of $W$ with respect to $\rho_{\sigma(F_{k, c})c ,\infty}$.
%For each $\varphi\in V$ the unital homomorphism $f\mapsto \varphi (f)(c)$ on $C(X)$ is
%given by evaluation at some point $x_\varphi$ in $X$.
For any two distinct elements $\varphi$ and $\psi$ of $V$ we have, for every $s\in F_{k, c}$, since $c\in \Lambda_{\varphi}\cap \Lambda_{\psi}$,
\begin{align*}
\rho(s\varphi(c), s\psi(c))&\ge \rho(\varphi(sc), \psi(sc))-\rho(s\varphi(c), \varphi(sc))-\rho(s\psi(c), \psi(sc))\\
&\ge  \rho(\varphi(sc), \psi(sc))-2\sqrt{\delta},
%|\varphi\circ\alpha_{s^{-1}} (f)(c) - \psi\circ\alpha_{s^{-1}} (f)(c)|
%&\geq | \sigma_{s^{-1}} \circ \varphi (f)(c) - \sigma_{s^{-1}} \circ \psi (f)(c) | \\
%&\hspace*{10mm} \ - |\psi\circ\alpha_{s^{-1}} (f)(c) - \sigma_{s^{-1}} \circ \psi (f)(c) | \\
%&\hspace*{10mm} \ - | \sigma_{s^{-1}} \circ\varphi (f)(c) - \varphi\circ\alpha_{s^{-1}} (f)(c) | \\
%&\geq |\varphi(f)(\sigma^{-1}_{s^{-1}}(c))-\psi(f)(\sigma^{-1}_{s^{-1}}(c))|-2\sqrt{\delta}\\
%&= | \varphi (f)(\sigma_s (c)) - \psi (f)(\sigma_s (c)) | - 2\sqrt{\delta},
\end{align*}
and hence
\begin{align*}
\rho_{F_{k, c}}(\varphi(c), \psi(c))&=\max_{s\in F_{k, c}}\rho(s\varphi(c), s\psi(c))\ge \max_{s\in F_{k, c}}\rho(\varphi(sc), \psi(sc))-2\sqrt{\delta}>\varepsilon/2-\varepsilon/4=\varepsilon/4,
%\sup_{s\in F_{k, c}}\rho (sx_\varphi , sx_\psi )
%&= \sup_{s\in F_{k, c}} \sup_{f\in\cM} |f(sx_\varphi ) - f(sx_\psi )| \\
%&= \sup_{s\in F_{k, c}} \sup_{f\in\cM} |\varphi\circ\alpha_{s^{-1}} (f)(c) - \psi\circ\alpha_{s^{-1}} (f)(c)| \\
%&\geq \sup_{s\in F_{k, c}} \sup_{f\in\cM} | \varphi (f)(\sigma_s (c)) - \psi (f)(\sigma_s (c)) | - 2\sqrt{\delta} \\
%&= \rho_{\cM, \sigma(F_{k, c})c, \infty}(\varphi, \psi)-2\sqrt{\delta}\\
%&> \varepsilon/2 - \varepsilon/4 = \varepsilon/4,
\end{align*}
granted that $\delta$ is taken small enough.
Thus $\{\varphi(c) : \varphi\in V \}$ is a $(\rho_{F_{k, c}}, \varepsilon/4 )$-separated subset of $X$ of cardinality $|V|$,
so that
\begin{align*}
|V|\le N_{\varepsilon/4}(X, \rho_{F_{k, c}})\le N_{\varepsilon/4}(X, \rho_{F_k})
\overset{\eqref{E-top upper}}\le \exp((h_\topol(X,G) + \kappa)|F_k|).
\end{align*}
Therefore
\[
|V_{k, c}|
%=\spa_{\varepsilon/2}(W,  \rho_{\cP ,\sigma(F_{k, c})c ,\infty})
\le N_{\varepsilon/2}(W, \rho_{\sigma(F_{k, c})c ,\infty})\le \exp((h_\topol(X,G) + \kappa)|F_k|),
\]
as we wished to show.

Set
\[
H = \{ 1,\dots ,d\} \setminus \bigcup \{ \sigma (F_{k, c})c : (k,c)\in \sL\} .
\]
and take an $(\varepsilon/2)$-spanning subset $V_H$ of $W$ with respect to $\rho_{H,\infty}$ of minimal cardinality.
%Then $V_H$ is $(\varepsilon/2)$-spanning with respect to $\rho_{\cS ,H,\infty}$/
%Since the image of every element of $\cM$ under a unital homomorphism from $C(X)$ into $\Cb^H$
%is contained in $[-1,1]^H$,
We have
\[
|V_H | \leq (N_{\varepsilon/4}(X, \rho ))^{|H|} \leq (N_{\varepsilon/4}(X, \rho ))^{2\eta d}.
\]

Write $U$ for the set of all maps
$\varphi : \{1, \dots, d\}\rightarrow X$ such that $\varphi|_H\in V_H|_H$ and
$\varphi|_{\sigma (F_{k, c})c}\in V_{k,c}|_{\sigma (F_{k, c})c}$ for all $(k,c)\in\sL$.
Then, by our choice of $\eta$,
\begin{align*}
|U | &= |V_H | \prod_{(k,c)\in\sL} |V_{k,c} |
\le (N_{\varepsilon/4}(X, \rho ))^{2\eta d} \exp\bigg(\sum_{(k, c)\in \sL}(h_\topol(X,G) + \kappa)|F_k|\bigg)\\
&= (N_{\varepsilon/4}(X, \rho ))^{2\eta d}\exp\bigg((h_\topol(X,G) + \kappa)\sum_{k=1}^\ell|F_k| |C_k|\bigg)\\
&\le \exp(\kappa d)
\exp\bigg(\frac{1}{1-\eta}(h_\topol(X,G) + \kappa)d\bigg)\\
&\le \exp(\kappa d)
\exp((h_\topol(X,G) + 2\kappa)d)=\exp((h_\topol(X,G) + 3\kappa)d).
%\\
%&\le \exp ((h_\topol (X,G )+3\kappa)d) .
\end{align*}

Now since every element of $W$ lies within $\rho_{\infty}$-distance $\varepsilon/2$ to an element
of $U$ and $W$ is $\varepsilon$-separated with respect to $\rho_{\infty}$,
the cardinality of $W$ is at most that of $U$.
Therefore
\begin{align*}
N_{\varepsilon}(\Map (\rho, F_\ell, \delta, \sigma), \rho_\infty)&=|D|\le \exp(\kappa d)|W|\le \exp(\kappa d)|U|\\
&\le \exp(\kappa d)\exp((h_\topol(X,G) + 3\kappa)d)\\
&=\exp((h_\topol(X,G) + 4\kappa)d),
\end{align*}
%\begin{align*}
%\frac1d \log N_{\varepsilon}(\Hom (\cS, F_\ell^{-1}, 2^{-m}\delta, \sigma), \rho_\cS) &= \frac1d \log |W|\le \frac1d \log |U|
%\le h_\topol (X,G )+3\kappa .
%\end{align*}
%Since $\sigma$ was any sufficiently good sofic approximation,
%we conclude that
%\[
%h^{\varepsilon}_{\Sigma} (\cS)\le h^{\varepsilon}_{\Sigma} (\cS , F_\ell^{-1} ,\delta ) \le h_\topol(X,G) + 3\kappa ,
%\]
as desired.
\end{proof}

\begin{lemma}\label{L-ineq2}
Let $G$ be an amenable countable discrete group acting continuously on a compact metrizable space $X$.
%Let $\Sigma$ be a sofic approximation sequence for $G$.
Then $h_\Sigma (X,G)\ge h_\topol(X,G)$.
\end{lemma}

\begin{proof}
Let $\rho$ be a compatible metric on $X$.
%Let $\cS = \{ p_k \}_{k=1}^\infty$ be a sequence in the unit ball of $C_\Rb (X)$ which generates
%$C(X)$ as a unital $C^*$-algebra, as exists by metrizability.
%Let $g$ be a nonnegative real number no bigger than $h_\topol(X, G)$, and let $\theta>0$.
Let $\cU$ be a finite open cover of $X$, and let $\theta>0$.
To prove the lemma it suffices to show that $h_{\Sigma ,\infty} (\rho ) \geq h_\topol(\cU) - 2\theta$.

%Since $\cS$ generates $C(X)$ as a unital $C^*$-algebra, we have the compatible metric $\rho'$ on $X$ defined by
%\[ \rho'(x, y)=\sum_{n=1}^{\infty}\frac{1}{2^n}|p_n(x)-p_n(y)|.\]
Take $\varepsilon>0$ such that every open $\varepsilon$-ball in $X$ with respect to $\rho$ is contained in some atom of
$\cU$. Then $N_{\varepsilon}(X, \rho_{F'})\ge N(\cU^{F'})$ for every nonempty finite subset $F'$ of $G$. Thus, when
$F'$ is sufficiently left invariant, one has $|F'|^{-1}\log N_{\varepsilon}(X, \rho_{F'})\ge  h_\topol(\cU) - \theta$.

%Take an $m\in\Nb$ such that $2^{-(m-1)} \leq \varepsilon$. Now by Lemma~\ref{L-infinity norm top}
Let $F$ be a nonempty finite subset of $G$ and $\delta > 0$. Let $\sigma$ be a map from $G$ to $\Sym (d)$
for some $d\in\Nb$.
Now it suffices to show that if $\sigma$ is a good enough sofic approximation then
\begin{align} \label{E-top lower}
\frac{1}{d} \log N_{\varepsilon} (\Map (\rho ,F,\delta ,\sigma ),\rho_{\infty} )\ge h_\topol(\cU) - 2\theta .
\end{align}

%Consider the continuous pseudometric $\rho$ on $X$ given by
%\[ \rho(x, y)=\max_{1\le n\le m}|p_n(x)-p_n(y)|.\]
%Since $\rho'\le \rho+2^{-(m-1)}\le \rho+\varepsilon$, for any nonempty finite subset $F'$ of $G$, every $(\rho'_{F'}, 2\varepsilon)$-separated subset of $X$ is $(\rho_{F'}, \varepsilon)$-separated. Thus, when
%$F'$ is sufficiently left invariant, one has $|F'|^{-1}\log N_{\varepsilon}(X, \rho_{F'})\ge |F'|^{-1}\log N_{2\varepsilon}(X, \rho'_{F'})\ge  h_\topol(\cU) - \theta$.

%Let $K$ be a nonempty finite subset of $G$ and let $\delta > 0$. Take an $m\in\Nb$ such that $2^{-m} \leq \delta /2$.
%Since $\cS$ generates $C(X)$ as a unital $C^*$-algebra, by a straightforward approximation argument
%we may assume by taking $m$ larger if necessary that there is an $\varepsilon > 0$ such that
%$h_\topol (X,G) - \kappa$ is less than the limit of $|F|^{-1} \log \sep_{d_\cM ,F} (\varepsilon )$
%as the nonempty finite set $F$ becomes more and more left invariant,
%where $\cM = \{ p_1 ,\dots ,p_m \}$ and $d_\cM$ is the pseudometric on $X$ given by
%\[
%d_\cM (x,y) = \sup_{f\in\cM} | f(x) - f(y) | .
%\]

Take $\delta'>0$ such that $\sqrt{\delta'}\diam_{\rho}(X)<\delta/2$ and
$(1-\delta' )(h_\topol (\cU ) - \theta ) \ge h_\topol (\cU ) - 2\theta$.
By Lemma~\ref{L-Rokhlin2} there are an $\ell\in \Nb$
%depending only on $\delta''$
and nonempty finite subsets $F_1, \dots, F_\ell$ of $G$
which are sufficiently left invariant so that
\[
\inf_{k=1,\dots ,\ell} \frac{1}{|F_k |} \log N_{\varepsilon}(X, \rho_{F_k}) \ge h_\topol(\cU) - \theta
\]
such that for every map $\sigma : G\to\Sym(d)$ for some $d\in \Nb$ which is a good
enough sofic approximation for $G$ there exist  $C_1, \dots, C_\ell\subseteq \{1, \dots, d\}$ satisfying the following:
\begin{enumerate}
\item for every $k=1, \dots, \ell$, the map $(s, c)\mapsto \sigma_s(c)$ from $F_k\times C_k$ to $\sigma(F_k)C_k$ is bijective,

\item the family $\{ \sigma(F_1)C_1, \dots, \sigma(F_\ell)C_\ell \}$ is disjoint and $(1-\delta')$-covers $\{1, \dots, d\}$.
\end{enumerate}

For every $k\in\{ 1,\dots, \ell\}$ pick an $\varepsilon$-separated set
$E_k \subseteq X$ with respect to $\rho_{F_k}$ of maximal cardinality.
For each $h = (h_k )_{k=1}^\ell \in\prod_{k=1}^\ell (E_k )^{C_k}$
take a map $\varphi_h : \{1, \dots, d\}\rightarrow X$ such that
\[
\varphi_h(s c) = s(h_k (c))
\]
for all $k\in \{ 1,\dots ,\ell \}$, $c\in C_k$, and $s\in F_k$.
Observe that if $\max_{k=1,\dots ,\ell} |FF_k \Delta F_k| /|F_k|$ is small enough, as will be the case
if we take $F_1, \dots, F_\ell$ to be sufficiently left invariant, and $\sigma$ is a good enough sofic approximation for $G$, then
we will have $\rho_2(\alpha_s\circ \varphi_h,  \varphi_h\circ \sigma_s)< \delta$ for all
$s\in F$, so that $\varphi_h \in\Map (\rho , F ,\delta , \sigma )$.

Now if $h = (h_k )_{k=1}^\ell$ and $h' = (h_k' )_{k=1}^\ell$ are distinct elements of
$\prod_{k=1}^\ell (E_k )^{C_k}$, then $h_k (c) \neq h_k' (c)$ for some $k\in \{ 1,\dots ,\ell \}$ and $c\in C_k$.
Since $h_k (c)$ and $h_k' (c)$ are distinct points in $E_k$ which is $\varepsilon$-separated with respect to $\rho_{F_k}$,
$h_k (c)$ and $h_k' (c)$ are $\varepsilon$-separated with respect to $\rho_{F_k}$, and thus
we have
%$\| \varphi_h (f) - \varphi_{h'} (f) \|_\infty \ge \varepsilon$ for some $f\in \{ p_1 ,\dots ,p_m \}$
%and hence
$\rho_{\infty} (\varphi_h ,\varphi_{h'} ) \ge \varepsilon$. Therefore
\begin{align*}
\frac1d \log N_{\varepsilon} (\Map (\rho ,F,\delta ,\sigma ),\rho_{\infty} )
&\geq \frac1d \sum_{k=1}^\ell |C_k | \log |E_k | \\
&\geq \frac1d \sum_{k=1}^\ell |C_k | |F_k | (h_\topol (\cU) - \theta ) \\
&\geq (1-\delta' )(h_\topol (\cU) - \theta ) \\
&\geq h_\topol (\cU ) - 2\theta ,
\end{align*}
as desired.
%Since $K$ was an arbitrary nonempty finite subset of $G$ and $\delta$ an arbitrary positive number,
%It follows that
%\[
%\liminf_{i\to\infty} \frac{1}{d_i} \log N_{\varepsilon} (\Map (\rho ,F,\delta ,\sigma_i ),\rho_{\infty} )
%\geq (1-\delta' )(h_\topol (\cU) - \theta ).
%\]
%where $F$ ranges over all nonempty finite susbet of $G$.
%By Lemma~\ref{L-infinity norm top} there is an $\varepsilon' > 0$ such that the left side of the above
%inequality is smaller that $h_\Sigma^{\varepsilon'} (\cS ) + \kappa$, which in turn is less than or equal to
%$h_\Sigma (\cS ) + \kappa$. Since $\kappa$ was an arbitrary positive number, we conclude that
%$h_\Sigma (\cS ) \geq h_\topol (X,G)$.
%Letting $\delta'\to 0$, we get the inequality \eqref{E-top lower}.
%\[\liminf_{i\to\infty} \frac{1}{d_i} \log N_{\varepsilon /2^m} (\Hom (\cS ,F,\delta ,\sigma_i ),\rho_{\cS ,\infty} )
%\geq h_\topol (cU) - \theta,\]
%as desired.
\end{proof}

Combining Lemmas~\ref{L-ineq1} and \ref{L-ineq2} we obtain the desired equality of entropies:

\begin{theorem}\label{T-amenable}
Let $G$ be an amenable countable discrete group acting continuously on a compact metrizable space $X$.
Let $\Sigma$ be a sofic approximation sequence for $G$. Then
\[
h_\Sigma (X,G) = h_\topol (X,G) .
\]
\end{theorem}

%An immediate consequence of the definition of sofic topological entropy is that it does not decrease
%when restricting the action and the given sofic approximation sequence to a subgroup. We thus obtain the following corollary
%to the above theorem.

%\begin{corollary}
%Let $G$ be an amenable countable discrete group acting continuously on a compact metrizable space $X$.
%Let $H$ be a subgroup of $G$. Then the restriction of the action to $H$ satisfies $h_\topol (X,H) \geq h_\topol (X,G)$.
%\end{corollary}

%\begin{proposition}\label{P-subgroup top}
%Let $G$ be a countable sofic group acting continuously on a compact metrizable space $X$.
%Let $H$ be a subgroup of $G$. Then restriction of the action to $H$ satisfies
%$h_{\Sigma |_H} (X,H) \geq h_\Sigma (X,G)$ where $\Sigma |_H$ denotes the restriction of $\Sigma$ to $H$.
%\end{proposition}
%In view of Proposition~\ref{P-subgroup top}, we obtain the following corollary.

\section{Measure entropy in the amenable case}\label{S-measure}

Let $G$ be an amenable countable discrete group acting on a standard probability space $(X,\mu )$
by measure-preserving transformations. The entropy of a measurable partition $\cQ$ of $X$ is defined by
\[
H_\mu (\cQ) = - \sum_{Q\in\cQ} \mu (Q) \log \mu (Q) .
\]
For a nonempty finite set $F\subseteq G$ we abbreviate $\bigvee_{s\in F} s^{-1} \cQ$ to $\cQ^F$.
By the subadditivity result in Section~6 of \cite{LinWei00}, for a finite measurable partition $\cQ$
of $X$ the quantities
\[
\frac{1}{|F|} \log H_\mu (\cQ^F )
\]
converge to a limit as the nonempty finite set $F\subseteq G$ becomes more and more left invariant in the sense
that for every $\varepsilon > 0$ there are a nonempty finite set $K\subseteq G$ and a $\delta > 0$
such that the displayed quantity is within $\varepsilon$ of the limiting value whenever $|KF\Delta F| \leq \delta |F|$.
We write this limit as $h_\mu (\cQ )$.
The classical Kolmogorov-Sinai measure entropy $h_\mu (X,G)$ is defined as the supremum of the quantities $h_\mu (\cQ )$
over all finite measurable partitions $\cQ$ of $X$.

Throughout this section $\Sigma = \{ \sigma_i : G\to\Sym (d_i ) \}_{i=1}^\infty$ is a fixed but arbitrary
sofic approximation sequence for $G$.
Our objective in this section is to show that
the sofic entropy $h_{\Sigma ,\mu} (X,G)$ agrees with the classical measure entropy $h_\mu (X,G)$.
The proof of the topological analogue of this equality
in Section~\ref{S-topological}
provides a basis for the argument, but the measure-preserving condition requires us in addition to keep
track of statistical distributions along orbits. For this we will need a particular form of the
Shannon-McMillan theorem which asserts, for infinite $G$, the $L^1$-convergence
of the mean information functions to the entropy function (Lemma~\ref{L-SM}). In the proof of Lemma~\ref{L-SM} and elsewhere
we also require the ergodic decomposition of entropy, which relies on the affineness of the entropy function \cite{Mou85}
(see \cite[Thm.\ 8.4]{Walters} for the $\Zb$-action case) and hence requires
$G$ to be infinite. The proof of Lemma~\ref{L-lower bound} also requires $G$ to be infinite
for different reasons.
We will therefore need to handle the case of finite $G$ separately, which we do
in Lemmas~\ref{L-finite upper bound} and \ref{L-finite lower bound}.

Given that the sofic measure entropy $h_{\Sigma ,\mu} (X,G)$ essentially amounts to counting unions of partial orbits
over F{\o}lner sets in the case that $G$ is amenable, our arguments will pass through some of the ideas in
the proof of Theorem~1.1 of \cite{Kat80}, which gives a formula for the entropy of an ergodic measure-preserving transformation
in terms of orbit growth in the spirit of Rufus Bowen's definition of topological entropy.
Note however that we do not assume our actions to be ergodic.

Consider a Borel action of a countable group $G$ on a standard Borel space $(X, \cB_X)$.
We consider the $\sigma$-algebra
\[
\cB_{X, G}=\{ A\in \mathcal{B}_X: sA=A \text{ for all } s \in G\}.
\]
Denote by $\cM(X, G)$ the set of $G$-invariant probability measures on $(X, \cB_X)$ and by $\cM^\erg (X, G)$ the set of $G$-invariant
ergodic probability measures on $(X, \cB_X)$. Assume that $\cM(X, G)$ is nonempty.
Endow $\cM^\erg (X, G)$ with the smallest $\sigma$-algebra making the functions
$\mu\mapsto \mu(A)$ on $\cM^\erg (X, G)$ measurable for all $A\in \cB_X$.
Then $\cM^\erg (X, G)$ is a standard Borel space (in particular, $\cM^\erg (X, G)$ is nonempty) and there is a surjective Borel map $X\rightarrow \cM^\erg (X, G)$ sending $x$ to $\mu_x$  \cite[Thm.\ 4.2 and p.\ 204]{Var63} satisfying the following conditions:
\begin{enumerate}
\item $\mu_{sx}=\mu_x$ for all $x\in X$ and $s\in G$,

\item for each $\nu\in \cM^\erg (X, G)$ if we set $X_{\nu}=\{x\in X: \mu_x=\nu\}$ then $\nu$ is the unique $\mu$ in
$\cM(X, G)$ satisfying $\mu(X_{\nu})=1$,

\item for every $\mu\in \cM(X, G)$ and $A\in \cB_{X}$ we have $\mu(A)=\int_X \mu_x(A)\, d\mu(x)$.
\end{enumerate}
Furthermore, this map is essentially unique in the sense that if $x\mapsto \mu'_x$ is another map satisfying the above conditions
then there exists an $A\in \cB_{X, G}$ such that $\mu(A)=0$ for every $\mu\in \cM(X, G)$ and $\mu_x=\mu'_x$ for all $x\in X\setminus A$.
It follows that $\mu=\int_X \mu_x\, d\mu(x)$ is the ergodic decomposition of $\mu$ for every $\mu\in \cM(X, G)$, and that for each $\mu\in \cM(X, G)$ and each $\Cb$-valued bounded Borel function $f$ on $X$ one has
$\Eb_{\mu}(f|\cB_{X, G})(x)=\int_X f\, d\mu_x$
for $\mu$-a.e.\ $x$, where $\Eb_{\mu}(f|\cB_{X, G})$ denotes the conditional expectation of $f$ in $L^{\infty}(f, \cB_{X, G}, \mu)$.

When $G$ is an amenable countably infinite discrete group, for any finite measurable partition $\cQ$ of $X$ and any $\mu\in \cM(X, G)$,
one has $h_{\mu}(\cQ)=\int_Xh_{\mu_x}(\cQ)\, d\mu(x)$, as one can deduce from \cite[Propositions 5.3.2 and 5.3.5]{Mou85} and the proof
in the case $G=\Zb$ in \cite[Theorem 8.4.(i)]{Walters}.

For a finite measurable partition $\cQ$ of $X$ and a $\mu\in \cM(X, G)$, the information function $I_{\mu}(\cQ)$ is defined by
\[
I_{\mu}(\cQ)(x)=-\sum_{Q\in \cQ}1_Q(x)\log \mu(Q)
\]
for all $x\in X$.

\begin{lemma} \label{L-SM}
Consider a Borel action of an amenable countably infinite discrete group $G$ on a standard Borel space $(X, \cB_X)$.
Let $\cQ$ be a finite measurable
partition of $X$ and $\mu\in \cM(X, \cB_X)$. Then the functions $\frac{1}{|F|}I_{\mu}(\cQ^F)$ converge to the function $x\mapsto h_{\mu_x}(\cQ)$
in $L^1(X, \cB_X, \mu)$ as the nonempty finite set $F\subseteq G$ becomes more and more left invariant in the sense
that for every $\varepsilon > 0$ there are a nonempty finite set $K\subseteq G$ and a $\delta > 0$
such that $\frac{1}{|F|}I_{\mu}(\cQ^F)$ is within $\varepsilon$ of the function $x\mapsto h_{\mu_x}(\cQ)$ in the $L^1$-norm
whenever $|KF\Delta F| \leq \delta |F|$.
\end{lemma}

\begin{proof}
By the Shannon-McMillian theorem \cite[Thm.\ 4.4.2]{Mou85}, there exists an $f\in L^1(X, \cB_{X, G}, \mu)$
such that the function $\frac{1}{|F|}I_{\mu}(\cQ^F)$ converges to $f$
in $L^1(X, \cB_X, \mu)$ as the nonempty finite set $F\subseteq G$ becomes more and more left invariant.
%Clearly this holds also when $G$ is finite.
Set $g(x)= h_{\mu_x}(\cQ)$ for all $x\in X$. Then $g$ is a bounded $\cB_{X, G}$-measurable function on $X$. We just need to show that
$f(x)=g(x)$ for $\mu$-a.e. $x$.

We claim that $\int_Af\, d\mu\ge \int_Ag\, d\mu$ for all $A\in \cB_{X,G}$.
Let $A\in \cB_{X, G}$. We may assume that $\mu(A)>0$.
Define $\nu\in \cM(X, G)$ by $\nu(B)=\frac{1}{\mu(A)}\mu(B\cap A)$ for all $B\in \cB_X$.
Let $F$ be a nonempty finite subset of $G$.
Set $\xi(t)=-t\log t$ for $t\ge 0$. Then $\xi$ is concave on $[0, +\infty)$. Thus
\begin{align*}
H_{\nu}(\cQ^F)-\frac{1}{\mu(A)}\int_A I_{\mu}(\cQ^F)\, d\mu
&=\sum_{B \in \cQ^F}-\nu(B)\log \nu(B)-\sum_{B\in \cQ^F}-\nu(B)\log \mu(B)\\
&= \sum_{B\in \cQ^F}-\nu(B)\log \frac{\nu(B)}{\mu(B)}\\
&= \sum_{B\in \cQ^F}\mu(B)\xi \bigg(\frac{\nu(B)}{\mu(B)}\bigg)\\
&\le \xi\bigg(\sum_{B\in \cQ^F}\mu(B)\frac{\nu(B)}{\mu(B)}\bigg)\\
&=\xi(1)=0,
\end{align*}
where the inequality comes from the concavity of $\xi$. Dividing the above inequality
by $|F|$ and taking limits with $F$ becoming more and more left invariant, we get
$h_{\nu}(\cQ)-\frac{1}{\mu(A)}\int_Af\, d\mu\le 0$.
Thus
$$\int_A g\, d\mu=\mu(A)\int_X h_{\mu_x}(\cQ)\, d\nu(x)=\mu(A) h_{\nu}(\cQ)\le \int_Af\, d\mu.$$
This proves our claim. It follows that $f(x)-g(x)\ge 0$ for $\mu$-a.e.\ $x$.

For $A=X$ the argument in the above paragraph shows that $\int_Xf\, d\mu=\int_Xg\, d\mu$. Thus $f(x)-g(x)=0$ for $\mu$-a.e.\ $x$.
\end{proof}

%Given a finite set $F\subseteq G$ and a $\delta > 0$, we say that a map $\sigma : G\to\Sym (d)$ for some $d\in\Nb$
%is a {\it sofic $(F,\delta )$-approximation} if
%\begin{enumerate}
%\item $\big| \{ k\in \{ 1,\dots ,d\} : \sigma_{st} (k) = \sigma_s \sigma_t (k) \} \big| \geq (1-\delta )d$ for all $s,t\in F$, and
%
%\item $\big| \{ k\in \{ 1,\dots ,d_i \} : \sigma_s (k) \neq \sigma_t (k) \} \big| \geq (1-\delta )d$ for all distinct $s,t\in F$.
%\end{enumerate}

\begin{lemma}\label{L-fixed}
Let $\kappa > 0$. Then there are $\delta_0 > 0$, $M\in\Nb$, and $\omega : \Nb\to (0,1)$ such that if
$F$ is a finite subset of a group $G$  with $|F|\geq M$, $\delta \in (0,\delta_0 )$,
$d\in\Nb$, and $\sigma : G\to\Sym (d)$ is a map with $\big| \bigcup_{s, t\in F, s\neq t} \{ k\in \{ 1,\dots ,d \} : \sigma_s (k) = \sigma_t (k) \} \big| \leq \delta d$, then the number of subsets $A\subseteq \{ 1,\dots ,d\}$ such that
$\max_{s\in F}| A\Delta\sigma_s (A)| \le \omega (|F|)d$ is at most $\exp(\kappa d)$.
\end{lemma}

\begin{proof}
Partition $\{ 1,\dots ,d\}$ into sets $Q_1 , \dots ,Q_n$ each of which is invariant under the
subgroup $\langle \sigma (F) \rangle$ of $\Sym (d)$ generated by $\sigma (F)$ and has no nonempty
proper subset with this property. Write $I$ for the set of all $i\in \{ 1,\dots ,n \}$ such that $| Q_i |\geq |F|$
and set $I' =\{ 1,\dots ,n \}\setminus I$. Then $|I|\le d/|F|$.
For each $i\in I$ fix an element $a_i$ of $Q_i$.

Set $R=\bigcup_{i\in I} Q_i$ and $R'=\bigcup_{i\in I'}Q_i$.
For every $a\in R'$ we can find distinct $s , t \in F$ such that
$\sigma_s (a) = \sigma_t (a)$.
Since $\big| \bigcup_{s, t\in F, s\neq t} \{ k\in \{ 1,\dots ,d \} : \sigma_s (k) = \sigma_t (k) \} \big| \leq \delta d$,
it follows that $|R'| \le \delta d$.

Now let us estimate the number of sets $A\subseteq \{ 1,\dots ,d\}$ such that
$\max_{s\in F}| A\Delta \sigma_s (A)| \le \eta d$. Let $A$ be an arbitrary such set.
For each $s\in F$
define the function $\gamma_s : R \to \{ 0,1 \}$ by $\gamma_s (a) = 1$ if either (i)
$a\in A$ and $\sigma_s (a)\notin A$ or (ii) $a\notin A$ and $\sigma_s (a)\in A$, and $\gamma_s (a) = 0$ otherwise.
For each $s\in F$ define the function $\tilde{\gamma}_s : R \to \{ 0,1 \}$ by $\tilde{\gamma}_s (a) = 1$ if either (i)
$a\in A$ and $\sigma_s^{-1} (a)\notin A$ or (ii) $a\notin A$ and $\sigma_s^{-1} (a)\in A$,
and $\tilde{\gamma}_s (a) = 0$ otherwise. Also define a function $\beta: I\rightarrow \{0, 1\}$ by $\beta(i)=1$ if $a_i\in A$, and
$\beta(i)=0$ otherwise.

Let $s\in F$. Since $| A\Delta \sigma_s (A)| \le \eta d$, the number of $a\in R$ such that
$\gamma_s (a) = 1$ can be at most $\eta d$. Similarly, the number of $a\in R$ such that
$\tilde{\gamma}_s (a) = 1$ can be at most $\eta d$.

Note that the collection of functions $\{\gamma_s: s\in F\}\cup \{\tilde{\gamma}_s: s\in F\}\cup \{\beta\}$
uniquely specifies $A\cap R$, for if $i\in I$ and $a\in Q_i$ then for some $t_1 , \dots , t_k \in F$
and $e_1 , \dots ,e_k \in \{ 0,-1 \}$ the permutation $\omega = \sigma_{t_1}^{e_1} \cdots \sigma_{t_k}^{e_k}$
will send $a_i $ to $a$, which enables us to determine whether or not $a$ belongs to $A$ by using
the functions from $\{\gamma_s: s\in F\}\cup \{\tilde{\gamma}_s: s\in F\}\cup \{\beta\}$.
Thus the number of possibilities for $A\cap R$ is at most the number of possible collections
$\{\gamma_s: s\in F\}\cup \{\tilde{\gamma}_s: s\in F\}\cup \{\beta\}$ and hence is bounded above by
$\big(\sum_{k=0}^{\lfloor \eta d \rfloor} \binom{d}{k} \big)^{2|F|}2^{d/|F|}$.
By Stirling's approximation this
is bounded above by $\exp{(\beta |F|d)}2^{d/|F|}$ for some $\beta > 0$ not depending on $d$ or $|F|$
with $\beta\to 0$ as $\eta \to 0$.

For the number of possibilities for the intersection $A\cap R'$ we have the crude upper bound of
$2^{|R' |}$, which by the second paragraph is at most $2^{\delta d}$.
We deduce that the number of possibilities for $A$
is at most $\exp{(\beta |F|d)} 2^{d/|F| + \delta d}$, yielding the lemma.
\end{proof}

%In the proof of the following lemma,
%for a set $X$ with a pseudometric $\rho$, we will write
%$\spa_\varepsilon (X,\rho )$ for the minimal cardinality of a $(\rho, \varepsilon)$-spanning subset of $X$
%and $\sep_\varepsilon (X,\rho )$ for the maximal cardinality of a $(\rho, \varepsilon)$-separated subset of $X$.

\begin{lemma}\label{L-supremum upper bound}
Let $G$ be an amenable countably infinite discrete group acting continuously on a compact metrizable space $X$
and $\mu$ a $G$-invariant Borel probability measure on $X$.
%Let $\Sigma = \{ \sigma_i : G\to\Sym (d_i ) \}_{i=1}^\infty$ be a sofic approximation sequence for $G$.
%Let $\cP$ be a dynamically generating finite partition of unity in $C(X)$.
Let $\rho$ be a compatible metric on $X$.
%Define on $X$ the continuous pseudometric $\rho(x,y) = \sup_{f\in\cP} |f(x) - f(y)|$.
Let $\varepsilon>0$.
Let $\cQ$ be a finite Borel partition of $X$ with $\max_{Q\in \cQ} \diam_\rho(Q)<\varepsilon/16$.
%Let $\xi\ge 0$ such that $\rh_{\mu_x}(\cQ)\le \xi$ for $\mu$ almost all $x$. Assume that pointwise convergence holds in
%Lemma~\ref{L-Shannon-McMillian}.
Then
$h^{\varepsilon}_{\Sigma, \mu ,\infty}(\rho)\le h_{\mu}(\cQ)$.
\end{lemma}

\begin{proof}
Let $\kappa>0$. Take a finite $G$-invariant Borel partition $\cR'$ of $X$ such that
$\sup_{x\in R} h_{\mu_x}(\cQ)-\inf_{x\in R} h_{\mu_x}(\cQ)<\kappa$ for
every $R\in \cR'$, where $x\mapsto \mu_x$ is the Borel map from $X$ to $\cM^\erg (X,G)$ described at the beginning of the section.
Denote by $\cR$ the set of atoms in $\cR'$ with positive $\mu$-measure.
For each $R\in \cR$, set $\xi_R=\sup_{x\in R} h_{\mu_x}(\cQ)$. We will show that
$h^{\varepsilon}_{\Sigma, \mu ,\infty}(\rho)\le \sum_{R\in \cR}\xi_R\mu(R)+5\kappa$.
Since
\begin{align*}
h_{\mu}(\cQ)&=\int_X h_{\mu_x}(\cQ)\, d\mu(x)\ge \sum_{R\in \cR}\mu(R)\hspace*{0.3mm}\inf_{x\in R} h_{\mu_x}(\cQ) \\
&\ge \sum_{R\in \cR}\mu(R)(\xi_R-\kappa)=\sum_{R\in \cR}\xi_R\mu(R)-\kappa,
\end{align*}
this will imply $h^{\varepsilon}_{\Sigma, \mu,\infty}(\rho)\le h_{\mu}(\cQ)+6\kappa$. As $\kappa$
is an arbitrary positive number, the latter  will imply that
$h^{\varepsilon}_{\Sigma, \mu ,\infty}(\rho)\le  h_{\mu}(\cQ)$.

By Lemma~\ref{L-fixed}, there exist an $M'\in \Nb$ and  a function $\omega: \Nb\rightarrow (0, 1)$ such
that, for any finite subset $F'$ of $G$ with $|F'|\ge M'$,
whenever a map $\sigma : G\to\Sym(d)$ for some $d\in \Nb$
is a good enough sofic approximation for $G$ the number of sets
$A\subseteq \{1, \dots, d\}$ satisfying $\max_{s\in F'}|A\Delta \sigma_s(A)|/d\le \omega(|F'|)$ is at most
$\exp(|\cR|^{-1}\kappa d)$.

%Take $\lambda''>0$ with $\lambda''<1/M'$ such that for and $R\in \cR$ and any large enough finite set $\Upsilon$, the number of subsets %$\Upsilon'\subseteq \Upsilon$
%satisfying $|\Upsilon'|\ge |\Upsilon|(\mu(R)-\lambda'')/(\mu(R)+\lambda'')$ is at most $\exp(\kappa |\Upsilon|)$.

Take an $\eta>0$ such that $(N_{\varepsilon/4}(X, \rho))^{3\eta |\cR|}<\exp(\kappa)$, $\eta<2^{-1}\min_{R\in \cR}\mu(R)$,
\begin{align*}
\frac{1}{1-\eta}\bigg(\sum_{R\in \cR}\xi_R \mu(R)+\kappa+ 2\kappa |\cR|^2\eta +2|\cR|\eta\sum_{R\in \cR}\xi_R \bigg)<
\sum_{R\in \cR}\xi_R \mu(R)+2\kappa,
\end{align*}
and for every $R\in \cR$ and finite set $\Upsilon$ the number of sets $\Upsilon'\subseteq \Upsilon$
satisfying $|\Upsilon'|\ge |\Upsilon|(\mu(R)-\eta)/(\mu(R)+\eta)$ is at most $\exp(\kappa |\Upsilon|)$,
as is possible by Stirling's approximation.

By Lemma~\ref{L-Rokhlin},
there are an $\ell\in \Nb$ and an $\eta'>0$
such that, whenever $e\in F_1\subseteq F_2\subseteq \cdots \subseteq F_\ell$ are finite subsets of $G$ with
$|(F_{k-1}^{-1}F_k)/F_k|\le \eta'|F_k|$
for $k=2, \dots, \ell$, for every map $\sigma : G\to\Sym(d)$ for some $d\in \Nb$ which is a good
enough sofic approximation for $G$ and every $Y_R\subseteq \{1, \dots, d\}$ with $|Y_R|/d\ge \mu(R)-\eta$ for all $R\in \cR$,
there exist, for every $R\in \cR$, sets $C_{R, 1}, \dots, C_{R, \ell}\subseteq Y_R$ such that
\begin{enumerate}
\item for all $k=1, \dots, \ell$ and $c\in C_{R, k}$
the map $s\mapsto \sigma_s(c)$ from $F_k$ to $\sigma(F_k)c$ is bijective,

\item the family $\bigcup_{k=1}^\ell\{\sigma(F_k)c: c\in C_{R, k}\}$ is $\eta$-disjoint
and $(\mu(R)-2\eta)$-covers $\{1, \dots, d\}$.
\end{enumerate}

%Let $0<\tau<1/6$ be a small positive number which we will determine in a moment.
Take $0<\tau<\eta/4$.
Let $R\in \cR$. Note that $\mu(\cdot \cap R)/\mu(R)$ is a Borel probability measure on $X$, which we denote by $\mu_R$.
One has $ h_{\mu_x}(\cQ)\le \xi_R$ for $\mu_R$-almost every $x$.
By Lemma~\ref{L-SM}, there exist a nonempty finite subset $K_R$ of $G$ and a $\delta_R>0$
such that for every nonempty finite subset $F'$ of $G$ satisfying $|K_RF'\setminus F'|<\delta_R |F'|$ there exists
an $\cA_{R, F'}\subseteq \cQ^{F'}$ such that $\mu_R(\bigcup \cA_{R, F'})>1-\tau/\ell$,
%$\mu_R( A)>0$ for
%every $A\in \cA_{R, F'}$,
and for every $A\in \cA_{R, F'}$  we have $\mu_R(A)>0$ and $-|F'|^{-1}\log \mu_R(A)\le \xi_R+\kappa$, that is
\begin{align} \label{E-measure upper}
\mu_R(A)\ge \exp(-(\xi_R+\kappa)|F'|).
\end{align}
For each $A\in\cA_{R, F'}$ pick a point $x_A \in A\cap R$ and set $E_{R, F'} = \{ x_A : A\in \cA_{R, F'} \}$. Since $\max_{Q\in \cQ} \diam_\rho(Q)<\varepsilon/16$, the set
$E_{R, F'}$ is an $(\varepsilon/16)$-spanning subset of $\bigcup\cA_{R, F'}$ with respect to $\rho_{F'}$.

Now we fix finite subsets $F_1, \dots, F_\ell$ of $G$ such that
$e\in F_1\subseteq F_2\subseteq \cdots \subseteq F_\ell$,  $|F_\ell|\ge M'$, $|(F_{k-1}^{-1}F_k)/F_k|\le \eta'|F_k|$
for $k=2, \dots, \ell$, and $|K_RF_k\setminus F_k|<\delta_R|F_k|$ for every $R\in \cR$ and $k=1, \dots, \ell$.
Then we have $\cA_{R, F_k}$ and $E_{R, F_k}$ for every $R\in \cR$ and $k=1, \dots, \ell$.

Let $\lambda>0$ be a small number to be determined in a moment.
%Let $R\in \cR$. By the regularity of $\mu$, we can find an open neighborhood $U'_R$ of $R$ such
%that $\mu(U'_R\setminus R)<\lambda$. Set $U_R=U'_R\cap \bigcap_{s\in \bigcup_{k=1}^\ell F_k}s^{-1}U'_R$. Since $R$ is
%$G$-invariant, $U_R$ is an open neighborhood of $R$ with $sU_R\subseteq U'_R$ for every $s\in \{e\}\cup \bigcup_{k=1}^\ell F_k$.
%Set $Z_R=X\setminus (\bigcup_{R\neq R'\in \cR} U_{R'})$. Then $Z_R$ is a closed subset of $U_R$ and $\mu(Z_R)\ge \mu(R)-|\cR|\lambda$.
%
%The family $\{Y_R\}_{R\in \cR}$ is pairwise disjoint. Thus we may find some $g_R\in C(X)$ for each $R\in \cR$ such that $0\le g_R\le 1$, $g_R=1$ on %$Y_R$,
%$g_R=0$ outside $U_R$, and $g_R\cdot g_{R'}=0$ for any distinct $R, R'$ in $\cR$.
Let $R\in \cR$. Then $\mu_R \big(\bigcap_{k=1}^\ell\bigcup \cA_{R, F_k}\big)>1-\tau$.
By the regularity of $\mu_R$ \cite[Thm.\ 6.1]{Walters}, we can find a closed subset $Z_R$ of
$R\cap \bigcap_{k=1}^\ell \bigcup \cA_{R, F_k}$ such that $\mu_R(Z_R)>1-\tau-\lambda$
and a closed subset $Z'_R$ of $R$ such that $Z_R\subseteq Z'_R$ and $\mu_R(Z'_R)>1-\lambda$.
Then $F_\ell Z'_R$ is a closed subset of the $G$-invariant set $R$.

Since the closed sets $F_\ell Z'_R$ for $R\in \cR$ are pairwise disjoint, we can find an open neighborhood
$U_R$ of $F_\ell Z'_R$ for every $R\in \cR$ such that the sets $U_R$ for $R\in \cR$ are pairwise disjoint.
%$\overline{U_R}\cap (\overline{U_{R'}}\cup Z'_{R'})=\emptyset$ for all distinct $R, R'$ in
%$\cR$.

Let $R\in \cR$.
By the continuity of the action of $G$ on $X$,
we can find open neighborhoods $B_R$ and $B'_R$ of $Z_R$ and $Z'_R$ respectively, such that
$B_R\subseteq B'_R$, $ F_\ell B'_R\subseteq U_R$ and $E_{R, F_k}$ is a $(\rho_{F_k}, \varepsilon/8)$-spanning subset
of $B_R\cup \bigcup \cA_{R, F_k}$ for every $k=1, \dots, \ell$.
For each $k=1, \dots, \ell$ we have
\begin{align} \label{E-measure upper2}
N_{\varepsilon/4}(B_R, \rho_{F_k} )&\le N_{\varepsilon/4}\Big(B_R\cup \bigcup \cA_{R, F_k}, \rho_{F_k} \Big)
\le |E_{R, F_k}| = | \cA_{R, F_k} | \\
\nonumber &\leq \mu_R \Big(\bigcup\cA_{R, F_k}\Big) /\exp(-(\xi_R+\kappa )|F_k |)\\
\nonumber &\leq \exp((\xi_R+\kappa )|F_k |).
\end{align}
Take an $h_R\in C(X)$ such that $0\le h_R\le 1$, $h_R=1$ on $Z'_R$, and $h_R=0$ outside of $B'_R$.
Also, take a $g_R\in C(X)$ such that $0\le g_R\le 1$, $g_R=1$ on $Z_R$, and $g_R=0$ outside of $B_R$.
Replacing $g_R$ by $\min(g_R, h_R)$ if necessary, we may assume that $g_R\le h_R$.

Set $L=\bigcup_{R\in\cR} \{ h_R , g_R \}$.
%Since $\cP$ dynamically generates $C(X)$ as a $C^*$-algebra,
%we can find a finite subset $F$ of $G$ containing $F_\ell \cup F_\ell^{-1}$ and
%an $m\in\Nb$ such that, for every $f\in \bigcup_{R\in\cR} \{ h_R , g_R \}$, there exists an
%$\tilde{f}\in \spn (\cP_{F, m} )$ with $\|\tilde{f}-f\|_{\infty}<\lambda$.
%Denote by $M$ the maximum over all $f\in \bigcup_{R\in\cR} \{ h_R , g_R \}$ of the sum of the coefficients for
%$\tilde{f}$ written as a linear combination of elements in $\cP_{F, m}$.
Let $\delta>0$ be a small number which we will determine in a moment.
Let $\sigma$ be a map from $G$ to $\Sym(d)$ for some $d\in \Nb$ which is a good enough
sofic approximation for $G$. We will show that
$N_{\varepsilon}(\Map_\mu(\rho, F_\ell, L, \delta, \sigma), \rho_\infty)\le  \exp((\sum_{R\in \cR}\xi_R\mu(R)+5\kappa)d)$, which
will complete the proof since we can then conclude that
$h_{\Sigma, \mu ,\infty}^\varepsilon(\rho, F_\ell, L, \delta)\le \sum_{R\in \cR}\xi_R\mu(R)+5\kappa$
and hence $h^{\varepsilon}_{\Sigma, \mu ,\infty}(\rho)\le \sum_{R\in \cR}\xi_R\mu(R)+5\kappa$.

Denote by $\Lambda$ the set of all $a\in \{1, \dots, d\}$ satisfying
%$\sigma_{s^{-1}}\sigma_s(a)=a$ for all $s\in F_\ell$ and
$\sigma_e(a)=a$.
Let $\varphi\in \Map_\mu(\rho, F_\ell, L, \delta, \sigma)$.
Denote by $\Lambda_{\varphi}$ the set of all $a\in\{ 1,\dots ,d \}$ such that
\begin{enumerate}
\item $|h_R(\varphi(a))- h_R(s^{-1}\varphi(s a))| <1/2$
for all $R\in\cR$ and $s\in F_\ell$, and

\item $\rho(\varphi(s a), s\varphi(a)) < \sqrt{\delta}$ for all $s\in F_\ell$.
\end{enumerate}
Set
\begin{align*}
\Omega'_{R, \varphi} &=\{a\in \{1, \dots, d\}: h_R(\varphi(a))>0\}, \\
\Omega''_{R, \varphi} &=\{a\in \{1, \dots, d\}: h_R(\varphi(a))>1/2\}, \\
\Omega_{R, \varphi}&=\Omega''_{R, \varphi}\cap \Lambda_{\varphi}\cap \Lambda,
\end{align*}
and
\begin{align*}
\Theta'_{R, \varphi} &=\{a\in \{1, \dots, d\}: g_R(\varphi(a))>0\}, \\
\Theta''_{R, \varphi} &=\{a\in \{1, \dots, d\}: g_R(\varphi(a))>1/2\}, \\
\Theta_{R, \varphi}&=\Theta''_{R, \varphi}\cap \Lambda_{\varphi}\cap \Lambda.
\end{align*}

{\bf Claim I:} Assuming $\lambda, \delta$ are small enough
and $\sigma$ is a good enough sofic approximation for $G$,
for every
$\varphi\in \Map_\mu(\rho, F_\ell, L, \delta, \sigma)$ we have that
$|\Omega_{R, \varphi}|/d\le \mu(R)+\eta$ for every $R\in \cR$, the sets $\sigma(F_\ell)\Omega_{R, \varphi}$ for $R\in \cR$ are pairwise disjoint, and
\[
\frac{1}{d}\max_{s\in F_{\ell}}\bigg|\Omega_{R, \varphi}\Delta \sigma_s (\Omega_{R, \varphi})\bigg| \le \omega(|F_\ell|).
\]

To verify Claim I,
note first that if $\sigma$ is a good enough sofic approximation for $G$ then $|\Lambda|/d\ge 1-\lambda$.
Consider the continuous pseudometric $\rho'$ on $X$ defined by
\[
\rho'(x, y)=\max_{s\in F_\ell}\max_{R\in \cR}|h_R(s^{-1}x)-h_R(s^{-1}y)|.
\]
When $\delta$ is small enough, for any $x, y\in X$ with $\rho(x, y)< \sqrt{\delta}$, one has $\rho'(x, y)< 1/2$.
It follows that for any $a\in \{1, \dots, d\}$  and $s\in F_\ell$ with $\rho(\varphi(sa), s\varphi(a))<\sqrt{\delta}$, one
has $|h_R(\varphi(a))- h_R(s^{-1}\varphi(s a))|< 1/2$
for all $R\in\cR$.
Since $\varphi \in \Map_\mu(\rho, F_\ell, L, \delta, \sigma)$, for each $s\in F_\ell$ one has $\rho_2(\alpha_s\circ \varphi, \varphi\circ \sigma_s)<\delta$ and hence
\[
\big|\big\{a\in \{1, \dots, d\}: \rho(\varphi(sa), s\varphi(a))<\sqrt{\delta}\big\}\big|\ge (1-\delta) d.
\]
Therefore $|\Lambda_{\varphi}|/d\ge 1-|F_\ell|\delta$.

Now let $R\in \cR$, $a\in \Omega_{R, \varphi}$ and $s\in F_\ell$.
%The unital homomorphism $f\mapsto \varphi(f)(\sigma_s(a))$ on $C(X)$ is given by evaluation
%at some point $x_{s, a}\in X$.
%Since $a\in \Lambda_{\varphi}$, we have
%\begin{align*}
%|\varphi(h_R)(a)-(\varphi\circ \alpha_s(h_R))(\sigma_s(a))|
%&= |(\sigma_s \circ\varphi (h_R) - \varphi\circ\alpha_s (h_R))(\sigma_s(a))|
%\leq \sqrt{3Mm\delta+2\lambda},
%\end{align*}
%and hence
%\begin{align*}
%h_R(s^{-1}x_{s, a})&=
%\alpha_s(h_R)(x_{s, a})= \varphi(\alpha_s(h_R))(\sigma_s(a))\\
%&\ge \varphi(h_R)(a)-|\varphi(h_R)(a)-(\varphi\circ \alpha_s(h_R))(\sigma_s(a))|\\
%&\ge \frac{1}{2}-\sqrt{3Mm\delta+2\lambda}>0
%\end{align*}
%when $\delta$ and $\lambda$ are small enough.
One has
$$ h_R(s^{-1}\varphi(sa))\ge h_R(\varphi(a))-|h_R(\varphi(a))-h_R(s^{-1}\varphi(sa))|>1/2-1/2=0.$$
Therefore $s^{-1}\varphi(s a)\in B'_R$, and hence $\varphi(s a)\in F_{\ell}B'_R\subseteq U_R$. Since the sets $U_R$ for $R\in \cR$ are pairwise
disjoint, the sets $\sigma(F_\ell)\Omega_{R, \varphi}$ for $R\in \cR$ are pairwise disjoint.

Let $R\in \cR$.
We have
\begin{align} \label{E-Claim I.1}
(\varphi_*\zeta)(h_R)\ge \mu(h_R)-\delta\ge \mu(Z'_R)-\delta\ge \mu(R)(1-\lambda)-\delta.
%\zeta\circ \varphi(h_R)&\ge \zeta\circ \varphi(\widetilde{h_R})-|\zeta\circ \varphi(\widetilde{h_R})-\zeta\circ \varphi(h_R)|\\
%\nonumber &\ge \mu(\widetilde{h_R})-|\mu(\widetilde{h_R})-\zeta\circ \varphi(\widetilde{h_R})|-\|\widetilde{h_R}-h_R\|_{\infty}\\
%\nonumber &\ge \mu(h_R)-|\mu(h_R)-\mu(\widetilde{h_R})|-M\delta-\|\widetilde{h_R}-h_R\|_{\infty}\\
%\nonumber &\ge \mu(Z'_R)-\|\widetilde{h_R}-h_R\|_{\infty}-M\delta-\|\widetilde{h_R}-h_R\|_{\infty}\\
%\nonumber &\ge \mu(R)(1-\lambda)-2\lambda-M\delta.
\end{align}
Since $h_R\le 1$,
we have
\begin{align*}
\frac{|\Omega'_{R, \varphi}|}{d}\ge (\varphi_*\zeta)(h_R)
\ge \mu(R)(1-\lambda)-\delta.
\end{align*}
Since $h_Rh_{R'}=0$ for all distinct $R, R'\in \cR$,
the sets $\{\Omega'_{R, \varphi}\}_{R\in \cR}$ are pairwise disjoint.
Therefore
\begin{align} \label{E-Claim I.2}
\frac{|\Omega'_{R, \varphi}|}{d} &\le 1-\sum_{R'\in \cR\setminus \{R\}}\frac{|\Omega'_{R', \varphi}|}{d}\\
\nonumber &\le  1-\sum_{R'\in \cR\setminus \{R\}}(\mu(R')(1-\lambda)-\delta)\\
\nonumber &\le \mu(R)(1-\lambda)+\lambda+|\cR|\delta,
\end{align}
and hence
\begin{align*}
\frac{|\Omega_{R, \varphi}|}{d}\le \frac{|\Omega'_{R, \varphi}|}{d} &\le  \mu(R)(1-\lambda)+\lambda+|\cR|\delta\\
&\le \mu(R)+\eta
\end{align*}
when $\lambda, \delta$ are small enough.
We have
\begin{align*}
(\varphi_* \zeta )(h_R)&\le \frac{|\Omega''_{R, \varphi}|}{d}+\frac{|\Omega'_{R,  \varphi}\setminus \Omega''_{R, \varphi}|}{2d}
=\frac{|\Omega'_{R, \varphi}|}{2d}+\frac{|\Omega''_{R,\varphi}|}{2d}\\
&\le \frac{\mu(R)(1-\lambda)+\lambda+|\cR|\delta}{2}+\frac{|\Omega''_{R, \varphi}|}{2d}.
\end{align*}
Thus using \eqref{E-Claim I.1} we get
\begin{align} \label{E-Claim I.3}
\frac{|\Omega''_{R, \varphi}|}{d}\ge \mu(R)(1-\lambda)-\lambda-(2+|\cR|)\delta,
\end{align}
and hence
\begin{align*}
\frac{|\Omega_{R, \varphi}|}{d}&\ge
\frac{|\Omega''_{R, \varphi}|}{d}-\bigg(1-\frac{|\Lambda_{\varphi}|}{d}\bigg)-\bigg(1-\frac{|\Lambda|}{d}\bigg)\\
&\ge \mu(R)(1-\lambda)-2\lambda-(2+|\cR|+|F_\ell|)\delta.
\end{align*}

Since the sets $\sigma(F_\ell)\Omega_{R, \varphi}$ for $R\in \cR$ are pairwise disjoint, for every $R\in \cR$ we have
\begin{align*}
\frac{|\sigma(F_\ell)\Omega_{R, \varphi}|}{d} - \frac{|\Omega_{R, \varphi}|}{d}
&\le 1-\sum_{R'\in \cR} \frac{|\Omega_{R', \varphi}|}{d}\\
&\le (1+2|\cR|)\lambda+|\cR|(2+|\cR|+|F_\ell|)\delta,
\end{align*}
and hence, using the fact that
$\sigma(F_{\ell})\Omega_{R, \varphi}\supseteq \sigma_e\Omega_{R, \varphi}=\Omega_{R, \varphi}$,
\begin{align*}
\max_{s\in F_{\ell}}\frac{|\Omega_{R, \varphi}\Delta \sigma_s\Omega_{R, \varphi}|}{d}
&\le
%2\max_{s\in F_\ell}\bigg(\frac{|\sigma(F_\ell)\Omega_{R, \varphi}|}{d}-\frac{| \Omega_{R, \varphi}|}{d}\bigg)=
2\bigg(\frac{|\sigma(F_\ell)\Omega_{R, \varphi}|}{d}-\frac{| \Omega_{R, \varphi}|}{d}\bigg)\\
&\le 2(1+2|\cR|)\lambda+2|\cR|(2+|\cR|+|F_\ell|)\delta\\
&\le \omega(|F_\ell|),
\end{align*}
when $\lambda, \delta$ are small enough. This proves Claim I.

{\bf Claim II:}
Assuming $\lambda, \delta$ are small enough and $\sigma$ is a good enough sofic approximation for $G$,
for every $\varphi\in \Map_\mu(\rho, F_\ell, L, \delta, \sigma)$ and $R\in \cR$
we have $\Theta_{R, \varphi}\subseteq \Omega_{R, \varphi}$, $|\Theta_{R,\varphi}|/d\ge \mu(R)-\eta$,
and for every $a\in \Theta_{R,\varphi}$ one has $\varphi(a)\in B_R$.

To verify Claim II, first observe that, for every $R\in \cR$, since $h_R\ge g_R$ we have
$\Theta'_{R, \varphi}\subseteq \Omega'_{R, \varphi}$ and $\Theta_{R, \varphi}\subseteq \Omega_{R, \varphi}$.
Also, for $R\in \cR$ and $a\in \Theta_{R, \varphi}$,
 since $g_R(\varphi(a))>0$ we see that $\varphi(a)$ lies in $B_R$.

Now let $R\in \cR$. One has
\begin{align*}
(\varphi_*\zeta)(g_R)\ge \mu(g_R)-\delta\ge \mu(Z_R)-\delta\ge \mu(R)(1-\tau-\lambda)-\delta.
\end{align*}
We also have
\[
\frac{|\Theta'_{R, \varphi}|}{d}\le \frac{|\Omega'_{R, \varphi}|}{d}\overset{\eqref{E-Claim I.2}}\le \mu(R)(1-\lambda)+\lambda+|\cR|\delta.
\]
Similarly to \eqref{E-Claim I.3}, we get
\begin{align*}
\frac{|\Theta''_{R, \varphi}|}{d}\ge \mu(R)(1-2\tau-\lambda)-\lambda-(2+|\cR|)\delta,
\end{align*}
and hence, using $\tau<\eta/4$,
\begin{align*}
\frac{|\Theta_{R, \varphi}|}{d}&\ge \frac{|\Theta''_{R, \varphi}|}{d}-\bigg( 1-\frac{|\Lambda_{\varphi}|}{d}\bigg)-\bigg( 1-\frac{|\Lambda|}{d}\bigg)\\
&\ge \mu(R)(1-2\tau-\lambda)-2\lambda-(2+|\cR|+|F_\ell|)\delta \\
&\ge \mu(R)-\eta,
\end{align*}
when $\lambda, \delta$ are small enough and $\sigma$ is a good enough sofic approximation for $G$. This proves Claim II.

For each $J\subseteq \{1, \dots, d\}$ we
%write $P_J$ for the canonical projection of $\Cb^d$
%onto $\Cb^J$ and define on the set of unital homomorphisms from $C(X)$ to $\Cb^d$ the pseudometric
define on the set of maps from $\{1, \dots, d\}$ to $X$ the pseudometric
\[
\rho_{J,\infty} (\varphi , \psi ) = \rho_{\infty} (\varphi|_J,  \psi|_J).
\]
%When $J = \{ 1,\dots ,d\}$ we simply write $\rho_{\cM ,\infty}$.

Take a $(\rho_\infty ,\varepsilon )$-separated subset $D$ of $\Map_\mu (\rho ,F_\ell , L, \delta , \sigma )$ of maximal cardinality.
%Then $D$ is also $\varepsilon$-separated with respect to $\rho_{\cP ,\infty}$.

By Claim I and the second paragraph of the proof of this lemma, when $\sigma$ is a good enough sofic approximation for $G$
there is a subset $D'$ of $D$ with $\exp(\kappa d)|D'|\ge |D|$ such that,
for each $R\in \cR$, the set $\Omega_{R, \varphi}$ is the same, say $\Omega_R$, for every $\varphi\in D'$.
By Claims I and II and the third paragraph of the proof, when $\sigma$ is a good enough sofic approximation for $G$, there is a subset $W$ of $D'$ with
$|W|\exp(\kappa d)\ge |W|\prod_{R\in \cR}\exp(\kappa|\Omega_R|)\ge |D'|$
such that, for each $R\in \cR$, the set $\Theta_{R, \varphi}$ is the same, say $\Theta_R$, for every $\varphi\in W$.

Now take $Y_R$ to be $\Theta_R$ in our application of Lemma~\ref{L-Rokhlin} in the fourth paragraph of the proof
in order to obtain sets $C_{R, 1}, \dots, C_{R, \ell} \subseteq \Theta_R$ satisfying the two conditions listed there.
Denote by $\sL_R$ the set of all pairs $(k,c)$ such that $k\in \{ 1,\dots ,\ell\}$ and $c\in C_{R, k}$. By $\eta$-disjointness,
for every $(k, c)\in \sL_R$ we can find an $F_{k, c}\subseteq F_k$ with
$|F_{k, c}|\ge (1-\eta)|F_k|$ such that the sets $\sigma(F_{k, c})c$ for $(k, c)\in \sL_R$ are pairwise disjoint.

Let $(k, c)\in \sL_R$. Take an $(\varepsilon/2)$-spanning subset $V_{k,c}$ of $W$ with respect to
$\rho_{\sigma(F_{k, c}) c,\infty}$ of minimal cardinality. We record here the relation between these sets:
$$ V_{k, c}\subseteq W\subseteq D'\subseteq D\subseteq \Map_\mu (\rho ,F_\ell , L, \delta , \sigma ).$$

{\bf Claim III:} Assuming $\delta$ is small enough, one has
\[
|V_{k, c}|\le \exp((\xi_R+\kappa)|F_k|).
\]

To verify Claim III,
let $V$ be an $(\varepsilon/2)$-separated subset of $W$ with respect to $\rho_{\sigma(F_{k, c})c ,\infty}$.
For each $\varphi\in V$,
%the unital homomorphism $f\mapsto \varphi (f)(c)$ on $C(X)$ is
%given by evaluation at some point $x_\varphi$ in $X$, and
since $c\in C_{R, k}\subseteq \Theta_R$
the point $\varphi(c)$ lies in $B_R$.
Let $\varphi$ and $\psi$ be distinct elements of $V$. Then for every  $s\in F_{k, c}$, since $c\in \Lambda_{\varphi}\cap \Lambda_{\psi}$ we have
\begin{align*}
\rho(s\varphi(c), s\psi(c))&\ge \rho(\varphi(sc), \psi(sc))-\rho(s\varphi(c), \varphi(sc))-\rho(s\psi(c), \psi(sc))\\
&\ge  \rho(\varphi(sc), \psi(sc))-2\sqrt{\delta},
%|\varphi\circ\alpha_{s^{-1}} (f)(c) - \psi\circ\alpha_{s^{-1}} (f)(c)|
%&\geq | \sigma_{s^{-1}} \circ \varphi (f)(c) - \sigma_{s^{-1}} \circ \psi (f)(c) | \\
%&\hspace*{10mm} \ - |\psi\circ\alpha_{s^{-1}} (f)(c) - \sigma_{s^{-1}} \circ \psi (f)(c) | \\
%&\hspace*{10mm} \ - | \sigma_{s^{-1}} \circ\varphi (f)(c) - \varphi\circ\alpha_{s^{-1}} (f)(c) | \\
%&\geq |\varphi(f)(\sigma^{-1}_{s^{-1}}(c))-\psi(f)(\sigma^{-1}_{s^{-1}}(c))|-2\sqrt{\delta}\\
%&= | \varphi (f)(\sigma_s (c)) - \psi (f)(\sigma_s (c)) | - 2\sqrt{\delta},
\end{align*}
and hence
\begin{align*}
\rho_{F_{k, c}}(\varphi(c), \psi(c))&=\max_{s\in F_{k, c}}\rho(s\varphi(c), s\psi(c))\ge \max_{s\in F_{k, c}}\rho(\varphi(sc), \psi(sc))-2\sqrt{\delta}>\varepsilon/2-\varepsilon/4=\varepsilon/4,
%\sup_{s\in F_{k, c}}\rho (sx_\varphi , sx_\psi )
%&= \sup_{s\in F_{k, c}} \sup_{f\in\cM} |f(sx_\varphi ) - f(sx_\psi )| \\
%&= \sup_{s\in F_{k, c}} \sup_{f\in\cM} |\varphi\circ\alpha_{s^{-1}} (f)(c) - \psi\circ\alpha_{s^{-1}} (f)(c)| \\
%&\geq \sup_{s\in F_{k, c}} \sup_{f\in\cM} | \varphi (f)(\sigma_s (c)) - \psi (f)(\sigma_s (c)) | - 2\sqrt{\delta} \\
%&= \rho_{\cM, \sigma(F_{k, c})c, \infty}(\varphi, \psi)-2\sqrt{\delta}\\
%&> \varepsilon/2 - \varepsilon/4 = \varepsilon/4,
\end{align*}
%\begin{align*}
%|\varphi\circ\alpha_{s^{-1}} (f)(c) - \psi\circ\alpha_{s^{-1}} (f)(c)|
%&\geq | \sigma_{s^{-1}} \circ \varphi (f)(c) - \sigma_{s^{-1}} \circ \psi (f)(c) | \\
%&\hspace*{10mm} \ - |\psi\circ\alpha_{s^{-1}} (f)(c) - \sigma_{s^{-1}} \circ \psi (f)(c) | \\
%&\hspace*{10mm} \ - | \sigma_{s^{-1}} \circ\varphi (f)(c) - \varphi\circ\alpha_{s^{-1}} (f)(c) | \\
%&\geq |\varphi(f)(\sigma^{-1}_{s^{-1}}(c))-\psi(f)(\sigma^{-1}_{s^{-1}}(c))|-2\sqrt{\delta}\\
%&= | \varphi (f)(\sigma_s (c)) - \psi (f)(\sigma_s (c)) | - 2\sqrt{\delta},
%\end{align*}
%and hence
%\begin{align*}
%\sup_{s\in F_{k, c}}\rho (sx_\varphi , sx_\psi )
%&= \sup_{s\in F_{k, c}} \sup_{f\in\cP} |f(sx_\varphi ) - f(sx_\psi )| \\
%&= \sup_{s\in F_{k, c}} \sup_{f\in\cP} |\varphi\circ\alpha_{s^{-1}} (f)(c) - \psi\circ\alpha_{s^{-1}} (f)(c)| \\
%&\geq \sup_{s\in F_{k, c}} \sup_{f\in\cP} | \varphi (f)(\sigma_s (c)) - \psi (f)(\sigma_s (c)) | - 2\sqrt{\delta} \\
%&= \rho_{\cP, \sigma(F_{k, c})c, \infty}(\varphi, \psi)-2\sqrt{\delta}\\
%&> \varepsilon/2 - \varepsilon/4 = \varepsilon/4,
%\end{align*}
granted that $\delta$ is taken small enough.
Thus $\{\varphi(c) : \varphi\in V \}$ is a $(\rho_{F_{k, c}}, \varepsilon/4)$-separated subset of $B_R$ of cardinality $|V|$,
so that
\begin{align*}
|V|\le N_{\varepsilon/4}(B_R, \rho_{F_{k, c}})\le N_{\varepsilon/4}(B_R, \rho_{F_k})\overset{\eqref{E-measure upper2}}\le \exp((\xi_R+\kappa)|F_k|).
\end{align*}
Therefore
\[|V_{k, c}|
%=\spa_{\varepsilon/2}(W,  \rho_{\cP ,\sigma(F_{k, c})c ,\infty})
\le N_{\varepsilon/2}(W, \rho_{\sigma(F_{k, c})c ,\infty})\le \exp((\xi_R+\kappa)|F_k|).\]
This proves Claim III.

Set
\[
H = \{ 1,\dots ,d\} \setminus \bigcup_{R\in \cR}\bigcup \{ \sigma (F_{k, c})c : (k,c)\in \sL_R\},
\]
and take an
$(\varepsilon/2)$-spanning subset $V_H$ of $W$ with respect to $\rho_{H,\infty}$ of minimal cardinality.

{\bf Claim IV:}
\[
|V_H | \leq  (N_{\varepsilon/4}(X, \rho))^{3|\cR|\eta d} .
\]

To verify Claim IV, first note that
for each $R\in \cR$ we have
\[
\bigg|\bigcup \{ \sigma (F_{k, c})c : (k,c)\in \sL_R\}\bigg|
\ge (1-\eta)\bigg|\bigcup_{k=1}^\ell \sigma (F_k)C_{R, k}\bigg|\ge (1-\eta)(\mu(R)-2\eta)d.
\]
Since the sets $\bigcup_{k=1}^\ell \sigma (F_k)C_{R, k}$ for $R\in \cR$ are pairwise disjoint, we get
\[
\bigg|\bigcup_{R\in \cR}\bigcup \{ \sigma (F_{k, c})c : (k,c)\in \sL_R\}\bigg|\ge \sum_{R\in \cR}(1-\eta)(\mu(R)-2\eta)d= (1-\eta)(1-2|\cR|\eta)d.
\]
Therefore
\[
|H| \leq  (\eta+2(1-\eta)|\cR|\eta)d\le (1+2|\cR|)\eta d\le 3|\cR|\eta d,
\]
and thus
%, since the image of every element of $\cP$ under a unital homomorphism from $C(X)$ into $\Cb^H$
%is contained in $[0,1]^H$, we have
\[
|V_H | \leq (N_{\varepsilon/4}(X, \rho))^{|H|} \leq (N_{\varepsilon/4}(X, \rho))^{3|\cR|\eta d}.
\]
This proves Claim IV.

Write $U$ for the set of all maps
$\varphi : \{1, \dots, d\}\to X$ such that $\varphi|_H\in V_H|_H$ and
$\varphi|_{\sigma(F_{k, c})c}\in V_{k,c}|_{\sigma(F_{k, c})c}$
for all $R\in \cR$ and $(k,c)\in\sL_R$.

{\bf Claim V:}
\[
|U|\le \exp\bigg(\bigg(\sum_{R\in \cR}\xi_R \mu(R)+3\kappa\bigg) d\bigg).
\]

To verify Claim V, observe that,
since the sets $\bigcup_{k=1}^\ell \sigma(F_k)C_{R, k}$ for $R\in \cR$ are pairwise disjoint, for each $R\in \cR$ we have
\begin{align*}
\bigg|\bigcup_{k=1}^\ell \sigma(F_k)C_{R, k}\bigg|&\le d
-\sum_{R'\in \cR \setminus \{ R \}}\bigg|\bigcup_{k=1}^\ell \sigma(F_k)C_{R', k}\bigg|\\
&\le d-\sum_{R'\in \cR \setminus \{ R \}}(\mu(R')-2\eta)d\\
&\le \mu(R)d+2|\cR|\eta d.
\end{align*}
Therefore, by our choice of $\eta$ in the third paragraph of the proof,
\begin{align*}
|U | &=|V_H | \prod_{R\in \cR}\prod_{(k,c)\in\sL_R} |V_{k,c} |
\le (N_{\varepsilon/4}(X, \rho))^{3|\cR|\eta d} \exp\bigg(\sum_{R\in \cR}\sum_{(k, c)\in \sL_R}(\xi_R+\kappa)|F_k|\bigg)\\
&= (N_{\varepsilon/4}(X, \rho))^{3|\cR|\eta d} \exp\bigg(\sum_{R\in \cR}(\xi_R+\kappa)\sum_{k=1}^\ell|F_k|\cdot |C_{R, k}|\bigg)\\
&\le \exp(\kappa d) \exp\bigg(\frac{1}{1-\eta}\sum_{R\in \cR}(\xi_R+\kappa)\bigg|\bigcup_{k=1}^\ell \sigma(F_k)C_{R, k}\bigg|\bigg)\\
&\le \exp(\kappa d) \exp\bigg(\frac{1}{1-\eta}\sum_{R\in \cR}(\xi_R+\kappa)(\mu(R)d+2|\cR|\eta d)\bigg)\\
&\le \exp(\kappa d) \exp\bigg(\frac{1}{1-\eta}\bigg(\sum_{R\in \cR}\xi_R \mu(R)+\kappa+2\kappa |\cR|^2\eta +2|\cR|\eta\sum_{R\in \cR}\xi_R\bigg)
d\bigg)\\
&\le \exp(\kappa d)\exp\bigg(\bigg(\sum_{R\in \cR}\xi_R \mu(R)+2\kappa\bigg)d\bigg)
=\exp\bigg(\bigg(\sum_{R\in \cR}\xi_R \mu(R)+3\kappa)d\bigg).
\end{align*}
This proves Claim V.

Note that every element of $W$ lies within $\rho_{\infty}$-distance $\varepsilon/2$ to an element
of $U$, and since $W$ is $\varepsilon$-separated with respect to $\rho_{\infty}$ this means that
the cardinality of $W$ is at most that of $U$.
Therefore
\begin{align*}
N_{\varepsilon}(\Map_\mu(\rho, F_\ell, L, \delta, \sigma), \rho_\infty)&=
|D|\le \exp(\kappa d)|D'|\le \exp(2\kappa d)|W|\le  \exp(2\kappa d)|U|\\
&\le \exp(2\kappa  d)\exp\bigg(\bigg(\sum_{R\in \cR}\xi_R \mu(R)+3\kappa\bigg) d\bigg)\\
&=\exp\bigg(\bigg(\sum_{R\in \cR}\xi_R \mu(R)+5\kappa\bigg)d\bigg) ,
\end{align*}
as we aimed to show.
\end{proof}

\begin{lemma}\label{L-lower bound}
Let $G$ be an amenable countably infinite discrete group acting continuously on a compact metrizable space $X$
and $\mu$ a $G$-invariant Borel probability measure on $X$.
%Let $\Sigma = \{ \sigma_i : G\to\Sym (d_i ) \}_{i=1}^\infty$ be a sofic approximation sequence for $G$.
Let $\rho$ be a compatible metric on $X$. Then
\[
h_{\Sigma, \mu ,\infty}(\rho)\ge h_{\mu}(X, G).
\]
\end{lemma}

\begin{proof}
By \cite[Lemmas 5.3.6 and 5.3.4]{Mou85}, the entropy $h_{\mu}(X, G)$ is equal to the supremum of $h_{\mu}(\cQ)$ for $\cQ$ ranging over finite Borel partitions
of $X$ with $\max_{Q\in \cQ}\mu(\partial Q)=0$, where $\partial Q$ denotes the boundary of $Q$.
Thus it suffices to show that $h_{\Sigma, \mu ,\infty}(\rho)\ge h_{\mu}(\cQ)$ for every such $\cQ$.
Let $\cQ$ be such a partition.

Since $h_{\mu}(\cQ)=\int_X h_{\mu_x}(\cQ)\, d\mu(x)$ and the function $x\mapsto h_{\mu_x}(\cQ)$ is $X_{\cB, G}$-measurable,
where $\cB_{X, G}$ denotes
the $\sigma$-algebra of $G$-invariant Borel subsets of $X$ and $x\mapsto \mu_x$ is the Borel map from $X$ to $\cM^\erg (X,G)$ described at the beginning of the section, it suffices to show that,
for every nonnegative simple $\cB_{X, G}$-measurable function $g$ on $X$
with $g(x)\le h_{\mu_x}(\cQ)$ for every $x\in X$, one has $h_{\Sigma, \mu ,\infty}(\rho)\ge \int_Xg\, d\mu$.
Let $g$ be such a function.

%Since $h_{\Sigma ,\mu} (\cP ) = \sup_{\varepsilon > 0} \ch_{\Sigma ,\mu}^\varepsilon (\cP )$
%according to the discussion in the introduction (see Proposition~5.4 of \cite{KerLi10}),
It is enough to show that for every $\theta>0$ there is an $\varepsilon>0$ such that
$h^{\varepsilon}_{\Sigma, \mu ,\infty}(\rho)\ge \int_Xg \,d\mu-2\theta$.
So let $\theta > 0$.
%By Lemma~\ref{L-infinity norm} it suffices to show that there is an $\varepsilon>0$ such that
%$\limsup_{i\to\infty} d_i^{-1} \log N_{\varepsilon} (\Map_\mu^X (\rho ,F, V,\delta ,\sigma_i ),\rho_{\infty} )
%\ge \int_Xg\, d\mu-\theta$
%for all nonempty finite sets $F\subseteq G$, finite sets $V\subseteq C(X)$, and $\delta>0$.

Let $0<\eta<\theta/8$. Also let $\kappa>0$, which we will determine in a moment.
%Consider the continuous pseudometric $\rho$ on $X$ defined by
%$\rho(x, y)=\max_{f\in \cP}|f(x)-f(y)|$.
For every
%nonempty finite subset $K$ of $G$,
$\varepsilon > 0$ and $Q\in \cQ$ we write
$Q^{\varepsilon}$ for the open $\varepsilon$-neighbourhood of $Q$ with respect to $\rho$.
%which as usual is given by $\rho_K(x, y)=\max_{s\in K}\rho(sx, sy)$.
%Since $\cP$ dynamically generates $C(X)$, $\rho$ dynamically generates
%the topology of $X$ in the sense of \cite[Sect.\ 4]{Li10}.
Thus $\bigcap_{\varepsilon > 0} Q^{2\varepsilon} = \overline{Q}$ for each $Q\in \cQ$
where $\overline{Q}$ denotes the closure of $Q$.
As $\max_{Q\in \cQ}\mu(\partial Q)=0$, we can find
a particular $\varepsilon$,
which we now fix, such that $\sum_{Q\in \cQ} \mu (Q^{2\varepsilon} \setminus Q ) < \kappa^2$.
Set $D = X\setminus \bigcup_{Q\in \cQ} (Q^{2\varepsilon} \setminus Q)$. Now it suffices to show
$h^{\varepsilon}_{\Sigma, \mu ,\infty}(\rho)\ge \int_Xg \,d\mu-2\theta$.

Let $F$ be a nonempty finite subset of $G$, $L$ a finite subset of $C(X)$, and $\delta>0$.
Let $\sigma$ be a map from $G$ to $\Sym (d)$ for some $d\in\Nb$. It suffices to show that if $\sigma$
is a good enough sofic approximation then
\begin{align}\label{E-infty lower}
\frac{1}{d} \log N_{\varepsilon} (\Map_\mu (\rho ,F, L,\delta ,\sigma ),\rho_{\infty} )
\ge \int_Xg\, d\mu-2\theta.
\end{align}

Given a nonempty finite subset $F'$ of $G$, set
\[
D_{F'} = \bigg\{ x\in X : \sum_{s\in F'} 1_D (sx) \geq (1-\kappa )|F' | \bigg\} .
\]
Then
\begin{align*}
\kappa |F'| \mu (X\setminus D_{F'}) &\leq \int_{X\setminus D_{F'}} \sum_{s\in F'} 1_{X\setminus D} (sx)\, d\mu\\
&\leq \int_X \sum_{s\in F'} 1_{X\setminus D} (sx)\, d\mu = \mu(X\setminus D)|F' | < \kappa^2 |F' |,
\end{align*}
so that $\mu (D_{F'}) >1-\kappa$.
By Lemma~\ref{L-SM},
when $F'$ is sufficiently left invariant there is
a Borel subset $V_{F'}$ of $X$ with $\mu(V_{F'})\ge 1-\kappa$ such that
if $x\in V_{F'}$ and the atom of $\cQ^{F'}$ containing $x$ is $A$ then
\[
\mu (A) \leq \exp(-(h_{\mu_x}(\cQ)-\eta)|F'|)\leq \exp(-(g(x)-\eta)|F'|).
\]

%Let $F$ be a nonempty finite subset of $G$, $m\in \Nb$, and $\delta>0$.
%Take $n\in \Nb$ with $1/(2n)<\delta$.
%For each $0\le k<n$, set $B_k=\{x\in X: \mu_x(f)\}$.
Since $g$ and the functions $x\mapsto \mu_x(f)$ on $X$ for $f\in L$ are all $\cB_{X, G}$-measurable, we can take a finite $\cB_{X, G}$-measurable partition $\tilde{\cB}$ of $X$ such that $g$ is a constant function on each atom of $\tilde{\cB}$ and for every $B\in \tilde{\cB}$ one has
\[
\max_{f\in L}\Big(\sup_{x\in B}\mu_x(f)-\inf_{y\in B}\mu_y(f)\Big)<\frac{\delta}{8} .
\]
Denote by $\cB$ the set of atoms of $\tilde{\cB}$ with positive $\mu$-measure.
Set $\tau=\min_{B\in \cB}\mu(B)$.
The mean ergodic theorem \cite[page 44]{Mou85} states that, as the nonempty finite set $F'\subseteq G$
becomes more and more left invariant,
$|F'|^{-1}\sum_{s\in F'}\alpha_s(f)$ converges to $\Eb_{\mu}(f|\cB_{X, G})$ in $L^2(X, \cB_X, \mu)$ for every $f\in L^2(X, \cB_X, \mu)$,
where $\cB_X$ denotes the $\sigma$-algebra of Borel subsets of $X$.
%where $\Eb_{\mu}(g|\cB_{X, G})$ denotes the conditional expectation of $g$ on $L^2(X, \cB_{X, G}, \mu)$. For every bounded Borel function
%$g$ on $X$, we have $\Eb_{\mu}(g|\cB_{X, G})(x)=\mu_x(g)$ for $\mu$-a.e. $x$.
Thus,  when $F'$ is sufficiently right invariant, in other words when $(F')^{-1}$ is sufficiently left invariant, we
can find a Borel subset $W_{F'}$ of $X$ such that $\mu(W_{F'})>1-\min(\kappa, \tau/2)$
and $\big||F'|^{-1}\sum_{s\in F'}f(sx)-\mu_x(f)\big|<\delta/8$ for all
$f\in L$ and $x\in W_{F'}$. In particular,
when $F'$ is sufficiently right invariant we can find, for each $B\in \cB$,
a point $x_{B, F'}\in B$ such that $\big||F'|^{-1}\sum_{s\in F'}f(sx_{B, F'})-\mu_{x_{B, F'}}(f)\big|<\delta/8$
for all $f\in L$.

Now consider a nonempty finite subset $F'$ of $G$ which is sufficiently two-sided invariant so that both
$V_{F'}$ and $W_{F'}$ exist. For each $B\in \cB$, write $\cA_{B, F'}$ for the collection
of all $A\in \cQ^{F'}$ such that $\mu(A\cap B\cap D_{F'}\cap V_{F'}\cap W_{F'})>0$.
For each $A\in\cA_{B, F'}$ pick a point $x_A \in A\cap B\cap D_{F'}\cap V_{F'}\cap W_{F'}$ and set $E'_{B, F'} = \{ x_A : A\in\cA_{B, F'} \}$.
Denote by $g_B$ the constant value of $g$ on $B$.

{\bf Claim I:} Assuming $\kappa$ is small enough there is,
for each $B\in \cB$, a $(\rho_{F'}, \varepsilon)$-separated subset $E_{B, F'}$ of $E'_{B, F'}$ such that
\begin{align} \label{E-lower Claim I}
|E_{B, F'}|\ge \mu(B\cap D_{F'}\cap V_{F'}\cap W_{F'})\exp(\max(g_B-2\eta ,0)|F'|).
\end{align}

To verify Claim I, note first that
\[
\mu\Big(\bigcup\cA_{B, F'}\Big)\ge \mu(B\cap D_{F'}\cap V_{F'}\cap W_{F'}),
\]
and since $\mu(A)\le \exp(-(g(x_A)-\eta)|F'|)=\exp(-(g_B-\eta)|F'|)$ for every $A\in \cA_{B, F'}$, one has
\begin{align*}
|\cA_{B, F'}|&\ge \mu\Big(\bigcup\cA_{B, F'}\Big)/\exp(-(g_B-\eta)|F'|)\\
&\ge \mu(B\cap D_{F'}\cap V_{F'}\cap W_{F'})\exp((g_B-\eta)|F'|).
\end{align*}

For each $x\in E'_{B, F'}$, since $x\in D_{F'}$ there exists a $J_x\subseteq F'$ with
$|J_x|=|F'|-\lfloor \kappa |F'|\rfloor$ such that $sx\in D$ for every $s\in J_x$, where $\lfloor t\rfloor$ denotes the largest integer no bigger than $t$.
Then there exists an $E''_{B, F'}\subseteq E'_{B, F'}$ with $\binom{|F'|}{\kappa |F'|}|E''_{B, F'}|\ge |E'_{B, F'}|$ such
that $J_x$ is the same, say $J_{B, F'}$,  for all $x\in E''_{B, F'}$.

%Set $J_{B, F'}=\{s\in J'_{B, F'}: Ks\subseteq F'\}$.
%When $F'$ is sufficiently left invariant, one has $|\{s\in F': Ks\subseteq F'\}|\ge (1-\kappa)|F'|$ and hence
%\begin{align*}
% |J_{B, F'}|\ge |J'_{B, F}|-\kappa |F'|\ge (1-2\kappa)|F'|.
%\end{align*}

Let $x\in E''_{B, F'}$. Let $y\in E''_{B, F'}$ be such that $\rho_{F'}(x, y)\le \varepsilon$.
Then $sx$ and $sy$ lie in the same atom of $\cQ$ for each $s\in J_{B, F'}$, for if
$sx$ and $sy$ were contained in different atoms of $\cQ$ for some $s\in J_{B, F'}$ then since $sx, sy\in D$ we
would have $\rho(sx, sy)\ge 2\varepsilon$,
%and
%hence $\rho(tsx, tsy)\ge 2\varepsilon$ for some $t\in K$
which is impossible since $s\in F'$.
It follows that there are at most
$|\cQ|^{|F'|-|J_{B, F'}|}$ many $y\in E''_{B, F'}$ satisfying $\rho_{F'}(x, y)\le \varepsilon$. Hence there must exist a
$(\rho_{F'}, \varepsilon)$-separated subset
$E_{B, F'}$ of $E''_{B, F'}$ such that $|\cQ|^{|F'|-|J_{B, F'}|}|E_{B, F'}|\ge |E''_{B, F'}|$. We then have
\begin{align*}
|E_{B, F'}|&\ge |\cQ|^{|J_{B, F'}|-|F'|}|E''_{B, F'}|
\ge |\cQ|^{-\kappa |F'|}|E'_{B, F'}| \binom{|F'|}{\kappa |F'|}^{-1}\\
&=|\cQ|^{-\kappa |F'|}|\cA_{B, F'}|\binom{|F'|}{\kappa |F'|}^{-1}\\
&\ge \mu(B\cap D_{F'}\cap V_{F'}\cap W_{F'})\exp((g_B-\eta)|F'|)|\cQ|^{-\kappa |F'|}\binom{|F'|}{\kappa |F'|}^{-1} .
\end{align*}
Stirling's approximation then implies that when $\kappa$ is small enough we have
%\begin{align*}
%|E_{B, F'}|&\ge \mu(B\cap D_{F'}\cap V_{F'}\cap W_{F'})\exp((g_B-3\eta)|F'|)
%\end{align*}
the inequality \eqref{E-lower Claim I} for all sufficiently right invariant $F'$. This proves Claim I.

Let $\delta'>0$ be such that $\delta'<\tau$, $4\delta'|\cB|\max_{f\in L} \| f \|_\infty <\delta$ and
$\delta'\sum_{B\in \cB}g_B<\theta/4$.
Let $M$ be a large positive integer to be specified below.

Let $\delta'' > 0$, which we will determine in a moment.
It follows from Lemma~\ref{L-Rokhlin2} that there are an $\ell\in \Nb$
and
sufficiently two-sided invariant nonempty finite subsets
$F_1, \dots, F_\ell$ of $G$ such that for every map $\sigma : G\to\Sym(d)$ for some $d\in \Nb$ which is a good
enough sofic approximation for $G$ there exist  $C_1, \dots, C_\ell\subseteq \{1, \dots, d\}$ such that
\begin{enumerate}
\item for every $k=1, \dots, \ell$, the map $(s, c)\mapsto \sigma_s(c)$ from $F_k\times C_k$ to $\sigma(F_k)C_k$ is bijective,

\item the family $\{ \sigma(F_1)C_1, \dots, \sigma(F_\ell)C_\ell \}$ is disjoint and $(1-\delta'')$-covers $\{1, \dots, d\}$.
\end{enumerate}
%Fix $F_1, \dots, F_\ell$.
Write $\Lambda$ for the set of all $k\in\{1, \dots, \ell\}$ such that $|C_k|\ge M$. Taking $M$ to be large enough,
for every $k\in \Lambda$ we can find a partition $\{C_{k, B}\}_{B\in \cB}$ of $C_k$ such that
$||C_{k, B}|/|C_k|-\mu(B)|<\delta'$ for every $B\in \cB$.

For each $k\in \Lambda$, set $\cB'_k=\{B\in \cB: \mu(B\cap D_{F_k}\cap V_{F_k}\cap W_{F_k})\ge \tau/2\}$.
If $B\in \cB\setminus \cB'_k$, then $\mu(B\setminus (D_{F_k}\cap V_{F_k}\cap W_{F_k}))>\tau/2>\mu(B\cap D_{F_k}\cap V_{F_k}\cap W_{F_k})$,
and hence $\mu(B)<2\mu(B\setminus (D_{F_k}\cap V_{F_k}\cap W_{F_k}))$.
Since $\mu( X\setminus (D_{F_k}\cap V_{F_k}\cap W_{F_k}))<3\kappa$, we have
\[
\mu\Big(\bigcup (\cB\setminus \cB'_k)\Big)\le 2\mu(X\setminus (D_{F_k}\cap V_{F_k}\cap W_{F_k}))<6\kappa.
\]
Taking $\kappa$ to be small enough, we may require that $\int_Yg\, d\mu<\theta/4$ for every Borel
set $Y\subseteq X$ with $\mu(Y)<6\kappa$. Then
\begin{align} \label{E-lower}
 \int_{\bigcup (\cB\setminus \cB'_k)}g\, d\mu<\theta/4.
\end{align}
%Thus
%\begin{eqnarray*}
%\sum_{B\in \cB'_k}|C_{k, B}|/|C_k|\ge \sum_{B\in \cB'_k}\mu(B)-|\cB|\delta''\ge 1-2\kappa-4\delta'- %|\cB|\delta''.
%\end{eqnarray*}

For each $h = (h_{k,B} )_{k,B} \in \prod_{k\in \Lambda}\prod_{B\in \cB'_k}(E_{B, F_k})^{C_{k, B}}$, take
a map $\varphi_h$ from $\{1, \dots, d\}$ to $X$ such that
for every $k\in \Lambda$, $B\in \cB$, $c\in C_{k, B}$, and $s\in F_k$ the point
$\varphi_h(s c)$  is equal to $s(h_{k,B} (c))$ or $s(x_{B, F_k})$
depending on whether $B\in \cB'_k$ or $B\in \cB\setminus \cB'_k$.
%and
%for every $1\le i\le \ell$, $B\in \cB\setminus \cB'$,  $c\in C_{B, i}$, and $s\in F_i$, the unital homomorphism
%$f\mapsto \varphi_h(f)(\sigma_s(c))$ on $C(X)$ is given by evaluation at $s(h(c))$,

{\bf Claim II}:
Assuming $F_1, \dots, F_\ell$ are sufficiently left invariant, $\delta''$ is small enough, and
$\sigma$ is a good enough sofic approximation for $G$, one has $|\bigcup_{k\in \Lambda}\sigma(F_k)C_k|>(1-2\delta'')d$ and the map $\varphi_h$ lies in
$\Map_\mu(\rho, F, L, \delta, \sigma)$ for every $h\in \prod_{k\in \Lambda}\prod_{B\in \cB'_k}(E_{B, F_k})^{C_{k, B}}$.

To verify Claim II, suppose we are given $k\in \Lambda$, $B\in \cB'_k$ and $f\in L$.
Since $B\in \cB_{X, G}$, we have
\begin{align*}
\int_Bf\, d\mu=\int_B\Eb_{\mu}(f|\cB_{X, G})\, d\mu=\int_B\mu_x(f)\, d\mu(x).
\end{align*}
For each $c\in C_{k, B}$, since $h_{k,B} (c)\in E_{B, F_k}\subseteq W_{F_k}\cap B$ one has
\begin{align*}
\lefteqn{\bigg|\frac{1}{|F_k|}\sum_{s\in F_k}f(\varphi_h(s c))-\frac{1}{\mu(B)}\int_Bf\, d\mu\bigg| } \hspace*{10mm}\\
\hspace*{10mm} &\le \bigg|\frac{1}{|F_k|}\sum_{s\in F_k}f(s(h_{k,B}(c)))-\mu_{h_{k,B}(c)}(f)\bigg|
+\bigg|\mu_{h_{k,B}(c)}(f)-\frac{1}{\mu(B)}\int_B\mu_x(f)\, d\mu(x)\bigg|\\
&<\frac{\delta}{8}+\frac{\delta}{8}=\frac{\delta}{4}.
\end{align*}
As $|\sigma(F_k)C_{k, B}|^{-1}\sum_{a\in \sigma(F_k)C_{k, B}}f(\varphi_h(a))$ is a convex combination of the quantities
$|F_k|^{-1}\sum_{s\in F_k}f(\varphi_h(s c))$ for $c\in C_{k, B}$, we get
\begin{align} \label{E-lower bound}
\bigg|\frac{1}{|\sigma(F_k)C_{k, B}|}\sum_{a\in \sigma(F_k)C_{k, B}}f(\varphi_h(a))-\frac{1}{\mu(B)}\int_Bf\, d\mu\bigg|
<\delta/4.
\end{align}
Inequality~\eqref{E-lower bound} also holds similarly for all
$k\in \Lambda$, $B\in \cB\setminus \cB'_k$, and $f\in L$.
Thus, for all $k\in \Lambda$ and $f\in L$ we have
\begin{align*}
\lefteqn{\bigg|\frac{1}{|\sigma(F_k)C_k|}\sum_{a\in \sigma(F_k)C_k}f(\varphi_h(a))-\int_Xf\, d\mu\bigg|} \hspace*{10mm}\\
\hspace*{10mm} &\le \bigg|\sum_{B\in \cB}\frac{|\sigma(F_k)C_{k, B}|}{|\sigma(F_k)C_k|}\cdot
\frac{1}{|\sigma(F_k)C_{k, B}|}\sum_{a\in \sigma(F_k)C_{k, B}}f(\varphi_h(a)) \\
&\hspace*{50mm} \ - \sum_{B\in \cB}\frac{|\sigma(F_k)C_{k, B}|}{|\sigma(F_k)C_k|}\cdot \frac{1}{\mu(B)}\int_Bf\, d\mu\bigg|\\
&\hspace*{10mm} \ + \bigg|\sum_{B\in \cB}\frac{|\sigma(F_k)C_{k, B}|}{|\sigma(F_k)C_k|}\cdot
\frac{1}{\mu(B)}\int_Bf\, d\mu-\sum_{B\in \cB}\mu(B)\cdot \frac{1}{\mu(B)}\int_Bf\, d\mu\bigg|\\
&< \frac{\delta}{4}+\delta'|\cB|\max_{f\in L} \| f \|_\infty \le \frac{\delta}{4}+\frac{\delta}{4}=\frac{\delta}{2}.
\end{align*}
For all $f\in L$, as $|\bigcup_{k\in \Lambda}\sigma(F_k)C_k|^{-1}\sum_{a\in \bigcup_{k\in \Lambda}\sigma(F_k)C_k}f(\varphi_h(a))$
is a convex combination of the quantities $|\sigma(F_k)C_k|^{-1}\sum_{a\in \sigma(F_k)C_k}f(\varphi_h(a))$ for $k\in \Lambda$, we get
\[
\bigg|\frac{1}{|\bigcup_{k\in \Lambda}\sigma(F_k)C_k|}\sum_{a\in \bigcup_{k\in \Lambda}\sigma(F_k)C_k}f(\varphi_h(a))-\int_Xf\, d\mu\bigg|<\frac{\delta}{2}.
\]
Note that if $\sigma$ is a good enough sofic approximation for $G$ then $d$ will be large enough so that
%$\sigma$ is a good enough sofic approximation of $G$,
the family $\{\sigma(F_k)C_k: k\in \Lambda\}$ is a
$(1-2\delta'')$-covering of $\{1, \dots, d\}$.
It follows that,
when $F_1, \dots, F_\ell$ are sufficiently left invariant, $\delta''$ is small enough,
and $\sigma$ is a good enough sofic approximation for $G$, the map $\varphi_h$ lies in
$\Map_\mu(\rho, F, L, \delta, \sigma)$.
This proves Claim II.

{\bf Claim III:}
\begin{align} \label{E-lower Claim III}
\lefteqn{\frac{1}{d}\log N_{\varepsilon}(\Map_\mu(\rho, F, L, \delta, \sigma), \rho_{\infty})}\hspace*{30mm}\\
\nonumber \hspace*{30mm} &\ge \frac{1}{\min_{1\le k\le \ell}|F_k|}\log \frac{\tau}{2}+(1-2\delta'')\Big(\int_Xg\, d\mu-\theta\Big).
\end{align}

To verify Claim III, note first that
if $h=(h_{k, B})_{k, B}$ and $h'=(h'_{k, B})_{k, B}$ are distinct elements of $\prod_{k\in \Lambda}\prod_{B\in \cB'_k}(E_{B, F_k})^{C_{k, B}}$,
then $h_{k, B}(c)\neq h'_{k, B}(c)$ for some $k\in \Lambda$, $B\in \cB'_k$, and $c\in C_{k, B}$. Since $h_{k, B}(c)$ and $h'_{k, B}(c)$ are
$(\rho_{F_k}, \varepsilon)$-separated, we see that $\rho_{\infty}(\varphi_h, \varphi_{h'})\ge \varepsilon$.
Therefore
\begin{align*}
\lefteqn{N_{\varepsilon}(\Map_\mu(\rho, F, L, \delta, \sigma), \rho_{\infty})}\hspace*{10mm}\\
\hspace*{10mm} &\ge \bigg|\prod_{k\in \Lambda}\prod_{B\in \cB'_k}(E_{B, F_k})^{C_{k, B}}\bigg|\\
&\overset{\eqref{E-lower Claim I}}\ge \prod_{k\in \Lambda}\prod_{B\in \cB'_k}(\mu(B\cap D_{F_k}\cap V_{F_k}\cap W_{F_k}))^{|C_{k, B}|}
\exp(\max(g_B-2\eta ,0)|F_k|\cdot |C_{k, B}|)\\
&\ge \prod_{k\in \Lambda}\prod_{B\in \cB'_k}\Big(\frac{\tau}{2}\Big)^{|C_{k, B}|}\exp(\max(g_B-2\eta ,0)|F_k|\cdot |C_{k, B}|)\\
&\ge \prod_{k\in \Lambda}\Big(\frac{\tau}{2}\Big)^{\sum_{B\in \cB'_k}|C_{k, B}|}\exp\bigg(|F_k|\cdot |C_k|
\sum_{B\in \cB'_k}\max(g_B-2\eta ,0)(\mu(B)-\delta')\bigg)\\
&\ge \prod_{k\in \Lambda}\Big(\frac{\tau}{2}\Big)^{|C_k|}\exp\bigg(|F_k|\cdot |C_k|\sum_{B\in \cB'_k}(g_B-2\eta)(\mu(B)-\delta')\bigg)\\
&\ge \Big(\Big(\frac{\tau}{2}\Big)^{1/\min_{1\le k\le \ell}|F_k|}\Big)^d
\prod_{k\in \Lambda}\exp\bigg(|F_k|\cdot |C_k|\sum_{B\in \cB'_k}(g_B-2\eta)(\mu(B)-\delta')\bigg).
\end{align*}
%Taking $\kappa$ to be small enough, we may require that $\int_Yg\, d\mu<\theta/4$ for every Borel
%set $Y\subseteq X$ with $\mu(Y)<2\kappa+4\delta'$.
%Taking $\delta'$ to be small enough, we may also require
%that both $\delta'\sum_{B\in \cB}g_B$ is less than $\theta/4$.
For $k\in \Lambda$, one has
\begin{align*}
\sum_{B\in \cB'_k}(g_B-2\eta)(\mu(B)-\delta')&\ge \sum_{B\in \cB'_k}g_B\mu(B)- \delta'\sum_{B\in\cB'_k}g_B- 2\eta\sum_{B\in \cB'_k}(\mu(B)-\delta')\\
&\ge \int_{\bigcup \cB'_k}g\, d\mu-\delta'\sum_{B\in \cB}g_B-2\eta\\
%&\overset{\eqref{E-lower}}> \Big(\int_Xg\, d\mu-\theta/4\Big)-\theta/4-\theta/4>\int_Xg\, d\mu-\theta.
&> \Big(\int_Xg\, d\mu-\theta/4\Big)-\theta/4-\theta/4>\int_Xg\, d\mu-\theta,
\end{align*}
where \eqref{E-lower} has been used to obtain the second last inequality.
Thus
\begin{align*}
\lefteqn{N_{\varepsilon}(\Map_\mu(\rho, F, L, \delta, \sigma), \rho_{\infty})}\hspace*{20mm}\\
\hspace*{20mm}&\ge \Big(\Big(\frac{\tau}{2}\Big)^{1/\min_{1\le k\le \ell}|F_k|}\Big)^d\prod_{k\in \Lambda}
\exp\Big(|F_k|\cdot |C_k|\Big(\int_Xg\, d\mu-\theta\Big)\Big)\\
&\ge \Big(\Big(\frac{\tau}{2}\Big)^{1/\min_{1\le k\le \ell}|F_k|}\Big)^d\exp\Big(d(1-2\delta'')\Big(\int_Xg\, d\mu-\theta\Big)\Big),
\end{align*}
and hence the inequality \eqref{E-lower Claim III} holds. This proves Claim III.
%\begin{align*}
%\lefteqn{\liminf_{i\to \infty}\frac{1}{d_i}\log N_{\varepsilon}(\Hom^X_{\mu}(\cP, F, m, \delta, \sigma_i), \rho_{\cP ,\infty})}\hspace*{30mm}\\
%\hspace*{30mm} &\ge \frac{1}{\min_{1\le k\le \ell}|F_k|}\log \frac{\tau}{2}+(1-2\delta'')\Big(\int_Xg\, d\mu-\theta\Big).
%\end{align*}

Now by taking $\delta''$ to be small enough and $F_1, \dots, F_\ell$ to be sufficiently left invariant, we obtain
\begin{align*}
\frac{1}{\min_{1\le k\le \ell}|F_k|} \log \frac{\tau}{2} + (1-2\delta'' ) \bigg(\int_X g\, d\mu - \theta \bigg)
\ge \int_X g\, d\mu - 2\theta ,
\end{align*}
yielding \eqref{E-infty lower}.
%By taking $F_1, \dots, F_\ell$ to be sufficiently left invariant we can render the quantity
%$\min_{1\le k\le \ell}|F_k|$ as large as we wish. Letting $\delta''\to 0$,
%we get
%\[
%\liminf_{i\to \infty}\frac{1}{d_i}\log N_{\varepsilon}(\Map_\mu(\rho, F, L, \delta, \sigma_i), \rho_{\infty})\ge \int_Xg\, d\mu-\theta,
%\]
%as desired.
%Now by Lemma~\ref{L-infinity norm} there are an $\varepsilon' > 0$ and a $\gamma > 0$
%not depending on $F$, $m$, or $\delta$ with $\gamma\to 0$ as $\varepsilon\to 0$
%such that the left side of the above inequality is bounded above by
%$\bar{h}^{\varepsilon'}_{\Sigma, \mu}(\cP, F, m, \delta) + \gamma$.
%By assuming $\varepsilon$ to be small enough, we will then obtain
%\[
%\bar{h}^{\varepsilon'}_{\Sigma, \mu}(\cP, F, m, \delta) \geq \int_Xg\, d\mu-2\theta
%\]
%for all nonempty finite sets $F\subseteq G$, $m\in\Nb$, and $\delta > 0$.
%%Since $F$ was an arbitrary finite subset of $G$, $m$ an arbitrary positive integer, and $\delta$ an arbitrary positive number,
%We thus conclude that $\bar{h}^{\varepsilon'}_{\Sigma, \mu}(\cP)\ge \int_Xg(x)\, d\mu(x)-2\theta$, as desired.
\end{proof}

When a finite group $G$ acts on a standard probability space $(X,\mu )$
by measure-preserving transformations, from the definition of $h_\mu(X,G)$
%the Kolmogorov-Sinai measure entropy
given at the beginning of this section we have
$h_\mu(X,G) =+\infty$ when $\sum_{x\in X}\mu(\{x\})<1$, and $h_\mu(X,G) =|G|^{-1}\sum_{x\in X}\xi(\mu(\{x\}))$ when $\sum_{x\in X}\mu(\{x\})=1$, where $\xi(t)=-t\log t$ for $t\in [0, 1]$.

\begin{lemma} \label{L-finite upper bound}
Let $G$ be a finite group acting on a standard probability space $(X,\mu )$
by measure-preserving transformations.
%Let $\Sigma$ be a sofic approximation sequence for $G$.
Then
\[
h_{\Sigma ,\mu} (X,G) \le h_\mu (X,G) .
\]
\end{lemma}

\begin{proof}
Set $\xi(t)=-t\log t$ for $t\in [0, 1]$. We may assume that $h_{\mu}(X, G)<+\infty$. Then there is a $G$-invariant countable subset $Z$ of $X$ with $\sum_{z\in Z}\mu(\{z\})=1$, $\mu(\{z\})>0$ for every $z\in Z$, and
$h_\mu(X, G)=|G|^{-1}\sum_{z\in Z}\xi(\mu(\{z\}))$. Let $\kappa>0$. It suffices to show
that $h_{\Sigma, \mu}(X, G)\le h_\mu(X, G)+3\kappa$.

Up to measure conjugacy, we may assume that $X$ is a compact metrizable space such that each point of $Z$ is isolated in $X$, $G$ acts on $X$ continuously, and $\mu$ is a $G$-invariant Borel probability measure on $X$. Let $\rho$ be a compatible metric on $X$ with $\diam_{\rho}(X)\le 1$.
By Proposition~\ref{P-measure entropy} it suffices to show that
$h_{\Sigma, \mu ,2}^{\varepsilon}(\rho)\le h_{\mu}(X, G)+3\kappa$ for every $\varepsilon>0$.

Let $\varepsilon>0$.
Take a finite subset $Z'$ of $Z$ such that the orbits $Gz$ for $z\in Z'$ are pairwise disjoint,  $1-\mu(GZ')<\varepsilon^2/2$, and $\xi(1-\mu(GZ'))<\kappa$.
Say $Z'=\{z_1, \dots, z_n\}$.
For each $k=1,\dots ,n$ write $p_k$ for the characteristic function of $\{z_k\}$ in $C(X)$,
%Also take $p_{n+1}$ to be a function on $X$ vanishing on
%$(X\setminus Z)\cup GZ'$ and separating the points of $Z\setminus GZ'$ such that $0<p_{n+1}(z)\le 1$ for every $z\in Z\setminus GZ'$ and the only limit point of $p_{n+1}(Z\setminus GZ')$, if one exists, is $0$.
%Then the set $\cP=\{p_1, \dots, p_n, p_{n+1}, 1-\sum_{k=1}^{n+1}p_k\}$ is a dynamically generating finite
%partition of unity in $L^{\infty}(X, \mu)$.
%Replacing $X$ by the Gelfand spectrum of the unital $G$-invariant sub-$C^*$-algebra of
%$L^{\infty}(X, \mu)$ generated by $\cP$, we may assume that
%$G$ acts on $X$ continuously, $\mu$ is a $G$-invariant Borel probability measure on $X$, and $\cP$ is a dynamically generating partition of unity in $C(X)$.
%Since $h_{\Sigma, \mu}(X, G)=h_{\Sigma, \mu}(\cP)=\sup_{\varepsilon>0} \bar{h}_{\Sigma, \mu}^{\varepsilon}(\cP)$,
%Set $V=\{p_1, \dots, p_n\}$.
%For each $k=1, \dots, n$ we
and write $c_k$ and $G_k$ for $\mu(Gz_k)$ and $\{s\in G: sz_k=z_k\}$, respectively.
Let $\tau$ be a strictly positive number
%satisfying $1-\mu(GZ')+n\tau<\varepsilon^2/8$
to be specified in a moment.
Let $\delta>0$ be such that the $\rho$-distance of each point in $GZ'$ from any other point in $X$ is bigger than $\sqrt{\delta}$,
$(c_k+\delta|G/G_k|)/(1-\delta)<c_k+\tau$ for every $k=1, \dots, n$,
$(2+|G|)\delta|G|<\tau$, and $n|G|(2+|G|)\delta<\varepsilon^2/2$.
%and $1-(1-\delta)(\mu(GZ')-n\tau)<\varepsilon^2/4$.
Set $L=\{p_1, \dots, p_n\}$. Now it suffices to show that
$h^{\varepsilon}_{\Sigma, \mu ,2}(\rho, G, L, \delta)\le h_{\mu}(X, G)+3\kappa$.

By Lemma~\ref{L-Rokhlin2}, when a map $\sigma$ from $G$ to $\Sym(d)$ for some $d\in \Nb$ is a good enough sofic approximation for $G$, there exists a subset $C$ of $\{1, \dots, d\}$ such
that the map $(s, c)\mapsto \sigma_s(c)$ from $G\times C$ to $\sigma(G)C$ is bijective, $|\sigma(G)C|/d\ge 1-\delta$, and $\sigma_e(c)=c$ and
$\sigma_s\sigma_t(c)=\sigma_{st}(c)$ for all $c\in C$ and $s, t\in G$.

Let $\varphi\in \Map_\mu(\rho, G, L, \delta, \sigma)$. Let $1\le k\le n$ and $s\in G$.
%Since $(\varphi(\alpha_s(p_k)))^2=\varphi(\alpha_s(p_k)^2)=\varphi(\alpha_s(p_k))$,
%the function $\varphi(\alpha_s(p_k))$ is the characteristic function of some subset of $\{1, \dots, d\}$. Denote this subset by
Set $Y_{k, s, \varphi}=\varphi^{-1}(sz_k)$. When $s=e$, we write $Y_{k, \varphi}$ for $Y_{k, e, \varphi}$.
If $a\in Y_{k, \varphi}$ and $sa\not\in Y_{k, s, \varphi}$, then $\rho(\varphi(sa), s\varphi(a))>\sqrt{\delta}$.
It follows that
$$\frac{1}{d}|(sY_{k, \varphi})\setminus Y_{k, s,\varphi}|\delta\le (\rho_2(\alpha_s\circ \varphi, \varphi\circ \sigma_s))^2<\delta^2,$$
and hence $|(sY_{k, \varphi})\setminus Y_{k, s,\varphi}|/d<\delta$.
%Note that
%\begin{align*}
%\frac1d |\sigma_s(Y_{k, \varphi})\Delta Y_{k,s,\varphi}|
%&=\|\sigma_s(\varphi(\alpha_e(p_k)))-\varphi(\alpha_s(p_k))\|_2^2 \\
%&=\|\sigma_s(\varphi(p_k))-\varphi(\alpha_s(p_k))\|_2^2 \\
%&<\delta^2<\delta.
%\end{align*}
Set $Y'_{k, \varphi}=Y_{k, \varphi}\cap \sigma(G)C\cap \bigcap_{s\in G}\sigma_s^{-1}(Y_{k, s, \varphi})$. Then
$|Y_{k, \varphi}\setminus Y'_{k, \varphi}|/d\le (1+|G|)\delta\le \tau$.
%Thus, given $(Y'_{1, \varphi}, \dots, Y'_{n, \varphi})$, the number of possibilities for
%$(Y_{1, \varphi}, \dots, Y_{n, \varphi})$ is at most
%\[
%M_1:=\bigg[\tau d\binom{d}{\tau d}\bigg]^n.
%\]
%By Stirling's approximation, when $\tau$ is small enough one has $M_1\le \exp(\kappa d)$
%for all $d\in \Nb$.

For each $k=1, \dots, n$ set $C_{k, \varphi}=\{c\in C: \sigma(G)c\subseteq \sigma(G)Y'_{k, \varphi}\}$.
Then $\sigma(G)C_{k, \varphi}=\sigma(G)Y_{k', \varphi}$.
Note that if $s\in G$ and $t\in sG_k$, then $tz_k=sz_k$ and hence $Y_{k, t, \varphi}=Y_{k, s, \varphi}$. If
$t\in G\setminus sG_k$, then $tz_k\neq sz_k$ and hence $Y_{k,t, \varphi}\cap Y_{k, s, \varphi}=\emptyset$.
Let $a\in Y'_{k, \varphi}$. Then $\sigma(sG_k)a\subseteq \bigcup_{t\in sG_k}Y_{k, t, \varphi}=Y_{k, s, \varphi}$ for every $s\in G$.
It follows that $Y_{k, s, \varphi}\cap \sigma(G)a=\sigma(sG_k)a$ for every $s\in G$. In particular, for each $c\in C_{k, \varphi}$,
the set $Y'_{k, \varphi}\cap \sigma(G)c$ is of the form $\sigma(G_ks)c$ for some $s\in G$. Thus,
given $(C_{1, \varphi}, \dots, C_{n, \varphi})$, the number of possibilities for $(Y'_{1, \varphi}, \dots, Y'_{n, \varphi})$ is at most
\[
M_1:=\prod_{k=1}^n|G/G_k|^{|C_{k, \varphi}|}.
\]

Let $k=1, \dots, n$. We have
\[
\big| |Y_{k,\varphi}|/d-c_k|G/G_k|^{-1} \big| = |(\varphi_*\zeta)(p_k)-\mu(p_k)|<\delta,
\]
and hence
\[
|Y'_{k, \varphi}|/d\le |Y_{k, \varphi}|/d<c_k|G/G_k|^{-1}+\delta
\]
and
\begin{align} \label{E-finite upper bound}
|Y'_{k, \varphi}|/d\ge |Y_{k, \varphi}|/d-(1+|G|)\delta\ge c_k|G/G_k|^{-1}-(2+|G|)\delta.
\end{align}
Since $|Y'_{k, \varphi}|=|C_{k, \varphi}|\cdot |G_k|$, we get
%\[c_k+\delta>(|C_{k, \varphi}|\cdot |G_k|)/d>c_k-\tau.\]
%\[\big|\big|\bigcup_{s\in G}Y_{k, s, \varphi}\big|/d-c_k|G/G_k|\big|<\delta|G/G_k|.\]
%Since $(\bigcup_{s\in G}Y_{k, s,\varphi})\cap \sigma(G)(C\setminus C_{k, \varphi})\subseteq \bigcup_{s\in G}(Y_{k, s, \varphi}\setminus \sigma_s(Y_{k, %\varphi}))$,
%\[|(\bigcup_{s\in G}Y_{k, s,\varphi})\cap \sigma(G)(C\setminus C_{k, \varphi})|\le |\bigcup_{s\in G}(Y_{k, s, \varphi}\setminus \sigma_s(Y_{k, %\varphi}))|<|G|^2\delta d.\]
%Therefore
\begin{align*}
|C_{k, \varphi}|/|C|=|G/G_k|\cdot |Y'_{k, \varphi}|/|\sigma(G)C|< (c_k+\delta|G/G_k|)/(1-\delta)<c_k+\tau,
\end{align*}
and
\[
|C_{k, \varphi}|/|C|\ge |G/G_k|\cdot |Y'_{k, \varphi}|/d>c_k-(2+|G|)\delta|G|>c_k-\tau.
\]
%when $\delta$ is small enough.
Thus
\begin{align*}
M_1\le \prod_{k=1}^n|G/G_k|^{(c_k+\tau)|C|}\le \exp\bigg(\bigg(\kappa+\sum_{k=1}^nc_k\log |G/G_k|\bigg)\frac{d}{|G|}\bigg)
\end{align*}
granted that $\tau$ is small enough.

For each $k=1, \dots, n$, one has $\varphi(C_{k, \varphi})\subseteq \varphi(\sigma(G)Y'_{k, \varphi})\subseteq Gz_k$. Thus the sets $C_{1, \varphi}, \dots, C_{n, \varphi}$ are pairwise disjoint.
%For $1\le k<k'\le n$, one has
%\[
%\varphi\bigg(\sum_{s\in G}\alpha_s(p_k)\bigg)\varphi\bigg(\sum_{s\in G}\alpha_s(p_{k'})\bigg)=\varphi\bigg(\bigg(\sum_{s\in G}\alpha_s(p_k)\bigg)\bigg(\sum_{s\in G}\alpha_s(p_{k'})\bigg)\bigg)=0 ,
%\]
%and hence $C_{k, \varphi}\cap C_{k', \varphi}=\emptyset$.
Therefore the number of possibilities for the collection $(C_{1, \varphi}, \dots, C_{n, \varphi})$ is at most
\begin{align*}
M_2:=\sum_{j_1, \dots, j_n}\binom{|C|}{j_1}\binom{|C|-j_1}{j_2}\dots \binom{|C|-\sum_{k=1}^{n-1}j_k}{j_n},
\end{align*}
where the sum ranges over all nonnegative integers $j_1, \dots, j_n$ such that
$|j_k/|C|-c_k|<\tau$ for all $1\le k\le n$ and $\sum_{k=1}^nj_k\le |C|$.
By Stirling's approximation, for such $j_1, \dots, j_n$ one has
\begin{align*}
\lefteqn{\binom{|C|}{j_1}\binom{|C|-j_1}{j_2}\cdots \binom{|C|-\sum_{k=1}^{n-1}j_k}{j_n}} \hspace*{30mm}\\
\hspace*{30mm} &\le b \exp\bigg(\bigg(\sum_{k=1}^n\xi(j_k/|C|)+\xi\bigg(1-\sum_{k=1}^nj_k/|C|\bigg)\bigg)|C|\bigg)
\end{align*}
for some $b>0$ independent of $|C|$ and $j_1, \dots, j_n$.
Since the function $\xi$ is continuous,
when $\tau$ is small enough one has
\begin{align*}
\sum_{k=1}^n\xi(t_k)+\xi\bigg(1-\sum_{k=1}^nt_k\bigg)&<\sum_{k=1}^n\xi(c_k)+\xi\bigg(1-\sum_{k=1}^nc_k\bigg)+\kappa\\
&= \sum_{k=1}^n\xi(c_k)+\xi\bigg(1-\mu(GZ')\bigg)+\kappa\\
&<\sum_{k=1}^n\xi(c_k)+2\kappa
\end{align*}
whenever $t_k\ge 0$,
$|t_k-c_k|<\tau$ for all $k=1,\dots ,n$ and $\sum_{k=1}^nt_k\le 1$. Therefore
\begin{align*}
M_2&\le b(2\tau d)^n\exp\bigg(\bigg(\sum_{k=1}^n\xi(c_k)
+2\kappa\bigg)\frac{d}{|G|}\bigg).
%&\le a_2(2\tau d)^n\exp(\kappa d/2).
\end{align*}

Let $D$ be a $(\rho_2, \varepsilon)$-separated subset of $\Map_\mu(\rho, G, L, \delta, \sigma)$ of maximal cardinality.
Then there is a $(\rho_2, \varepsilon)$-separated subset $W$ of $D$ with $M_1M_2|W|\ge |D|$ such that the collection $(Y'_{1, \varphi}, \dots, Y'_{n, \varphi})$ is the same, say $(Y'_1, \dots, Y'_n)$, for every $\varphi \in W$.
Note that the elements of $W$ are all equal on $\bigcup_{1\le k\le n}\sigma(G)Y'_k$. By our choice of $\delta$,
\begin{align*}
\frac{1}{d}\bigg|\bigcup_{1\le k\le n}\sigma(G)Y'_k \bigg|&=\frac{1}{d}\sum_{1\le k\le n}|G/G_k|\cdot |Y'_k|
\overset{(\ref{E-finite upper bound})}\ge \sum_{1\le k\le n}(c_k-(2+|G|)|G/G_k|\delta )\\
&\ge  \sum_{1\le k\le n}(c_k-|G|(2+|G|)\delta )=\mu(GZ')-n|G|(2+|G|)\delta > 1-\varepsilon^2.
\end{align*}
%Let $\varphi, \psi\in W$. Then $\varphi(p_k)=\psi(p_k)$ for all $k=1, \dots, n$. Since $\varphi(\alpha_s(p_k))\varphi(p_{n+1})=\varphi(\alpha_s(p_k)p_{n+1})=0$ for every $k=1, \dots, n$ and $s\in G$,
%the function $\varphi(p_{n+1})$ vanishes on $\sigma(G)\big(\bigcup_{k=1}^nC_k \big)$.
%As $0\le \varphi(p_{n+1})\le 1$, we get
%\begin{align*}
%\|\varphi(p_{n+1})\|^2_2&\le 1-\frac1d \bigg|\sigma(G)\bigcup_{k=1}^nC_k\bigg| \le 1-(1-\delta)\sum_{k=1}^n \frac{|C_k|}{|C|}
%\le 1-(1-\delta)\sum_{k=1}^n(c_k-\tau)\\
%&=1-(1-\delta)(\mu(GZ')-n\tau)<(\varepsilon/2)^2,
%\end{align*}
%and hence $\|\varphi(p_{n+1})\|_2<\varepsilon/2$. Similarly, $\|\psi(p_{n+1})\|_2<\varepsilon/2$.
It follows that any two elements of $W$ have $\rho_2$-distance less than $\varepsilon$.
%$\rho_\cP(\varphi, \psi)<\varepsilon$, and hence $\varphi=\psi$.
Thus $|W|\le 1$. Therefore
\begin{align*}
N_{\varepsilon}(\Map_\mu(\rho, G, L, \delta, \sigma), \rho_2)&=|D|\le M_1M_2\\
&\le b\exp\bigg(\bigg(\kappa+\sum_{k=1}^nc_k\log |G/G_k|\bigg) \frac{d}{|G|}\bigg)\\
&\hspace*{25mm}\ \times (2\tau d)^n\exp\bigg(\bigg(\sum_{k=1}^n\xi(c_k)
+2\kappa\bigg)\frac{d}{|G|}\bigg)\\
&\le b(2\tau d)^n \exp\bigg(\bigg(3\kappa+|G|^{-1}\sum_{z\in GZ'}\xi(\mu(\{z\}))\bigg)d\bigg)\\
&\le b(2\tau d)^n\exp((3\kappa+h_{\mu}(X, G))d).
\end{align*}
It follows that
\[
h^{\varepsilon}_{\Sigma, \mu ,2}(\rho, G, L, \delta)\le 3\kappa+h_{\mu}(X, G).
\]
\end{proof}

\begin{lemma} \label{L-finite lower bound}
Let $G$ be a finite group acting on a standard probability space $(X,\mu )$
by measure-preserving transformations.
%Let $\Sigma$ be a sofic approximation sequence for $G$.
Then
\[
h_{\Sigma ,\mu} (X,G) \ge h_\mu (X,G) .
\]
\end{lemma}

\begin{proof}
List the subgroups of $G$ as $H_1, \dots, H_{\ell}$. For $x\in X$ we write $G_x$ for $\{g\in G: gx=x\}$.

Since $G$ is finite, we can find a measurable subset $Y$ of $X$ such that $|Y\cap Gx|=1$ for every $x\in X$
\cite[Ex.\ 6.1 and Prop.\ 6.4]{KM}. For every $k=1,\dots ,\ell$ set $Y_k=\{x\in Y: G_x=H_k\}$.
We may add or remove a measure zero subset of $Y_k$ without changing either $h_{\Sigma, \mu}(X, G)$ or $h_\mu(X, G)$.
Then we may identify $Y_k$ with a closed subset of the interval $[2k, 2k+1]$
in such a way that there exist $2k\le t_k\le 2k+1$
and $2k+1\ge a_{k, 1}>a_{k, 2}>\dots\ge  t_k$ with $Y_k=[2k, t_k]\cup\{a_{k, 1}, a_{k, 2}, \dots\}$, $\mu(E)$ is the Lebesgue measure of $E$ for every Borel $E\subseteq [2k, t_k]$, and $\mu(a_{k, n})>0$ for every $n$ \cite[Thm.\ 17.41]{Kechris}.
Here we allow the set $\{a_{k,1}, a_{k, 2}, \dots\}$ to be finite or even empty. Reordering $H_1, \dots, H_\ell$ if necessary, we may assume that there is some $0\le \ell'\le \ell$ such that
$t_k>2k$ for all $1\le k\le \ell'$ and $t_k=2k$ for all $\ell'<k\le \ell$.

Now we may identify $X$ with the disjoint union $\bigsqcup_{k=1}^\ell Y_k\times (G/H_k)$ in a natural way. Equip $\bigsqcup_{k=1}^\ell Y_k\times (G/H_k)$
with its natural topology coming from the product topology of $Y_k\times (G/H_k)$. Then $G$ acts continuously
on the compact metrizable space $X$. To simply the notation, we shall also identify $Y_k$ with $Y_k \times \{eH_k\}\subseteq X$, and hence think of $Y$ as a subset of $X$.
Let $\rho$ be a compatible metric on $X$ such that $\rho(x, y)=|x-y|$ for all $x, y\in Y$.
%Define a real-valued function $p$ on $X$ by $p=0$ on $X\setminus Y$ and $p(x)=x/(2\ell+1)$ for $x\in Y$.
%Then $\cP:=\{p, 1-p\}$ is a dynamically generating partition of unity in $C(X)$.
By Proposition~\ref{P-measure entropy} one has  $h_{\Sigma, \mu}(X, G)=h_{\Sigma, \mu ,\infty}(\rho)$.
%=\sup_{\varepsilon>0}\bar{h}_{\Sigma, \mu}^{\varepsilon}(\rho)$.

We consider first the case $t_k>2k$ for some $1\le k\le \ell$, i.e. $\ell'\neq 0$. In this case we will
show that $h_{\Sigma, \mu ,\infty}(\rho)=+\infty$.
Let $N\in \Nb$. Take an $\varepsilon > 0$ such that $\varepsilon <(t_1-2)/N$. Let $L$ be a finite subset of
$C(X)$ and let $\delta>0$.

Let $\eta$ be a strictly positive number satisfying $\eta<\min_{1\le k\le \ell'}\mu(G[2k, t_k])/2$ to
be further specified in a moment.
By Lemma~\ref{L-Rokhlin2}, when a map $\sigma$ from $G$ to $\Sym(d)$ for some $d\in \Nb$ is a
good enough sofic approximation for $G$, there exists a subset $C$ of $\{1, \dots, d\}$ such
that the map $(s, c)\mapsto \sigma_s(c)$ from $G\times C$ to $\sigma(G)C$ is bijective,
$|\sigma(G)C|/d\ge 1-\eta$, and $\sigma_e(c)=c$ and
$\sigma_s\sigma_t(c)=\sigma_{st}(c)$ for all $c\in C$ and $s, t\in G$.

For each $k=1,\dots ,\ell$ take an $n_k\in \Nb\cup \{0\}$ such that the points $a_{k, 1}, \dots, a_{k, n_k}$ are defined and $\sum_{k=1}^\ell \sum_{j>n_k}\mu(Ga_{k, j})<\eta$. Denote by $\Lambda$ the set of $(k, j)$ such that either
$1\le k\le \ell'$ and $j=0$ or $1\le k\le \ell$ and $1\le j\le n_k$.
Note that $|C|\to +\infty$ as $d\to +\infty$. Thus when $d$ is large enough we can find a partition
$\{C_{k, j}\}_{(k, j)\in \Lambda}$ of $C$
with $\sum_{k=1}^{\ell'}||C_{k, 0}|/|C|-\mu(G[2k, t_k])|<\eta$ and $\sum_{k=1}^\ell\sum_{j=1}^{n_k}||C_{k, j}|/|C|-\mu(Ga_{k, j})|<\eta$.
Set $x_{k, j}=2k+j(t_k-2k)/|C_{k, 0}|$ for $1\le k\le \ell'$ and $1\le j\le |C_{k, 0}|$.
For each $h=(h_k)_{k=1}^{\ell'}$ consisting of a bijection from $C_{k, 0}$ to
$\{x_{k, j}: 1\le j\le |C_{k, 0}|\}$ for each $1\le k\le \ell'$,
we take a map $\varphi_h: \{1, \dots, d\}\rightarrow X$ sending $sc$ to $s(h_k(c))$
for $1\le k\le \ell'$, $c\in C_{k, 0}$, and $s\in G$,
and sending $sc$ to $sa_{k, j}$ for $1\le k\le \ell$, $c\in C_{k, j}$, and $s\in G$.
%Then we have  the  unital homomorphism $\varphi_h$ from $C(X)$ to $\Cb^d$ sending $f$ to $f \circ \tilde{h}$.
It is readily checked that, when $\eta$ is small enough and $d$ is large enough, every such $\varphi_h$ belongs to
$\Map_\mu(\rho, G, L, \delta, \sigma)$.

Note that $\rho(x_{1, j}, x_{1, j'})>\varepsilon$ for any $1\le j, j'\le |C_{1,0}|$ with $|j-j'|\ge |C_{1, 0}|/N$.
When $d$ is large enough, we may require that $N$  divides $|C_{1, 0}|$.
Denote by $\Gamma$ the set of permutations of $\{x_{1, j}: 1\le j\le |C_{1, 0}|\}$ preserving the subset $\{x_{1, j+k|C_{1, 0}|/N}: 0\le k<N\}$ for each $1\le j\le |C_{1, 0}|/N$. Fix one $h$ as above. For each $\gamma \in \Gamma$, set $h_{\gamma}=(\gamma \circ h_1, h_2, \dots, h_{\ell'})$.
Then the set $\{\varphi_{h_{\gamma}}: \gamma \in \Gamma\}$ is $(\rho_\infty, \varepsilon)$-separated.
%, where $\rho_{\cP, \infty}$ is defined before Lemma~\ref{L-infinity norm}.
Therefore
\[
N_{\varepsilon}(\Map_\mu(\rho, G, L, \delta, \sigma), \rho_\infty)\ge |\Gamma|=(N!)^{|C_{1, 0}|/N}.
\]
It follows that
\begin{align*}
\liminf_{i\to\infty} \frac{1}{d_i} \log N_{\varepsilon} (\Map_\mu (\rho, G, L, \delta, \sigma_i ),\rho_\infty )
&\ge \frac{(\mu(G[2, t_1])-\eta)(1-\eta)}{|G|N}\log (N!)\\
&\ge \frac{\mu(G[2, t_1])(1-\mu(G[2, t_1])/2)}{2|G|N}\log (N!).
\end{align*}
Since $N$ can be taken to be arbitrarily large, we conclude that
$h_{\Sigma, \mu ,\infty}(\rho)=\sup_{\varepsilon>0} h_{\Sigma, \mu ,\infty}^{\varepsilon}(\rho)=+\infty$.

In the case that $X$ is atomic, i.e. $\ell'=0$, one can see how to proceed by rewinding through the proof
of Lemma~\ref{L-finite upper bound}. We will simply outline the argument and leave the details to the reader.
The goal is to construct sufficiently many
approximately equivariant maps from a given sofic approximation space into $X$ in order to get the desired lower
bound for the sofic measure entropy.
Such a map is constructed as follows. Fixing a partition of $X$ into orbits, if the action is free
then we can pair off each base point from a finite collection of orbits
with sets of base points of the decomposition of a fixed sofic approximation as given by Lemma~\ref{L-Rokhlin},
subject to the requirement that the measures approximately match up.
If the action is not free then the components of the sofic approximation decomposition must be
further partitioned as necessary by means of cosets in order to enable the pairing off with base points in $X$
having nontrivial isotropy subgroup. The choices involved in this pairing procedure are controlled, up to some error,
by the product of the quantities $M_1$ and $M_2$ as in the proof of Lemma~\ref{L-finite upper bound}.
From this we obtain the desired lower bound.
%The proof for the atomic case is to reverse the argument in the proof of Lemma~\ref{L-finite upper bound}.
\end{proof}

\begin{theorem} \label{T-measure sofic equal classical}
Let $G$ be an amenable countable discrete group acting on a standard probability space $(X,\mu )$
by measure-preserving transformations.
Let $\Sigma$ be a sofic approximation sequence for $G$. Then
\[
h_{\Sigma ,\mu} (X,G) = h_\mu (X,G) .
\]
\end{theorem}

\begin{proof}
By Lemmas~\ref{L-finite upper bound} and \ref{L-finite lower bound}, we may assume that $G$ is infinite.
Since $(X,\mu )$ is standard, up to measure conjugacy
we may assume that $X$ is a compact metrizable space on which $G$ acts
continuously and $\mu$ is a $G$-invariant Borel probability measure on $X$.
Then $h_{\Sigma ,\mu} (X,G) \geq h_\mu (X,G)$ by Lemma~\ref{L-lower bound} and Proposition~\ref{P-measure entropy},
while the reverse inequality follows from
Lemma~\ref{L-supremum upper bound} and
%the fact that
%$h_{\Sigma ,\mu} (X,G) = \sup_{\varepsilon > 0} \ch_{\Sigma ,\mu}^\varepsilon (\rho )$
%by
Proposition~\ref{P-measure entropy}.
%according to the discussion in the introduction (see ???Proposition~5.4 of \cite{KerLi10}).
\end{proof}

We remark that, when the action is ergodic, Theorem~\ref{T-measure sofic equal classical} follows from
Theorem~\ref{T-amenable}, the Jewett-Krieger theorem for actions of amenable groups \cite{Rosenthal}, and the variational
principle \cite[Thm.\ 5.2.7]{Mou85}\cite[Thm.\ 6.1]{KerLi10}. 
Note however that much of the complication of the proof of Theorem~\ref{T-measure sofic equal classical} 
is due to the fact that we are not assuming the action to be ergodic.

%By Proposition~\ref{P-subgroup top} we obtain the following corollary.
%
%\begin{corollary}
%Let $G$ be an amenable countable discrete group acting on a standard probability space $(X,\mu )$
%by measure-preserving transformations.
%Let $H$ be a subgroup of $G$. Then the restriction of the action to $H$ satisfies $h_\mu (X,H) \geq h_\mu (X,G)$.
%\end{corollary}

\end{document}